\documentclass[10pt,twoside, a4paper, english, reqno]{amsart}
\usepackage[dvips]{epsfig}
\usepackage{amscd}
\usepackage{stmaryrd}
\usepackage{amssymb}
\usepackage{amsthm}
\usepackage{amsmath}
\usepackage{graphicx,enumitem,dsfont}
\usepackage{latexsym}
\usepackage{appendix}

\usepackage{upref}
\usepackage[colorlinks]{hyperref}
\usepackage{color}
\usepackage{subcaption}
\usepackage[usenames,dvipsnames]{xcolor}
\theoremstyle{plain}
\newtheorem{thm}{Theorem}[section]
\theoremstyle{plain}
\newtheorem{lem}[thm]{Lemma}
\newtheorem{prop}[thm]{Proposition}

\theoremstyle{definition}
\newtheorem{defi}{Definition}[section]
\newtheorem{rem}{Remark}[section]

\newtheorem*{maintheorem*}{Main Theorem}

\usepackage{graphicx}
\usepackage{epsfig}
\newenvironment{Assumptions}
{
\setcounter{enumi}{0}

\begin{enumerate}}
{\end{enumerate} }

\newcommand{\R}{\ensuremath{\mathbb{R}}}

\newcommand{\goto}{\ensuremath{\rightarrow}}

\newcommand{\eps}{\ensuremath{\varepsilon}}

\numberwithin{equation}{section} \allowdisplaybreaks

\title[Dirichlet problem for stochastic degenerate parabolic-hyperbolic equation]{Homogeneous Dirichlet problem for degenerate parabolic-hyperbolic PDE driven by L\'evy noise}
\date{}

\author[S. R. Behera]{Soumya Ranjan Behera}
\address[Soumya Ranjan Behera] {\newline 
Department of Mathematics,
Indian Institute of Technology Delhi,
Hauz Khas, New Delhi, 110016, India.}
\email[] {maz198759@iitd.ac.in}

\author[A. K. Majee]{Ananta K. Majee}
\address[Ananta K. Majee]{\newline
Department of Mathematics,
Indian Institute of Technology Delhi,
Hauz Khas, New Delhi, 110016, India. }
\email[]{majee@maths.iitd.ac.in}


\thanks{}
\begin{document}
\begin{abstract}
In this article, we study the homogeneous Dirichlet problem for a degenerate parabolic-hyperbolic PDE perturbed by L\'evy noise. In particular, we develop the well-posedness theory of entropy solution based on the Kru\v{z}kov's semi-entropy formulation. In comparison to the pioneered work by Bauzet et al. \cite[J. Funct. Anal. 266, (2014), 2503-2545]{Bauzet-2014}, concerning the existence and uniqueness of entropy solution for the Dirichlet problem for conservation laws driven by Brownian noise, our present analysis involves a simpler approach to obtain the global Kato's inequality.
\end{abstract}
\maketitle
\section{Introduction}
Let $(\Omega, \mathbb{P}, \mathcal{F}, \{\mathcal{F}_t\}_{t \ge 0})$ be a filtered probability space satisfying the usual hypothesis.
We consider the following non-linear degenerate stochastic partial differential equation:
\begin{equation}\label{eq:degenerate-SPDE}
    \begin{cases}
        \displaystyle du - \text{div}(F(u))\,dt - \Delta \Phi(u)\,dt  = \phi(u)\,dW(t) + \int_{E} \nu(u;z)\widetilde{N}(dz,dt), \quad \text{in}  \quad \Pi, \\
        u(0,x) = u_0 \quad \text{in}\quad  D \quad \text{and} \quad 
        u = 0 \quad \text{on} \quad \Sigma\,,
    \end{cases}
\end{equation}
where $D \subset \R^d$ is a bounded domain with Lipschitz boundary $\partial D$ if $d \ge 2,$ $T > 0$ is fixed, $\Pi = (0,T) \times D$, and $\Sigma = (0,T) \times \partial D$. In \eqref{eq:degenerate-SPDE}, $F: \R \rightarrow \R^d$ is a given flux function, $W(t)$ is a $\{\mathcal{F}_t\}_{t \ge 0}$-adapted one-dimensional continuous Brownian noise, and $\widetilde{N}(dz,dt) = N(dz,dt) - m(dz)dt$, where $N$ is a Poisson random measure on a $\sigma$-finite measure space $(E, \mathcal{E},m)$. Moreover, the noise coefficients $\phi(u)$ and $\nu(u;z)$ are real-valued functions defined on the domain $\R$ and $\R \times E$, respectively. Furthermore, $\Phi: \R \rightarrow \R$ is a given non-decreasing Lipschitz continuous function. The stochastic integral in the right-hand side of \eqref{eq:degenerate-SPDE} is defined in the It\^o-L\'evy sense. 
\subsection{Review of existing literature} Equations of type \eqref{eq:degenerate-SPDE} describe the convection-diffusion of ideal fluids, making them essential in various applications. For instance, \eqref{eq:degenerate-SPDE} is used to model the phenomenon of two or three-phase flows in porous media \cite{Karlsen-2000-LN} or sedimentation-consolidation processes \cite{Burger-1999}. Incorporating stochastic noise into the above physical model is entirely natural as it accounts for ambiguities surrounding particular physical quantities or external disturbances. In the absence of noise terms, \eqref{eq:degenerate-SPDE} becomes a deterministic degenerate parabolic–hyperbolic equation in a bounded domain. A number of authors have contributed since then, and we mention few of the works e.g, \cite{Otto-1996,Ammar-2006, Porretta-2003, Carrillo-1999, Mascia-2002,Sbihi-2015,Frid-2017} and references therein. The notion of entropy solution for homogeneous Dirichlet problem for deterministic conservation laws was first introduced by Otto in \cite{Otto-1996}, where he studied the problem in $L^\infty$-framework. The well-posedness analysis of entropy solution for non-homogeneous boundary condition was established by Carrillo et al. \cite{Ammar-2006}, whereas  Porretta and Vovelle \cite{Porretta-2003} delved into the problem in $L^1$-setting. The existence and uniqueness of entropy solution for non-linear degenerate PDE with homogeneous Dirichlet boundary condition was addressed by Carrillo \cite{Carrillo-1999} while the problem with non-homogeneous boundary condition was investigated in \cite{Mascia-2002}. We also refer to the works of Vallet in \cite{Vallet-2000} and \cite{Vallet-2005}, where the existence and uniqueness of the 
 measure-valued solution was established in the bounded domain. The well-posedness analysis for entropy solution of scalar conservation laws and anisotropic degenerate parabolic-hyperbolic PDE under mixed Dirichlet-Neumann boundary condition were studied in \cite{Sbihi-2015} and \cite{Frid-2017}, respectively. For the well-posedness analysis in the unbounded domain, we mention the works of 
\cite{dafermos, kruzkov, Volpert, Lions-1994} for scalar conservation laws and \cite{Maliki-2010, Chen-2005, 
Chen-2003, Hudjaev-1969} and their associated references for degenerate parabolic-hyperbolic PDEs.

\vspace{0.1cm}
 In the last decade, the study of conservation laws with stochastic forcing has attracted the attention of many researchers. In fact, Kim \cite{Kim} extended the Kru\v{z}kov well-posedness theory to one-dimensional balance laws driven by additive Brownian noise. However, one cannot directly apply the straightforward Kru\v{z}kov doubling variables technique to get uniqueness in case of multiplicative noise. Feng and Naurlart \cite{nualart:2008} addressed this issue by introducing the notion of strong stochastic entropy solution. They established the uniqueness of entropy solutions by recovering the additional information from the vanishing viscosity method, and existence was obtained by using  stochastic version of the compensated compactness method in one spatial dimension; see also Biswas et al. \cite{Majee-2014}. In \cite{Bauzet-2012}, Bauzet et al. introduced the notion of generalized entropy solution and established the well-posedness theory of entropy solution for multi-dimensional balance laws driven by Brownian noise. We refer to Chen et al. \cite{Chen:2012fk}, where the well-posedness of entropy solution was obtained in $L^p \cap BV$, via BV-framework. In \cite{Karlsen-2017}, the authors proved the well-posedness of entropy solution for stochastic hyperbolic conservation laws using Malliavin calculus---which simplifies some of the proofs of \cite{nualart:2008, Bauzet-2012}. Debusssche and Vovelle \cite{Vovelle2010} investigated the balance laws driven by Brownian noise in a multidimensional torus using the kinetic formulation framework; see also \cite{Vovelle-2018}. As a generalization of \cite{Bauzet-2012}, Bauzet et al. in \cite{Bauzet-2015} established the existence and uniqueness of entropy solution of \eqref{eq:degenerate-SPDE} in unbounded domain with $\nu = 0$ while Biswas et al. in \cite{Majee-2019} studied the problem with $\phi = 0$. Debusssche et al. \cite{Hofmanova-2016} 
established the well-posedness theory of kinetic solution for quasilinear degenerate parabolic PDEs driven by the Wiener process in periodic boundary condition. Very recently, Behera et al. \cite{SRB-2023} have established the well-posedness theory for renormalized entropy solution of degenerate parabolic-hyperbolic PDE perturbed by a multiplicative L\'evy noise with general $L^1$-data on $\R^d$.

\vspace{0.1cm}
The study of stochastic conservation laws in a bounded domain is sparser than those in the whole space. Bauzet et al. \cite{Bauzet-2014} established the existence and uniqueness of entropy solution for \eqref{eq:degenerate-SPDE} with $\Phi =0 =\nu$; see also \cite{Vallet-2009} for additive noise case. Recently, Kobayasi and Noboriguchi \cite{Kobayasi-2016} studied the stochastic balance laws with non-homogeneous Dirichlet boundary condition in a kinetic formulation framework. Very recently, the authors have developed the well-posedness theory for the kinetic solution of quasilinear degenerate parabolic-hyperbolic PDE driven by Brownian noise under non-homogeneous Dirichlet and mixed 
Dirichlet-Neumann boundary condition in \cite{Frid-2022} and \cite{Frid-2022a}, respectively.
\subsection{Aim and outline of the paper}
The entropy solution framework for the homogeneous Dirichlet problem for degenerate parabolic-hyperbolic PDE driven by L\'evy noise is addressed here for the first time. We introduce the definition of entropy solution for \eqref{eq:degenerate-SPDE} and prove the existence and uniqueness of entropy solution. As usual, the existence is obtained via classical viscosity arguments. Due to the degeneracy in the equation, one needs to get the point-wise convergence of the sequence of viscous solutions. To do so, we first prove the uniqueness of Young measure-valued limit processes of viscous solutions by using an appropriate version of classical Kru\v{z}kov's doubling the variables method in a bounded domain. In fact, the main arguments are based on obtaining the local and global Kato's inequalities, allowing comparison principle in $L^1$. To sum up the technicalities of the present article, we have:
\begin{itemize}
    \item [i)] Since the point-wise convergence of viscous approximations is needed, 
we compare any two weak solutions coming from artificial viscosity along with the Kru\v{z}kov's doubling the variables technique in a stochastic setting. 
Moreover, due to the lack of $H^2$-regularity of the viscous solution, one can't obtain the Kato's inequality (cf.~term $\mathcal{H}_6$ in Lemma \ref{lem:G7-H6}); we need to regularize the viscous solution. Since we are working in an arbitrary bounded domain, the regularization by the usual mollifier is not an admissible choice. We regularize the viscous solution with a modified version of space convolution; see \eqref{eq:mollification}. 
    \item [ii)] To obtain the global Kato's inequality, Bauzet et al. \cite{Bauzet-2014}  have used the following identity: 
    \begin{align*}
        ({\tt v} -{\tt u}^+)^+ + (-{\tt v} - {\tt u}^-)^+ = ({\tt v}- {\tt u})^+,
    \end{align*}
    where ${\tt u}$ and ${\tt v}$ are measure-valued entropy solution of the underlying problem. First, they estimated global Kato's inequality with respect to $({\tt v} -{\tt u}^+)$, and in the second step, global Kato's inequality for $(-{\tt v} - {\tt u}^-)$. To be more specific, in the first step, they used the following inequality to handle the terms related to the Laplacian of viscous solution.
    \begin{align}
        \Delta u j^\prime(u) \le \Delta j(u)\quad \text{in} \quad \mathcal{D}^\prime \quad \text{for} \quad j(s) = (k^+-s)^+, \label{inq:j-kato}
    \end{align}
    for any $k, s \in \R$. The presence of degenerate term $\Phi$ increases the technical difficulties if one tries to apply the method of \cite{Bauzet-2014}. However, our analysis is based on the identity
    \begin{align*}
        ({\tt v}^+ -{\tt u}^+)^+ + ({\tt u}^- - {\tt v}^-)^+ = ({\tt v}- {\tt u})^+, 
    \end{align*}
    in which we are able to handle the terms related to the Laplacian of viscous solution in a simpler manner without using \eqref{inq:j-kato} ~(see Lemma \ref{lem:G-11-H-10}). Needless to say, the change in our setup, as compared to Bauzet et al. \cite{Bauzet-2014}, forces us to revisit all the terms involved in the entropy inequalities generated from the doubling the variables method.
    \item[iii)] Like in deterministic case, we require a special function $\beta$ (cf.~Section \ref{sec:Technical Framework }) approximating the semi-Kru\v{z}kov entropy $x \mapsto x^+$. Since $\beta$ is asymmetric, to handle the terms associated with the degenerate term $\Phi$, the standard arguments, as done in the unbounded domain (see, \cite{Bauzet-2015} and \cite{Majee-2019}), will not be applicable. We use a different technique to tackle the degenerate term---which is discussed in detail in the proof of \eqref{eq:liminf-hat-G1}; see also Lemma \ref{lem:to-deal-with-degnerate-term}.
   \item[iv)] Using the partition of unity, Global Kato's inequality along with the local Kato's inequality and appropriate test function, we deduce the strong convergence of viscous approximations in $L^p_{\rm loc}(\Omega\times \Pi)$ for any $1\le p<2$---which then enable us to prove the existence of a unique entropy solution of the underlying problem. 
\end{itemize}

The remaining part of the article is structured as follows. We describe the technical framework and state the definition of entropy solution, the assumptions, and the main result in Section \ref{sec:Technical Framework }. Section \ref{sec:Well-posedness of Entropy Solution} is devoted to the proof of global Kato's inequality and the existence and uniqueness of entropy solution.

\section{Technical Framework and Statement of The Main Result}\label{sec:Technical Framework }
In this article, we use the letter $C$ to denote various generic constants. We denote by $N_\omega^2(0,T, L^2(D))$ the space of predictable $L^2(D)$-valued processes. Let $\mathcal{D}(D)$ be the space of $C^\infty(D)$-functions with compact support in $D$. Then $\mathcal{D}^+(D)$ denotes the subset of non-negative elements of $\mathcal{D}(D).$ Let $\mathcal{M}^+$ be the set of non-negative convex functions $\beta$ in $C^{2,1}(\R)$, approximating the semi-Kru\v{z}kov entropies $x \mapsto x^+$ such that $\beta(x) = 0$ if $x \le 0$ and there exists $\xi > 0$ such that $\beta'(x) = 1$ if $ x > \xi$. Note that for any $\beta\in \mathcal{M}^+$, $\beta''$ has a compact support, and $\beta$ and $\beta'$ are Lipschitz-continuous functions. A typical element $\beta_\xi$ of $\mathcal{M}^+$ is given by
\begin{equation}
    \beta_{\xi}(0) = 0 \quad \text{and} \quad \beta_\xi'(r) = 
    \begin{cases}
     1  \quad &\text{if}~~ r > \xi, \\
     \frac{1 + \sin(\frac{\pi}{2\xi}(2r - \xi))}{2} \quad &\text{if}~~ 0 \le r \le \xi, \\
     0 &\text{if}~~ r < 0.
    \end{cases} 
    \notag
\end{equation}
Set $\mathcal{M}^-:=\{\Breve{\beta}(\cdot) := \beta(-\cdot), \beta \in \mathcal{M}^+\}$ and for the definition of entropy inequality, we denote
\begin{align*}
    &\mathcal{A}^{+}:= \big\{(k, \psi, \beta) \in \R \times \mathcal{D}^+(\R^{d+1}) \times \mathcal{M}^+, k < 0 \Rightarrow \psi \in \mathcal{D}^+([0,T] \times D)\big\}\,, \\
    &\mathcal{A}^{-}:= \big\{(k, \psi, \beta):  (-k, \psi, \Breve{\beta}) \in \mathcal{A}^+\big\} \quad \text{and} \quad \mathcal{A} = \mathcal{A}^{+} \cup \mathcal{A}^{-}. 
\end{align*}
Observe that, 
\begin{align}\label{inq:for-beta-xi}
    \big|\beta_\xi(r) - r^+\big| \le  C\xi, \quad \beta^\prime_{\xi}(\cdot) \le 1, \quad \text{and} \quad \beta_{\xi}^{\prime\prime}(r) \le \frac{C}{\xi}{\tt 1}_{\{|r| \le \xi\}}.
\end{align}
Let us denote 
\begin{align*}
    &\text{sgn}^{+}(x) = 
    \begin{cases}
    1 &\text{if}~~ x > 0\\
    0 &\text{else}
    \end{cases}\,, \quad   \text{sgn}^{-}(x) = 
    \begin{cases}
    1  &\text{if}~~ x < 0 \\
    0 &\text{else}
    \end{cases}\,, \quad   \text{sgn}(x) = \text{sgn}^{+}(x) + \text{sgn}^{-}(x).\\
    & F(a,b) = \text{sgn}(a-b)\big[F(a) - F(b)\big]; \quad F^{\pm}(a,b) = \text{sgn}^{\pm}(a-b)\big[F(a) - F(b)\big]\,,\\
     & \Phi(a,b) = \text{sgn}(a-b)\big[\Phi(a) - \Phi(b)\big]; \quad \Phi^{\pm}(a,b) = \text{sgn}^{\pm}(a-b)\big[\Phi(a) - \Phi(b)\big]\,,\\
    & F^\beta(a,b) = \int_{b}^a\beta'(r-b)F'(r)\,dr\,, \quad \Phi^\beta(a,b) = \int_b^a\beta'(r-b)\Phi'(r)\,dr\, \quad \text{for any} \quad \beta \in \mathcal{A}^+ \cup \mathcal{A}^-.
\end{align*}
For technical reasons, we introduce a slightly modified version of convolution that can be used to define a valid mollification technique up to the boundary for any bounded domain in $\R^d$ satisfying the uniform cone condition. For the definition of uniform cone condition, we refer to see \cite[Definition 2.1]{Blouza-SIAM-2001}, \cite{Girault-1999} and the references therein. 
\begin{rem}
    A bounded domain of $\R^d$ has a Lipschitz boundary if and only if it satisfies the uniform cone condition; see \cite[Theorem
1.2.2.2]{Grisvard-1985}.
\end{rem}
Let $\rho$ be a standard mollifier in $\R^d$. We may assume that $D \subset \R^d$ satisfies the uniform cone condition with just one cone equal to $\mathcal{C}$. Let $\theta_\mathcal{C}, e_\mathcal{C}$, and $h_\mathcal{C}$ be the cone's angle, outward unit axis vector, and height of the cone, respectively. For $\kappa > 0$, we define
\begin{align*}
    \eta(\kappa):= \kappa\sin\Big(\frac{\theta_\mathcal{C}}{2}\Big), \quad and \quad \rho_\kappa(y):= \frac{1}{\big(\eta(\kappa)\big)^d}\rho\Big(\frac{y}{\eta(\kappa)}\Big).
\end{align*}
Let $\kappa_\mathcal{C}:= \frac{h_\mathcal{C}}{1 + \sin\big(\frac{\theta_\mathcal{C}}{2}\big)}$ and ${\tt f} \in L^p(D)$. For all $0 < \kappa < \kappa_\mathcal{C}$ and for all $x \in \Bar{D}$, we define
\begin{align}\label{eq:mollification}
    {\tt f}^\kappa(x)= \int_{B(0, \eta(\kappa))}{\tt f}(x-\kappa e -y)\rho_\kappa(y)\,dy := \widetilde{{\tt f}} \ast_{\kappa, e} \rho_\kappa,
\end{align}
where $e$ be the unit vector in $\R^d$. In \eqref{eq:mollification}, ${\tt f}^\kappa = \widetilde{{\tt f}} \ast_{\kappa,e}  \rho_\kappa\big|_{\Bar{D}}$, where $\widetilde{{\tt f}}$ is the extension of ${\tt f}$ by $0$ to $\R^n$.  By uniform cone property, $x-\kappa e - B(0, \eta(\kappa)) \subset x + \mathcal{C} \subset D.$ Thus, ${\tt f}^\kappa(\cdot)$ is well-defined up-to boundary and ${\tt f}^\kappa \in C^\infty(\Bar{D}).$ Using similar lines of argument that are used in standard mollification by approximating ${\tt f}$ by a compactly supported mollifier and performing convolution translation on ${\tt f}$, one can obtain ${\tt f}^\kappa \rightarrow {\tt f}$ a.e. in D and for $1 \le p < \infty$, ${\tt f}^\kappa \rightarrow {\tt f}$ in $L^p(D)$ as $\kappa \rightarrow 0$, see \cite[Theorem 2.4]{Blouza-SIAM-2001} and \cite{Girault-1999} for thorough proofs. From now on-wards, we denote $ \widetilde{{\tt f}} \ast_{\kappa,e}  \rho_\kappa$ by ${\tt f} \ast \rho_\kappa$ for simplicity of notation.
\subsection{Entropy inequality and well-posedness result}
We aim to present the notion of entropy inequality, the properties satisfied by such a solution, and the main result of this paper. Replacing the deterministic chain rule with the It\^o-L\'evy chain rule is essential in our setup to achieve the entropy inequality. With \eqref{eq:mollification} in hand, one can invoke a similar argument as done in \cite[Theorem A.1]{Majee-2019} under cosmetic changes to obtain the generalized version of It\^o-l\'evy formula in a bounded domain. We omit the details and directly define the stochastic entropy solution for \eqref{eq:degenerate-SPDE}, the proof of which is very similar to the explanations provided in \cite[Section 2]{Bauzet-2015} and \cite[Section 2.1]{Majee-2019} and can be directly applied in a bounded domain. 
\begin{defi}\label{def-1}(Stochastic Entropy Solution) A stochastic process $u \in N_\omega^2(0,T, L^2(D))$ is said to be a stochastic entropy solution of \eqref{eq:degenerate-SPDE} if the following conditions hold:
\begin{itemize}
 \item [i)] For each $ T>0$, $ \underset{0\leq t\leq T}\sup\, \mathbb{E}\Big[\|u(t, \cdot)\|_{L^2(D)}^{2} \Big] < \infty$ and $G(u)\in L^2(\Omega\times(0,T); H^1(D))$, where $G(\cdot)$ is the associated Kirchoff's function of the given nonlinear diffusion $\Phi$, defined by $$G(x)=\int_0^x \sqrt{\Phi^\prime(r)}\,dr.$$
 \item[ii)] For all $(k, \psi, \beta) \in \mathcal{A}$,\quad  $\Lambda(k,\psi,\beta) \ge 0$\quad$\mathbb{P}$-a.s., where
\begin{align*}
\Lambda(k,\psi,\beta) &:= \int_{\Pi} \big\{ \beta(u-k)\partial_t \psi
-F^\beta(u,k)\nabla\psi + \Phi^\beta(u,k)\Delta\psi\big\}\,dx\,dt \notag \\ 
&+\int_{\Pi}\phi(u)\beta'(u-k)\psi\,dx\,dW(t)+
   \frac{1}{2} \int_{\Pi} \phi^2(u)\beta''(u-k)\psi\,dx\,dt \notag \\
 &+ \int_{\Pi}\int_{E}\big\{\beta\big(u-k +\nu(u;z)\big)- \beta(u-k)\big\}\psi\,\widetilde{N}(dz,dt)\,dx \notag \\
 &+  \int_{\Pi}\int_{E} \big\{\beta\big(u-k +\nu(u;z)\big)- \beta(u-k) -\nu(u;z)\beta^\prime(u-k)\big\}\psi\,m(dz)\,dx\,dt \notag \\
&- \int_{\Pi}\beta^{\prime\prime}(u-k)|\nabla G(u)|^2 \psi\,dx\,dt +\int_{D}\beta(u_0-k)\psi(0,x)\,dx\,.
\end{align*}
\end{itemize}
\end{defi}
The main goal of this article is to prove the existence and uniqueness of entropy solution for \eqref{eq:degenerate-SPDE} in the sense of Definition \ref{def-1}, and we need the following assumptions to achieve it.
\begin{Assumptions}
\item \label{A1}  The initial function $u_0$ is a deterministic function satisfying $\|u_0\|_{L^2(D)} < \infty$.
\item \label{A2}$F = (F_1,F_2, \cdot\cdot\cdot,F_d): \R \rightarrow \R^d$ is a Lipschitz continuous function with $F_k(0) =0$ for all $1 \le k \le d$.
\item \label{A3} $\Phi: \R \rightarrow \R$ is a non-decreasing Lipschitz continuous function with $\Phi(0) = 0.$ 
\item  \label{A4} $\phi(0)  = 0$ and $\phi: \R \rightarrow \R$ is a Lipschitz continuous function.
\item \label{A5} The space $E$ is of the form $\mathcal{O} \times \R^*$, where $\mathcal{O}$ is a subset of Euclidean space. The Borel measure $m$ on $E$ has the form $\mu \times \lambda$, where $\mu$ is the Radon measure on $\mathcal{O}$ and $\lambda$ is one-dimensional L\'evy measure. 
\item \label{A6}$\nu(0;z) = 0$ for all $z \in E$ and $\nu : \R \times E \rightarrow \R$ is a non-decreasing function in $\R$. There exist a positive constant $\lambda^\star$ and a non-negative function $g\in L^2(E,m)$ with $0 \le g(z) \le 1$ such that
\begin{align*}
    &|\nu(u;z)-\nu(v;z)| \le  \lambda^\star|u-v|g(z), \quad \forall\, u,v \in \R,\, z \in E.  
\end{align*}
\end{Assumptions} 
\begin{rem} In case of the unbounded domain, the Lipschitzness of $\nu$ with Lipschitz constant $0<\lambda^\star <1$, was enough to control the error terms associated with jump noise---giving the uniqueness of solutions; see \cite{Majee-2019}. However, in our setup, the asymmetric property of $\beta$ and the non-local nature of the It\^o-L\'evy formula force us to assume $\nu$, in addition to Lipschitz continuity and without any restriction of Lipschitz constant $\lambda^\star$,  to be non-decreasing in $\R$; see the proof of Lemma \ref{lem:G9-H8}.
\end{rem}
We now state the main result of this paper.
\begin{thm}(Existence and Uniqueness) Let the assumptions \ref{A1}-\ref{A6} be true. Then, there exists a unique entropy solution for \eqref{eq:degenerate-SPDE} in the sense of Definition \ref{def-1}. Moreover, if $u_1, u_2$ are entropy solutions of \eqref{eq:degenerate-SPDE} corresponding to the initial data $u_{0_1},u_{0_2} \in L^2(D)$ respectively, then for all $t \in (0,T)$, 
\begin{align*}
    \mathbb{E}\Big[\int_{D} (u_1(t) - u_2(t))^+\,dx\Big] \le \mathbb{E}\Big[\int_{D} (u_{0_1}-u_{0_2})^+\,dx\Big].
\end{align*}
\end{thm}
\begin{rem} The above result still holds if we replace the one-dimensional Brownian noise  $W$ in \eqref{eq:degenerate-SPDE} with a cylindrical Wiener process defined in a separable Hilbert space $\mathcal{H}$ with the form $W(t) = \sum_{n \ge 1}\beta_n(t)e_n$, where $(e_n)_{n \ge 1}$ is a complete orthonormal basis in $\mathcal{H}$ and $(\beta_n)_{n \ge 1}$ is a sequence of independent real-valued Brownian motion.
\end{rem}
\section{Well-posedness of Entropy Solution}\label{sec:Well-posedness of Entropy Solution}
In this section, we prove the well-posedness of entropy solution for \eqref{eq:degenerate-SPDE} based on a variant of classical Kru\v{z}kov's doubling of variables technique. To do so, we begin with a technical lemma, which is very useful in dealing with the degenerate term in the subsequent analysis; see also \cite[Lemma 3.4]{Bendahmane-2005}.
\begin{lem}\label{lem:to-deal-with-degnerate-term}
For any $l \in L^\infty_{loc}(\R)$ and $\beta_\xi \in \mathcal{M}^+$, the followings hold:
\begin{itemize}
    \item [{\rm (i)}] Let $b \in \R$ be fixed. Then, for a.e. $a \in \R$, one has
    \begin{align*}
        \underset{\xi \rightarrow 0}{\lim}\int_a^b\beta_\xi^{\prime\prime}(a-\sigma)l(\sigma)\,d\sigma = -sgn^+(a-b)l(a).
    \end{align*}
    \item [{\rm (ii)}]Let $a \in \R$ be given. Then, a.e. $b \in \R$, it holds that
    \begin{align*}
         \underset{\xi \rightarrow 0}{\lim}\int_b^a\beta_\xi^{\prime\prime}(\sigma-b)l(\sigma)\,d\sigma = sgn^+(a-b)l(b).
    \end{align*}
\end{itemize}
\end{lem}
\begin{proof} If $b > a$, it is easy to observe that
 \begin{align*}
        \int_a^b\beta_\xi^{\prime\prime}(a-\sigma)l(\sigma)\,d\sigma = 0 = -sgn^+(a-b)l(a).
    \end{align*}
    Suppose $b < a$. For sufficiently small $\xi > 0$, we have
    \begin{align}
        &\int_a^b\beta_\xi^{\prime\prime}(a-\sigma)l(\sigma)\,d\sigma = -\int_{a-\xi}^a\frac{\pi}{2\xi}\cos\big(\frac{\pi}{2\xi}(2(a-\sigma)-\xi)\big)l(\sigma)\,d\sigma \notag \\
       & = -\int_{-\xi}^{\xi}\frac{\pi}{4\xi}\cos\big(\frac{\pi r}{2\xi}\big)l\big(a - \frac{\xi+r}{2}\big)\,dr = -l(a) + \int_{-\xi}^{\xi}\frac{\pi}{4\xi}\cos\big(\frac{\pi r}{2\xi}\big)\Big(l(a)-l\big(a - \frac{\xi+r}{2}\big)\Big)\,dr\,,\label{equality:sending-xi-0}
    \end{align}
    where we have used the fact that $\int_{-\xi}^{\xi}\frac{\pi}{4\xi}\cos\big(\frac{\pi r}{2\xi}\big)\,dr = 1$.
    Thanks to the Lebesgue point theorem,
    \begin{align}\label{xi-0}
        &\Big|\int_{-\xi}^{\xi}\frac{\pi}{4\xi}\cos\big(\frac{\pi r}{2\xi}\big)\Big(l(a)-l\big(a - \frac{\xi+r}{2}\big)\Big)\,dr\Big| \le \frac{C}{2\xi}\int_{-\xi}^{\xi}\big|l(a)-l\big(a - \frac{\xi+r}{2}\big)\big|\,dr \rightarrow 0, \quad \xi \rightarrow 0.
    \end{align} 
Sending $\xi \rightarrow 0$ in \eqref{equality:sending-xi-0} and using \eqref{xi-0}, we get ${\rm (i)}$. A similar lines of argument yields ${\rm (ii)}.$
\end{proof}
\subsection{ Existence of viscous solution}
To achieve our main result, a significant technical tool is to establish the existence of a weak solution for the viscous problem. For a small $\eps > 0$, consider the viscous problem
\begin{equation}\label{eq:viscous}
    \begin{aligned}
    & du_\eps - \text{div}(F(u_\eps))\,dt - \Delta \Phi(u_\eps)\,dt  = \phi(u_\eps)\,dW(t) + \int_{E} \nu(u_\eps;z)\widetilde{N}(dz,dt) + \eps\Delta u_\eps,~\text{in}~\Pi, \\
    & u_\eps(0,x) = u_0^\eps(x)~~\text{in}~~ D \quad \text{and}~~u_\eps = 0\quad \text{ on}~~\Sigma\,,
    \end{aligned}
\end{equation}   
where $u_0^\eps \in H^1(D)$ such that $u_0^\eps \rightarrow u_0$ in $L^2(D)$. Regarding the existence of weak solution to \eqref{eq:viscous}, we follow the implicit time discretization method as proposed in \cite{Vallet-2008} for \eqref{eq:viscous} with $\nu =0$ and \cite[Section 3]{Majee-2019} to arrive at the following lemma:
\begin{lem}
Let $\eps > 0$ be given and the assumptions \ref{A1}, \ref{A2}, \ref{A3}, \ref{A4}, \ref{A5} and \ref{A6} hold true. Then, there exists a weak solution $u_\eps \in N_\omega^2(0,T, H^1(D))$ to the problem \eqref{eq:viscous} with $$\partial_t\big(u_\eps - \int_0^t\phi(u)\,dW(s) - \int_0^t\int_{E}\nu(u;z)\widetilde{N}(dz,dt)\big) \in L^2(\Omega \times (0,T), H^{-1}(D)).$$ Moreover, there exists a constant $C > 0$, independent of $\eps$, such that
\begin{align}\label{inq:apriori-bound}
    \underset{0 \le t \le T}{sup}\mathbb{E}\Big[\|u_\eps(t)\|_{L^2(D)}^2\Big] + \eps \int_0^T\mathbb{E}\Big[\|\nabla u_\eps(s)\|_{L^2(D)}^2\Big]\,ds + \int_0^T\mathbb{E}\Big[\|\nabla G( u_\eps(s))\|_{L^2(D)}^2\Big]\,ds \le C.
\end{align} 
\end{lem}
\subsection{Global Kato's inequality}\label{sec:Global-Kato -ine}
To establish the uniqueness of entropy solutions for \eqref{eq:degenerate-SPDE}, the central part is to obtain the following global Kato's inequality. 
\begin{lem}\label{lem:Global-Kato} Let ${\tt u}$ and ${\tt v}$ be the Young measure-valued limit of 
viscous solutions to \eqref{eq:degenerate-SPDE} with initial data $\hat{u}_0 \in L^2(D)$ and  ${u}_0 \in L^2(D)$, respectively. Then, for any $ \psi \in \mathcal{D}^+([0,T] \times \R^d)$, it holds that
\begin{align}\label{inq:Global-Kato}
   0 \le &\,\mathbb{E}\Big[\int_{\Pi}\int_0^1\int_0^1\big({\tt u}(t, x, \alpha)-{\tt v}(t,x, \gamma)\big)^+\partial_t\psi(t,x)\,d\alpha\,d\gamma\,dx\,dt\Big]\notag \\
    -&\,\mathbb{E}\Big[\int_{\Pi}\int_0^1\int_0^1F^+({\tt u}(t, x, \alpha), {\tt v}(t,x, \gamma))\cdot\nabla\psi(t,x)\,d\alpha\,d\gamma\,dx\,dt\Big] \notag \\  
+&\,\mathbb{E}\Big[\int_{\Pi}\int_0^1\int_0^1\Phi^+({\tt u}(t, x, \alpha), {\tt v}(t,x, \gamma))\Delta\psi(t,x)\,d\alpha\,d\gamma\,dx\,dt\Big] +    \int_{D}\big(\hat{u}_0(x)- u_0(x)\big)^+\psi(0,x)\,dx\,.
\end{align}
\end{lem}
An inherent method to obtain the uniqueness of the entropy solution is to use a modified version of the traditional Kru\v{z}kov's doubling of variables technique. It is important to highlight that the primary obstacle arises from the doubling of the time variable, leading to the emergence of anticipative stochastic integrands, thereby defying interpretation within the conventional It\^o-L\'evy framework. To overcome this, we will compare two sequences $\{u_\theta(s,y)\}$ and $\{u_\eps(t,x)\}$ of weak solutions of \eqref{eq:viscous} with a special test function along with the regularized initial data $\hat{u}_0(y)$ and $u_0(x)$ respectively. The regularity of $u_\theta \in H^1(\R^d)$ is insufficient to prove the global Kato's inequality; see proof of the Lemma \ref{lem:G7-H6}. Thus, we regularize $u_\theta$ by a slightly modified version of the space mollifier $\{\rho_\kappa\}_{\kappa > 0}$ as described in \eqref{eq:mollification} and denote it by $u_\theta^\kappa$. Then $u_\theta^\kappa$ satisfies the following equation in $L^2(D)$:
\begin{equation}
     \partial_t \Big[u_\theta^\kappa - \int_0^t\phi( u_\theta) \ast \rho_\kappa\,dW(t) -\int_0^t \int_{|z| > 0}  \eta(y, u_\theta; z) \ast \rho_\kappa\,\widetilde{N}(dz,dt) \Big] = \Delta(\Phi(u_\theta) \ast \rho_\kappa) -\text{div} (F(u_\theta) \ast \rho_\kappa) + \theta\Delta u_\theta^\kappa, \notag
\end{equation}
with initial data $u_\theta^\kappa(0,x) = u_\theta(0) \ast \rho_\kappa$ and $u_\theta^\kappa = 0$ on $\Sigma$. Note that, one cannot get \eqref{inq:Global-Kato} directly by comparing $u_\theta$ and $u_\eps$. Thus, we will use the following identity
\begin{align}\label{eq:idenity-1}
    (a-b)^+ = (a^+ - b^+)^+ + (b^- - a^-)^+, \quad \forall~a,b\in \R\,.
\end{align}
As we mentioned, we need a special test function for applying the doubling the variables technique. In this regard, we choose a partition of unity subordinate to a covering of $\Bar{D}$ by balls  $\mathcal{B}_i$, for $1 \le i \le  \bar{k}$ satisfying $\mathcal{B}_0 \cap \partial D = \emptyset$ and $1 \le i \le \bar{k}$, $\mathcal{B}_i \subset \Bar{\mathcal{B}_i}$ with $\Bar{\mathcal{B}_i} \cap \partial D$ part of Lipschitz graph. Then, we define the following.
\begin{itemize}
    \item [i)] $\psi \in \mathcal{D}^+([0,T] \times \R^d)$ with $\text{supp}(\psi(t, \cdot)) \subset \mathcal{B} := \mathcal{B}_i $ for some $i \in \{1, \cdot\cdot\cdot, \bar{k}\};$
    \item [ii)] $\rho_n(t-s)$ is a sequence of mollifier in $\R$ with $\text{supp}(\rho_n) \subset [-\frac{2}{n}, 0]$;
    \item[iii)] $\rho_m(x-y)$ a shifted sequence of mollifier in $\R^d$ (see, \cite{Bauzet-2014,Girault-1999}) such that $y \mapsto \rho_m(x-y) \in \mathcal{D}(D)$ for all  $ x \in \mathcal{B} \cap D$.
\end{itemize}
Observe that for $m$ large enough, $y \mapsto \psi(t,x)\rho_m(x-y) \in \mathcal{D}(D)$, and we have
$$\int_0^T\int_D\psi(t,x)\rho_m(x-y)\rho_n(t-s)\,dy\,ds = \psi(t,x)\bar{\rho}_m(x)\bar{\rho}_n(t),$$
where
$\Bar{\rho}_m(x) = \int_D\rho_m(x-y)\,dy,\quad \Bar{\rho}_n(t) = \int_0^T\rho_n(t-s)\,ds$
are non-negative, non-decreasing sequences bounded by $1.$ 

\subsubsection{\textbf{The first half of Kato's inequality}}\label{sec:The first half of Kato's inequality}
In view of \eqref{eq:idenity-1}, we first get the global Kato's inequality \eqref{inq:Global-Kato} in terms of  $\big({\tt u}^+(t,x,\alpha) -{\tt v}^+(t,x,\gamma)\big)^+$, where ${\tt u}(t,x,\alpha)$ and ${\tt v}(t,x,\gamma)$ are the Young measure-valued limit process of $\{u_\theta\}$ and $\{u_\eps\}$ respectively. Let $J$ be the standard non-negative mollifier in $\R$  and $J_l(r) = \frac{1}{l}J_l(\frac{r}{l})$ with $\text{supp}(J_l) = [-\frac{1}{l},\frac{1}{l}]$. Applying the It\^o-L\'evy formula to $\beta_\xi(\beta_{\Tilde{\xi}}(u_\theta^\kappa(s,y)) -k)\psi\rho_m\rho_n,~0<\Tilde{\xi}<\xi$, multiplying by $J_l(u_\eps(t,x) - k)$, and integrating with respect to $t,x,k$, we have, after taking expectation together with the application of Fubini's theorem 
\begin{align}\label{inq:doublingvariable-1st}
&\mathbb{E}\Big[\int_{\Pi}\int_{D}\int_{\R}\beta_\xi\big(\beta_{\Tilde{\xi}}(u_\theta^\kappa(0,y))-k\big)\psi(t,x)\rho_n(t)\rho_m(x-y)J_l\big(u_\eps(t,x) - k\big)\,dk\,dy\,dx\,dt\Big] \notag \\
+\,&\mathbb{E}\Big[\int_{\Pi^2}\int_{\R}\beta_\xi\big(\beta_{\Tilde{\xi}}(u_\theta^\kappa(s,y))-k\big)\psi(t,x)\partial_s\rho_n(t-s)\rho_m(x-y)J_l\big(u_\eps(t,x) - k\big)\,dk\,dy\,ds\,dx\,dt\Big] \notag \\
- \,& \mathbb{E}\Big[\int_{\Pi^2}\int_{\R}\beta_\xi^{\prime}\big(\beta_{\Tilde{\xi}}(u_\theta^\kappa(s,y))-k\big)\beta_{\Tilde{\xi}}^\prime(u_\theta^\kappa(s,y))\big(F(u_\theta(s,y))\ast\rho_\kappa\big)\psi(t,x)\rho_n(t-s)\notag \\&\hspace{6cm}\times\nabla_y\rho_m(x-y)J_l\big(u_\eps(t,x) - k\big)\,dk\,dy\,ds\,dx\,dt\Big]\notag \\
- \,& \mathbb{E}\Big[\int_{\Pi^2}\int_{\R}\Big(\beta_\xi^{\prime\prime}\big(\beta_{\Tilde{\xi}}(u_\theta^\kappa(s,y))-k\big)|\beta_{\Tilde{\xi}}^\prime(u_\theta^\kappa(s,y))|^2 + \beta_\xi^{\prime}\big(\beta_{\Tilde{\xi}}(u_\theta^\kappa(s,y))-k\big)\beta_{\Tilde{\xi}}^{\prime\prime}(u_\theta^\kappa(s,y))\Big)\nabla u_\theta^\kappa(s,y)\notag \\&\hspace{3cm}\times\big(F(u_\theta(s,y))\ast\rho_\kappa\big)\psi(t,x)\rho_n(t-s)\rho_m(x-y)J_l\big(u_\eps(t,x) - k\big)\,dk\,dy\,ds\,dx\,dt\Big]\notag \\
- \,& \mathbb{E}\Big[\int_{\Pi^2}\int_{\R}\beta_\xi^{\prime}\big(\beta_{\Tilde{\xi}}(u_\theta^\kappa(s,y))-k\big)\beta_{\Tilde{\xi}}^\prime(u_\theta^\kappa(s,y))\big(\nabla\Phi(u_\theta(s,y))\ast\rho_\kappa\big)\psi(t,x)\rho_n(t-s)\notag \\&\hspace{6cm}\times\nabla_y\rho_m(x-y)J_l\big(u_\eps(t,x) - k\big)\,dk\,dy\,ds\,dx\,dt\Big]\notag\\
- \,& \mathbb{E}\Big[\int_{\Pi^2}\int_{\R} \beta_\xi^{\prime}\big(\beta_{\Tilde{\xi}}(u_\theta^\kappa(s,y))-k\big)\beta_{\Tilde{\xi}}^{\prime\prime}(u_\theta^\kappa(s,y))\nabla u_\theta^\kappa(s,y)\big(\nabla\Phi(u_\theta(s,y))\ast\rho_\kappa\big)\notag \\&\hspace{6cm}\times\psi(t,x)\rho_n(t-s)\rho_m(x-y)J_l\big(u_\eps(t,x) - k\big)\,dk\,dy\,ds\,dx\,dt\Big]\notag \\
+\, &\frac{1}{2}\,\mathbb{E}\Big[\int_{\Pi^2}\int_{\R}\Big(\beta_\xi^{\prime\prime}\big(\beta_{\Tilde{\xi}}(u_\theta^\kappa(s,y))-k\big)|\beta_{\Tilde{\xi}}^\prime(u_\theta^\kappa(s,y))|^2 + \beta_\xi^{\prime}\big(\beta_{\Tilde{\xi}}(u_\theta^\kappa(s,y))-k\big)\beta_{\Tilde{\xi}}^{\prime\prime}(u_\theta^\kappa(s,y))\Big)\notag \\&\hspace{3cm}\times\big|\phi(u_\theta(s,y))\ast\rho_\kappa\big|^2\psi(t,x)\rho_n(t-s)\rho_m(x-y)J_l\big(u_\eps(t,x) - k\big)\,dk\,dy\,ds\,dx\,dt\Big]\notag \\
+&\,\mathbb{E}\Big[\int_{\Pi^2}\int_{\R}\beta_\xi^{\prime}\big(\beta_{\Tilde{\xi}}(u_\theta^\kappa(s,y))-k\big)\beta_{\Tilde{\xi}}^\prime(u_\theta^\kappa(s,y))\big(\phi(u_\theta(s,y))\ast\rho_\kappa\big)\psi(t,x)\rho_n(t-s)\notag \\&\hspace{6cm}\times\rho_m(x-y)J_l\big(u_\eps(t,x) - k\big)\,dk\,dy\,dW(s)\,dx\,dt\Big]\notag \\
+&\,\mathbb{E}\Big[\int_{\Pi^2}\int_{\R}\int_{E}\Big(\beta_\xi\big(\beta_{\Tilde{\xi}}(u_\theta^\kappa(s,y)) + \nu( u_\theta(s,y); z)\ast \rho_\kappa -k\big) - \beta_{\xi}\big(\beta_{\Tilde{\xi}}(u_\theta^\kappa(s,y))-k\big)\notag\\  & \hspace{2cm}- (\nu( u_\theta(s,y); z)\ast \rho_\kappa)\beta_\xi^\prime\big(\beta_{\Tilde{\xi}}(u_\theta^\kappa(s,y) -k)\big)\beta_{\Tilde{\xi}}^\prime(u_\theta^\kappa(s,y))\Big)\psi(t,x)\rho_n(t-s)\rho_m(x-y)\notag\\&\hspace{6cm} \times J_l\big(u_\eps(t,x) - k\big)\,dk\,m(dz)\,dy\,ds\,dx\,dt\Big]\notag \\
+ &\,\mathbb{E}\Big[\int_{\Pi^2}\int_{\R}\int_{E}\Big(\beta_\xi\big(\beta_{\Tilde{\xi}}(u_\theta^\kappa(s,y)) + \nu( u_\theta(s,y); z)\ast \rho_\kappa -k\big) - \beta_{\xi}\big(\beta_{\Tilde{\xi}}(u_\theta^\kappa(s,y))-k\big)\Big)\notag \\ &\hspace{3cm}\times\psi(t,x)\rho_n(t-s)\rho_m(x-y)J_l\big(u_\eps(t,x) - k\big)\,dk\,\Tilde{N}(dz,ds)\,dy\,dx\,dt\Big] \notag \\
- \,& \theta\mathbb{E}\Big[\int_{\Pi^2}\int_{\R}\beta_\xi^{\prime}\big(\beta_{\Tilde{\xi}}(u_\theta^\kappa(s,y))-k\big)\beta_{\Tilde{\xi}}^\prime(u_\theta^\kappa(s,y))\nabla u_\theta^\kappa(s,y)\psi(t,x)\rho_n(t-s)\notag \\&\hspace{6cm}\times\nabla_y\rho_m(x-y)J_l\big(u_\eps(t,x) - k\big)\,dk\,dy\,ds\,dx\,dt\Big]\notag\\
- \,& \theta\mathbb{E}\Big[\int_{\Pi^2}\int_{\R} \Big(\beta_\xi^{\prime}\big(\beta_{\Tilde{\xi}}(u_\theta^\kappa(s,y))-k\big)\beta_{\Tilde{\xi}}^{\prime\prime}(u_\theta^\kappa(s,y)) + \beta_\xi^{\prime\prime}\big(\beta_{\Tilde{\xi}}(u_\theta^\kappa(s,y))-k\big)|\beta_{\Tilde{\xi}}^\prime(u_\theta^\kappa(s,y))|^2\Big)|\nabla u_\theta^\kappa(s,y)|^2\notag \\&\hspace{4cm}\times\psi(t,x)\rho_n(t-s)\rho_m(x-y)J_l\big(u_\eps(t,x) - k\big)\,dk\,dy\,ds\,dx\,dt\Big]\notag \\
\ge & \, \mathbb{E}\Big[\int_{\Pi^2}\int_{\R}\Big(\beta_\xi^{\prime\prime}\beta_{\Tilde{\xi}}(u_\theta^\kappa(s,y))-k\big)|\beta_{\Tilde{\xi}}^\prime(u_\theta^\kappa(s,y))|^2 \nabla u_\theta^\kappa(s,y)\big(\nabla\Phi(u_\theta(s,y))\ast\rho_\kappa\big)\notag \\&\hspace{2cm}\times\psi(t,x)\rho_n(t-s)\rho_m(x-y)J_l\big(u_\eps(t,x) - k\big)\,dk\,dy\,ds\,dx\,dt\Big] \notag \\
& \text{i.e.,}~~~ \sum_{i=1}^{12}\mathcal{G}_{i} \ge \mathcal{G}_0\,. 
\end{align}
On the other hand, an application of the It\^o-L\'evy formula to $\int_{D}\beta_{\xi}(k - u_\eps(t,x))\
\psi(t,x)\rho_n(t-s)\rho_m(x-y)\,dx$, multiplication by $J_l(k- \beta_{\xi}(u_\theta^\kappa(s,y)))$ and integration with respect to $k,s,y$ yields, after taking expectation
\begin{align}\label{inq:doubling-variable-2nd}
&\mathbb{E}\Big[\int_{\Pi}\int_{D}\int_{\R}\beta_\xi\big(k- u_\eps^0(x)\big)\psi(0,x)\rho_n(-s)\rho_m(x-y)J_l\big(k -\beta_{\Tilde{\xi}}(u_\theta^\kappa(s,y))\big)\,dk\,dx\,dy\,ds\Big] \notag \\   
+\,& \mathbb{E}\Big[\int_{\Pi^2}\int_{\R}\beta_\xi\big(k-u_\eps(t,x)\big)\partial_t\big(\psi(t,x)\rho_n(t-s)\big)\rho_m(x-y)J_l\big(k -\beta_{\Tilde{\xi}}(u_\theta^\kappa(s,y))\big)\,dk\,dx\,dy\,ds\,dt\Big]\notag \\
+ \,& \mathbb{E}\Big[\int_{\Pi^2}\int_{\R}\int_k^{u_\eps(t,x)}\beta_\xi^\prime(k-r)F^\prime(r)\,dr\,\nabla_x\big(\psi(t,x)\rho_m(x-y)\big)\rho_n(t-s)\notag \\ & \hspace{8cm}\times J_l\big(k -\beta_{\Tilde{\xi}}(u_\theta^\kappa(s,y))\big)\,dk\,dx\,dy\,ds\,dt\Big]\notag \\
+ \,& \mathbb{E}\Big[\int_{\Pi^2}\int_{\R}\beta_\xi^\prime\big(k-u_\eps(t,x)\big)\nabla\Phi(u_\eps(t,x))\nabla_x\big(\psi(t,x)\rho_m(x-y)\big)\rho_n(t-s)\notag \\ &\hspace{8cm} \times J_l\big(k -\beta_{\Tilde{\xi}}(u_\theta^\kappa(s,y))\big)\,dk\,dx\,dy\,ds\,dt\Big]\notag \\
+ \,& \frac{1}{2}\mathbb{E}\Big[\int_{\Pi^2}\int_{\R}\beta_\xi^{\prime\prime}(k-u_\eps(t,x))|\phi( u_\eps(t,x))|^2\psi(t,x)\rho_m(x-y)\rho_n(t-s)J_l\big(k -\beta_{\Tilde{\xi}}(u_\theta^\kappa(s,y))\big)\,dk\,dx\,dy\,ds\,dt\Big]\notag \\
-\,&\mathbb{E}\Big[\int_{\Pi^2}\int_{\R}\beta_\xi^{\prime}(k-u_\eps(t,x))\phi( u_\eps(t,x))\psi(t,x)\rho_m(x-y)\rho_n(t-s)J_l\big(k -\beta_{\Tilde{\xi}}(u_\theta^\kappa(s,y))\big)\,dk\,dx\,dy\,ds\,dW(t)\Big]\notag \\
+ &\mathbb{E}\Big[\int_{\Pi^2}\int_{E}\int_{\R}\Big(\beta_\xi\big(k-u_\eps(t,x) - \nu(u_\eps(t,x);z)\big) - \beta_\xi(k-u_\eps(t,x)) + \nu(u_\eps(t,x);z)\beta^\prime(k- u_\eps(t,x))\Big)\notag \\ &\hspace{3cm}\times\psi(t,x)\rho_m(x-y)\rho_n(t-s)J_l\big(k -\beta_{\Tilde{\xi}}(u_\theta^\kappa(s,y))\big)\,dk\,m(dz)\,dx\,dy\,ds\,dt\Big]\notag \\ 
+ & \mathbb{E}\Big[\int_{\Pi^2}\int_{E}\int_{\R}\Big(\beta_\xi\big(k-u_\eps(t,x) - \nu(u_\eps(t,x);z)\big) - \beta_\xi(k-u_\eps(t,x))\Big)\psi(t,x)\rho_m(x-y)\rho_n(t-s)\notag \\ &\hspace{6cm}\times J_l\big(k -\beta_{\Tilde{\xi}}(u_\theta^\kappa(s,y))\big)\,dk\,\widetilde{N}(dz,dt)\,dx\,dy\,ds\Big]\notag \\ 
+ \,& \eps\mathbb{E}\Big[\int_{\Pi^2}\int_{\R}\beta_\xi^\prime(k-u_\eps(t,x))\nabla u_\eps(t,x)\nabla_x\big(\psi(t,x)\rho_m(x-y)\big)\rho_n(t-s)J_l\big(k -\beta_{\Tilde{\xi}}(u_\theta^\kappa(s,y))\big)\,dk\,dx\,dy\,ds\,dt\Big]\notag \\
-&  \eps\mathbb{E}\Big[\int_{\Pi^2}\int_{\R}\beta_\xi^{\prime\prime}(k-u_\eps(t,x))|\nabla u_\eps(t,x)|^2\psi(t,x)\rho_m(x-y)\rho_n(t-s)\notag \\& \hspace{9cm}\times J_l\big(k -\beta_{\Tilde{\xi}}(u_\theta^\kappa(s,y))\big)\,dk\,dx\,dy\,ds\,dt\Big]
\notag \\
\ge \, &\mathbb{E}\Big[\int_{\Pi^2}\int_{\R}\beta_\xi^{\prime\prime}\big(k-u_\eps(t,x)\big)|\nabla u_\eps(t,x)|^2\Phi^\prime(u_\eps(t,x))\psi(t,x)\rho_m(x-y)\rho_n(t-s)\notag \\& \hspace{5cm}\times J_l\big(k -\beta_{\Tilde{\xi}}(u_\theta^\kappa(s,y))\big)\,dk\,dx\,dy\,ds\,dt\Big]\notag \\
& \text{i.e.,}~~~ \sum_{i=1}^{10}\mathcal{H}_i \ge \mathcal{H}_0\,. 
\end{align}
We add \eqref{inq:doublingvariable-1st} and \eqref{inq:doubling-variable-2nd}, and look for a passage to the limit with respect to the various small parameters involved. These limits are established using conventional methods, such as the Lebesgue point theorem, the Lebesgue dominated convergence theorem, the continuity of translations in Lebesgue space, properties of mollification and the {\em a-priori} bound \eqref{inq:apriori-bound}, whose proofs follow from \cite{Bauzet-2012,Majee-2014,Bauzet-2015} and \cite{Majee-2019} under modulo cosmetic changes.\\

Before proceeding further, we recall some easily observed identities, whose proof can be found in \cite{Bendahmane-2005}.
For any $a,b \in \R$, there hold
\begin{align}
    &(a^+ - b)^+ = (a^+ - b^+)^+ - \text{sgn}^+(b^-)b\,,\label{identity-for-dt-psi-0}\\
    &\text{sgn}^+(a^+ - b)(F(a^+) - F(b)) = \text{sgn}^+(a^+ - b^+)(F(a^+) - F(b^+)) - \text{sgn}^+(b^-)F(b)\,, \label{identity-for-dt-psi-01}\\
    &\text{sgn}^+(a^+ - b)\big(\Phi(a^+) - \Phi(b) \big) = \text{sgn}^+\big(a^+ - b^+\big)\big(\Phi(a^+) - \Phi(b^+)\big) - \text{sgn}^+(b^-)\Phi(b)\,.\label{identity-for-dt-psi-02}
\end{align}
Since $\text{supp}(\rho_n) \subset [-\frac{2}{n}, 0]$, one has $\mathcal{G}_1 = 0$. Moreover, we have the following result regarding $\mathcal{H}_1$. 
\begin{align*}
&\underset{l \rightarrow \infty}{\lim}\,\underset{\kappa \rightarrow 0}{\lim}\,\underset{n \rightarrow \infty}{\lim}\,\mathcal{H}_1= \mathbb{E}\Big[\int_{D}\int_{D}\beta_\xi\big(\beta_{\Tilde{\xi}}(u_\theta(0,y)- u_\eps^0(x)\big)\psi(0,x)\rho_m(x-y)\,dx\,dy\Big] =: \Bar{\mathcal{H}}_1\,,\\
        &\underset{\xi \rightarrow 0}{\lim} \underset{\Tilde{\xi} \rightarrow 0}{\lim}\,\Bar{\mathcal{H}}_1=\mathbb{E}\Big[\int_{D}\int_{D}\big((u_\theta^0)^+(y)- u_\eps^0(x)\big)^+\psi(0,x)\rho_m(x-y)\,dx\,dy\Big]\,.\notag 
\end{align*}
Thus, using \eqref{identity-for-dt-psi-0}, we have the following lemma.
\begin{lem}\label{lem:H1-G1}
It holds that,
  \begin{align}
&\underset{\xi \rightarrow 0}{\lim}\,\underset{\tilde{\xi} \rightarrow 0}{\lim}\,\underset{l \rightarrow \infty}{\lim}\,\underset{\kappa \rightarrow 0}{\lim}\,\underset{n \rightarrow \infty}{\lim} (\mathcal{G}_1 + \mathcal{H}_1)\notag \\ & =\int_{D}\int_{D}\big((u_\theta^0)^+(y)- (u_\eps^0)^+(x)\big)^+\psi(0,x)\rho_m(x-y)\,dx\,dy - \int_{D}\text{sgn}^+((u_\eps^0)^-)u_\eps^0(x)\psi(0,x)\bar{\rho}_m(x)\,dx\,.\notag 
\end{align}  
\end{lem}
\noindent Since $(\partial_t + \partial_s)\rho_n(t-s) = 0$, one has
\begin{align*}
   \mathcal{G}_2 + \mathcal{H}_2 =\mathbb{E}\Big[\int_{\Pi^2}\int_{\R}\beta_\xi\big(k-u_\eps(t,x)\big)\partial_t\psi(t,x)\rho_n(t-s)\rho_m(x-y)J_l\big(k -\beta_{\Tilde{\xi}}(u_\theta^\kappa\big)\,dk\,dx\,dy\,ds\,dt\Big],\notag
\end{align*}
and hence, passing to the limits, we get
\begin{align*}
& \underset{l \rightarrow \infty}{\lim}\,\underset{\kappa \rightarrow 0}{\lim}\,\underset{n \rightarrow \infty}{\lim}\,(\mathcal{G}_2 + \mathcal{H}_2) \notag \\
&=  \,\mathbb{E}\Big[\int_{\Pi}\int_{D}\beta_\xi\big(\beta_{\Tilde{\xi}}(u_\theta(t,y))-u_\eps(t,x)\big)\partial_t\psi(t,x)\rho_m(x-y)\,dx\,dy\,dt\Big] =: \bar{\mathcal{H}}_2\,, \notag \\
& \underset{\xi \rightarrow 0}{\lim}\,\underset{\tilde{\xi} \rightarrow 0}{\lim} \bar{\mathcal{H}}_2
=\mathbb{E}\Big[\int_{\Pi}\int_{D}\big(u_\theta^+(t,y)-u_\eps(t,x)\big)^+\partial_t\psi(t,x)\rho_m(x-y)\,dx\,dy\,dt\Big]\,.
\end{align*}
Moreover, thanks to \eqref{identity-for-dt-psi-0}, we have the following result.
\begin{lem}\label{lem:G-2-H-3} It holds that,
    \begin{align}
        \underset{\xi \rightarrow 0}{\lim}\,\underset{\Tilde{\xi} \rightarrow 0}{\lim}\,\underset{l \rightarrow \infty}{\lim}\,\underset{\kappa \rightarrow 0}{\lim}\,\underset{n \rightarrow \infty}{\lim}\,(\mathcal{G}_2 + \mathcal{H}_2) &=   \mathbb{E}\Big[\int_{\Pi}\int_{D}\big(u_\theta^+(t,y)-u_\eps^+(t,x)\big)^+\partial_t\psi(t,x)\rho_m(x-y)\,dx\,dy\,dt\Big]\notag \\ & \hspace{1cm}  -\mathbb{E}\Big[\int_{\Pi}\text{sgn}^+(u_\eps^-(t,x))u_\eps(t,x)\partial_t\psi(t,x)\bar{\rho}_m(x)\,dx\,dt\Big]\,. \notag
    \end{align}
\end{lem}
Regarding the terms associated with the flux functions $\mathcal{G}_3$, $\mathcal{G}_4$, and $\mathcal{H}_3$, we have  
\begin{align}
       &\underset{l \rightarrow \infty}{\lim}\,\underset{\kappa \rightarrow 0}{\lim}\,\underset{n \rightarrow \infty}{\lim}\,(\mathcal{G}_3 + \mathcal{G}_4) \notag \\
 = & - \mathbb{E}\Big[\int_{\Pi}\int_{D}\int_{u_\eps(t,x)}^{u_\theta(t,y)}\beta_{\Tilde{\xi}}^{\prime}(r)\beta_\xi^{\prime}\big(\beta_{\Tilde{\xi}}(r)-u_\eps(t,x)\big)F^\prime(r)\,dr\, \psi(t,x)\nabla_y\rho_m(x-y)\,dy\,dx\,dt\Big]\notag  \\
& \underset{\Tilde{\xi}\rightarrow 0}{\longrightarrow} \, - \mathbb{E}\Big[\int_{\Pi}\int_{D}\int_{u_\eps(t,x)}^{u_\theta(t,y)}\text{sgn}^+(r)\beta_\xi^{\prime}\big(r^+-u_\eps(t,x)\big)F^\prime(r)\,dr\, \psi(t,x)\nabla_y\rho_m(x-y)\,dy\,dx\,dt\Big]\notag \\
& \underset{\xi \rightarrow 0}{\longrightarrow} \, - \mathbb{E}\Big[\int_{\Pi}\int_{D}\int_{u_\eps(t,x)}^{u_\theta(t,y)}\text{sgn}^+(r)\text{sgn}^+\big(r^+ -u_\eps(t,x)\big)F^\prime(r)\,dr\, \psi(t,x)\nabla_y\rho_m(x-y)\,dy\,dx\,dt\Big]\notag\\
&= - \mathbb{E}\Big[\int_{\Pi}\int_{D}\int_{u_\eps^+(t,x)}^{u_\theta^+(t,y)}\text{sgn}^+\big(r^+ -u_\eps(t,x)\big)F^\prime(r)\,dr\, \psi(t,x)\nabla_y\rho_m(x-y)\,dy\,dx\,dt\Big]\notag\\
 &= -\mathbb{E}\Big[\int_{\Pi}\int_{D}\text{sgn}^+\big(u_\theta^+(t,y) -u_\eps(t,x)\big)\big(F(u_\theta^+)-F(u_\eps^+)\big)\, \psi(t,x)\nabla_y\rho_m(x-y)\,dy\,dx\,dt\Big]\,,\label{inq:G3G4} \\
& \underset{l \rightarrow \infty}{\lim}\,\underset{\kappa \rightarrow 0}{\lim}\,\underset{n \rightarrow \infty}{\lim} \mathcal{H}_3 \notag \\ = &\,\mathbb{E}\Big[\int_{\Pi}\int_{D}\int_{\beta_{\Tilde{\xi}}(u_\theta(t,y))}^{u_\eps(t,x)}\beta_{\xi}^\prime(\beta_{\Tilde{\xi}}(u_\theta(t,y))-r)F^\prime(r)\,dr\, \nabla_x\big(\psi(t,x)\rho_m(x-y)\big)\,dx\,dy\,dt\Big] \notag \\
&\underset{\tilde{\xi} \rightarrow 0}{\longrightarrow}\, \mathbb{E}\Big[\int_{\Pi}\int_{D}\int_{u_\theta^+(t,y)}^{u_\eps(t,x)}\beta_{\xi}^\prime(u_\theta^+(t,y)-r)F^\prime(r)\,dr\, \nabla_x\big(\psi(t,x)\rho_m(x-y)\big)\,dx\,dy\,dt\Big] \notag \\
&\underset{\xi \rightarrow 0}{\longrightarrow} \, \mathbb{E}\Big[\int_{\Pi}\int_{D}\int_{u_\theta^+(t,y)}^{u_\eps(t,x)}\text{sgn}^+(u_\theta^+(t,y)-r)F^\prime(r)\,dr\, \nabla_x\big(\psi(t,x)\rho_m(x-y)\big)\,dx\,dy\,dt\Big] \notag \\
= & - \mathbb{E}\Big[\int_{\Pi}\int_{D}\text{sgn}^+(u_\theta^+(t,y)-u_\eps(t,x))\big(F(u_\theta^+(t,y)) - F(u_\eps(t,x))\big)\, \nabla_x\big(\psi(t,x)\rho_m(x-y)\big)\,dx\,dy\,dt\Big]\notag \\
= & - \mathbb{E}\Big[\int_{\Pi}\int_{D}\text{sgn}^+(u_\theta^+(t,y)-u_\eps^+(t,x))\big(F(u_\theta^+(t,y)) - F(u_\eps^+(t,x))\big)\, \nabla_x\big(\psi(t,x)\rho_m(x-y)\big)\,dx\,dy\,dt\Big]\notag \\
& + \mathbb{E}\Big[\int_{\Pi}\int_{D}\text{sgn}^+(u_\eps^-(t,x))F(u_\eps(t,x)) \nabla_x\big(\psi(t,x)\rho_m(x-y)\big)\,dx\,dy\,dt\Big]\,,\label{inq:H3}
\end{align}
thanks to \eqref{identity-for-dt-psi-01}. Combining \eqref{inq:H3} and \eqref{inq:G3G4} and using the fact that $\nabla_y\rho_m(x-y) = - \nabla_x\rho_m(x-y)$ and the identity
$$\text{sgn}^+(a^+-b^+)(F(a^+) - F(b^+)) = \text{sgn}^+(a^+-b)(F(a^+)-F(b^+)),~~a,b\in \R\,,$$
we have the following result.
\begin{lem} 
We have,
    \begin{align} 
    &\underset{\xi, \Tilde{\xi} \rightarrow 0}{\lim}\,\underset{l \rightarrow \infty}{\lim}\,\underset{\kappa \rightarrow 0}{\lim}\,\underset{n \rightarrow \infty}{\lim}\big(\mathcal{G}_3 + \mathcal{G}_4 + \mathcal{H}_3\big)\notag \\ = &  
    - \mathbb{E}\Big[\int_{\Pi}\int_{D}\text{sgn}^+(u_\theta^+(t,y)-u_\eps^+(t,x))\big(F(u_\theta^+(t,y)) - F(u_\eps^+(t,x))\big)\, \nabla_x\psi(t,x)\,dx\,dy\,dt\Big]\notag \\
& + \mathbb{E}\Big[\int_{\Pi}\int_{D}\text{sgn}^+(u_\eps^-(t,x))F(u_\eps(t,x)) \nabla_x\big(\psi(t,x)\rho_m(x-y)\big)\,dx\,dy\,dt\Big]\,.\notag
\end{align}
\end{lem}
Next, we focus on the degenerate terms. One has
\begin{align}
&\underset{l \rightarrow \infty}{\lim}\,\underset{\kappa \rightarrow 0}{\lim}\,\underset{n \rightarrow \infty}{\lim} \mathcal{G}_5\notag \\ &= - \mathbb{E}\Big[\int_{\Pi}\int_{D}\beta_\xi^{\prime}\big(\beta_{\Tilde{\xi}}(u_\theta(t,y))-u_\eps(t,x)\big)\beta_{\Tilde{\xi}}^\prime(u_\theta(t,y))\nabla\Phi(u_\theta(t,y))\psi(t,x)\nabla_y\rho_m(x-y)\,dy\,dx\,dt\Big] \notag \\
& = \mathbb{E}\Big[\int_{\Pi}\int_{D}\int_{u_\eps(t,x)}^{u_\theta(t,y)}\beta_{\Tilde{\xi}}^{\prime}(r)\beta_\xi^{\prime}\big(\beta_{\Tilde{\xi}}(r)-u_\eps(t,x)\big)\Phi^\prime(r)\,dr\,\psi(t,x)\Delta\rho_m(x-y)\,dy\,dx\,dt\Big]\notag \\
& \underset{\Tilde{\xi} \rightarrow 0}{\longrightarrow} \mathbb{E}\Big[\int_{\Pi}\int_{D}\int_{u_\eps(t,x)}^{u_\theta(t,y)}\text{sgn}^+(r)\beta_\xi^{\prime}\big(r^+-u_\eps(t,x)\big)\Phi^\prime(r)\,dr\,\psi(t,x)\Delta\rho_m(x-y)\,dy\,dx\,dt\Big] \notag \\
&\underset{\xi \rightarrow 0}{\longrightarrow}  \mathbb{E}\Big[\int_{\Pi}\int_{D}\int_{u_\eps(t,x)}^{u_\theta(t,y)}\text{sgn}^+(r)\text{sgn}^+\big(r^+-u_\eps(t,x)\big)\Phi^\prime(r)\,dr\,\psi(t,x)\Delta\rho_m(x-y)\,dy\,dx\,dt\Big]\notag \\
& = \mathbb{E}\Big[\int_{\Pi}\int_{D}\text{sgn}^+\big(u_\theta^+(t,y)-u_\eps(t,x)\big)\big(\Phi(u_\theta^+(t,y))-\Phi(u_\eps^+(t,x))\big)\,\psi(t,x)\Delta_y\rho_m(x-y)\,dy\,dx\,dt\Big]\,. \label{inq:g5}
\end{align}
Since $\beta_\xi^{\prime\prime}, \beta_\xi^\prime \ge 0$ and $\Phi$ is non-decreasing, we see that
\begin{align}
&\underset{l \rightarrow \infty}{\lim}\,\underset{\kappa \rightarrow 0}{\lim}\,\underset{n \rightarrow \infty}{\lim} \mathcal{G}_6 = - \mathbb{E}\Big[\int_{\Pi \times D}\beta_\xi^{\prime}\big(\beta_{\Tilde{\xi}}(u_\theta(t,y))-u_\eps(t,x)\big)\beta_{\Tilde{\xi}}^{\prime\prime}(u_\theta(t,y))|\nabla u_\theta(t,y)|^2 \notag \\
& \hspace{5cm} \times \Phi^\prime(u_\theta(t,y))\psi(t,x)\rho_m(x-y)\,dy\,dx\,dt\Big] \le 0\,. \label{inq:G6}
\end{align}

A standard argument reveals that
\begin{align}
&\underset{l \rightarrow \infty}{\lim}\,\underset{\kappa \rightarrow 0}{\lim}\,\underset{n \rightarrow \infty}{\lim}\mathcal{H}_4 \notag \\
& = - \mathbb{E}\Big[\int_{\Pi}\int_{D}\int_{\beta_{\Tilde{\xi}}(u_\theta(t,y))}^{u_\eps(t,x)}\beta_\xi^\prime\big(\beta_{\Tilde{\xi}}(u_\theta(t,y))-r)\Phi^\prime(r)\,dr\,\Delta_x\big(\psi(t,x)\rho_m(x-y)\big)\,dy\,dx\,dt\Big]\notag \\
&  \underset{\Tilde{\xi} \rightarrow 0}{\longrightarrow} - \mathbb{E}\Big[\int_{\Pi}\int_{D}\int_{u_\theta^+(t,y)}^{u_\eps(t,x)}\beta_\xi^\prime\big(u_\theta^+(t,y)-r)\Phi^\prime(r)\,dr\Delta_x\big(\psi(t,x)\rho_m(x-y)\big)\,dy\,dx\,dt\Big]\notag \\
& \underset{\xi \rightarrow 0}{\longrightarrow} \mathbb{E}\Big[\int_{\Pi}\int_{D}\text{sgn}^+(u_\theta^+(t,y)-u_\eps(t,x))\big(\Phi(u_\theta^+(t,y))-\Phi(u_\eps(t,x))\big)\Delta_x\big(\psi(t,x)\rho_m(x-y)\big)\,dy\,dx\,dt\Big]\notag \\
& = \mathbb{E}\Big[\int_{\Pi}\int_{D}\text{sgn}^+(u_\theta^+(t,y)-u_\eps^+(t,x))\big(\Phi(u_\theta^+(t,y))-\Phi(u_\eps^+(t,x))\big)\Delta_x\big(\psi(t,x)\rho_m(x-y)\big)\,dy\,dx\,dt\Big]\notag \\
 & \quad -  \mathbb{E}\Big[\int_{\Pi}\int_{D}\text{sgn}^+(u_\eps^-(t,x))\Phi(u_\eps(t,x))\Delta_x\big(\psi(t,x)\rho_m(x-y)\big)\,dy\,dx\,dt\Big]\,, \label{inq:H4}
\end{align}
where in the last equality, we have used \eqref{identity-for-dt-psi-02}. Using the identity
\begin{align}\label{id:phi}
  \text{sgn}(a^+-b^+)(\Phi(a^+)-\Phi(b^+)) = \text{sgn}(a^+-b)(\Phi(a^+)-\Phi(b^+)),\quad \forall~a,b\in \R\,,  
\end{align}
and combining \eqref{inq:g5}, \eqref{inq:G6} and \eqref{inq:H4}, we get
\begin{align}\label{eq:g5h4}
   &\underset{\xi \rightarrow 0}{\lim}\,\underset{\Tilde{\xi} \rightarrow 0}{\lim}\,\underset{l \rightarrow 0}{\lim}\,\underset{\kappa \rightarrow 0}{\lim}\,\underset{n \rightarrow \infty}{\lim}\big(\mathcal{G}_{5} + \mathcal{G}_{6} + \mathcal{H}_{4}\big)\notag \\= &\, 2\,\mathbb{E}\Big[\int_{\Pi}\int_{D}\text{sgn}^+\big( u_\theta^+(t,y)-u_\eps^+(t,x)\big)\big(\Phi(u_\theta^+(t,y)) - \Phi(u_\eps^+(t,x))\big)\notag \\ & \hspace{4cm}\times\big(\psi(t,x)\Delta_y\rho_m(x-y)- \nabla \psi(t,x)\nabla_y\rho_m(x-y)\big)\,dy\,dx\,dt\Big]  \notag \\
   + &\,\mathbb{E}\Big[\int_{\Pi}\int_{D}\text{sgn}(u_\theta^+(t,y)-u_\eps^+(t,x))\big(\Phi( u_\theta^+(t,y))-\Phi(u_\eps^+(t,x))\big)\Delta\psi(t,x)\rho_m(x-y)\,dy\,dx\,dt\Big]\notag \\ \,
    -&\mathbb{E}\Big[\int_{\Pi}\int_{D}\text{sgn}^+(u_\eps^-(t,x))\Phi(u_\eps(t,x))\Delta_x\big(\psi(t,x)\rho_m(x-y)\big)\,dy\,dx\,dt\Big] \notag \\ := &\sum_{i=1}^{2}\Hat{\mathcal{G}}_{i} - \mathbb{E}\Big[\int_{\Pi}\int_{D}\text{sgn}^+(u_\eps^-(t,x))\Phi(u_\eps(t,x))\Delta_x\big(\psi(t,x)\rho_m(x-y)\big)\,dy\,dx\,dt\Big].
\end{align}
Thanks to the properties of Lebesgue point theorem and mollification, one can pass to the limit in the terms $\mathcal{G}_0$ and $\mathcal{H}_0$ as done in \cite[Lemma 4.1]{Majee-2019} and \cite[Lemma 3.4]{Bauzet-2014}. In fact, we have
\begin{align}\label{eq:g01-h01}
   & \underset{\Tilde{\xi} \rightarrow 0}{\lim}\,\underset{l \rightarrow \infty}{\lim}\,\underset{\kappa \rightarrow 0}{\lim}\,\underset{n \rightarrow \infty}{\lim}\big(\mathcal{G}_0 + \mathcal{H}_0\big) \notag \\
   & = \mathbb{E}\Big[\int_{\Pi}\int_{D}\beta_\xi^{\prime\prime}\big(u_\theta^+(t,y)-u_\eps(t,x)\big)|\nabla u_\theta^+(t,y)|^2\Phi^\prime(u_\theta^+(t,y))\psi(t,x)\rho_m(x-y)\,dy\,dx\,dt\Big]\notag \\
   & + \mathbb{E}\Big[\int_{\Pi}\int_{D}\beta_\xi^{\prime\prime}(u_\theta^+(t,y)-u_\eps(t,x))\Phi^\prime(u_\eps(t,x))|\nabla u_\eps(t,x)|^2\psi(t,x)\rho_m(x-y)\,dx\,dy\,dt\Big]\notag \\
   & = \mathbb{E}\Big[\int_{\Pi}\int_{D}\beta_\xi^{\prime\prime}\big(u_\theta^+(t,y)-u_\eps(t,x)\big)\text{sgn}^+(u_\eps(t,x))|\nabla u_\theta^+(t,y)|^2\Phi^\prime(u_\theta^+(t,y))\psi(t,x)\rho_m(x-y)\,dy\,dx\,dt\Big]\notag \\
& +\mathbb{E}\Big[\int_{\Pi}\int_{D}\beta_\xi^{\prime\prime}\big(u_\theta^+(t,y)-u_\eps(t,x)\big)(1 -\text{sgn}^+(u_\eps(t,x)))|\nabla u_\theta^+(t,y)|^2\Phi^\prime(u_\theta^+(t,y))\psi(t,x)\rho_m(x-y)\,dy\,dx\,dt\Big]\notag \\
   & + \mathbb{E}\Big[\int_{\Pi}\int_{D}\beta_\xi^{\prime\prime}(u_\theta^+(t,y)-u_\eps(t,x))\text{sgn}^+(u_\eps(t,x))\Phi^\prime(u_\eps(t,x))|\nabla u_\eps(t,x)|^2\psi(t,x)\rho_m(x-y)\,dx\,dy\,dt\Big]\notag\\
   & + \mathbb{E}\Big[\int_{\Pi}\int_{D}\beta_\xi^{\prime\prime}(u_\theta^+(t,y)-u_\eps(t,x))\big(1 - \text{sgn}^+(u_\eps(t,x))\big)\Phi^\prime(u_\eps(t,x))|\nabla u_\eps(t,x)|^2\psi(t,x)\rho_m(x-y)\,dx\,dy\,dt\Big]\notag\\
   & \ge \mathbb{E}\Big[\int_{\Pi}\int_{D}\beta_\xi^{\prime\prime}\big(u_\theta^+(t,y)-u_\eps(t,x)\big)\text{sgn}^+(u_\eps(t,x))|\nabla u_\theta^+(t,y)|^2\Phi^\prime(u_\theta^+(t,y))\psi(t,x)\rho_m(x-y)\,dy\,dx\,dt\Big]\notag \\
   &+ \mathbb{E}\Big[\int_{\Pi}\int_{D}\beta_\xi^{\prime\prime}(u_\theta^+(t,y)-u_\eps(t,x))\text{sgn}^+(u_\eps(t,x))\Phi^\prime(u_\eps(t,x))|\nabla u_\eps(t,x)|^2\psi(t,x)\rho_m(x-y)\,dx\,dy\,dt\Big]\, 
   \notag \\ & := \mathcal{G}_{0,1} + \mathcal{H}_{0,1}.
\end{align}
Thanks to the fact that $a^2 + b^2 \ge 2ab,$ for any $a,b \in \R$, one has
\begin{align}
    &\mathcal{G}_{0,1} + \mathcal{H}_{0,1} \ge 2\, \mathbb{E}\Big[\int_{\Pi}\int_{D}\beta_\xi^{\prime\prime}(u_\theta^+(t,y)-u_\eps(t,x))\text{sgn}^+(u_\eps(t,x))\sqrt{\Phi^\prime(u_\eps(t,x))}\nabla u_\eps(t,x)\notag \\ & \hspace{3cm}\times\sqrt{\Phi^\prime(u_\theta^+(t,y))}\nabla u_\theta^+(t,y)\psi(t,x)\rho_m(x-y)\,dx\,dy\,dt\Big] \notag \\
    & =  2\,\mathbb{E}\Big[\int_{\Pi}\int_{D}\sqrt{\Phi^\prime(u_\theta^+(t,y))}\nabla u_\theta^+(t,y)\nabla_x\Big(\int_{u_\theta^+(t,y)}^{u_\eps(t,x)}\beta_\xi^{\prime\prime}(u_\theta^+(t,y)-\sigma)\text{sgn}^+(\sigma)\sqrt{\Phi^\prime(\sigma)}\,d\sigma\Big) \notag \\ & \hspace{4cm}\times \psi(t,x)\rho_m(x-y)\,dx\,dy\,dt\Big]\notag \\
    & =  - 2\,\mathbb{E}\Big[\int_{\Pi}\int_{D}\sqrt{\Phi^\prime(u_\theta^+(t,y))}\nabla u_\theta^+(t,y)\Psi(u_\theta^+(t,y), u_\eps(t,x))\nabla_x\big(\psi(t,x)\rho_m(x-y)\big)\,dx\,dy\,dt\Big] =: \hat{H}\,,\notag
\end{align}
where $\Psi(a,b) = \displaystyle \int_{a}^{b}\beta_\xi^{\prime\prime}(a-\sigma)\text{sgn}^+(\sigma)\sqrt{\Phi^\prime(\sigma)}\,d\sigma.$ Observe that $\Psi(a,b)$ is bounded and $a \mapsto \Psi(a,b)$ is continuous. Thus, using \cite[Lemma 5.3]{Bauzet-2015}, we have
\begin{align}
    \hat{H} & = - 2\,\mathbb{E}\Big[\int_{\Pi}\int_{D}\nabla_y \Big( \int_{u_\eps(t,x)}^{u_\theta^+(t,y)}\sqrt{\Phi^\prime(\mu)}\Psi(\mu, u_\eps(t,x))\,d\mu\Big)\cdot\nabla_x\big(\psi(t,x)\rho_m(x-y)\big)\,dx\,dy\,dt\Big]\notag \\
    & = 2\,\mathbb{E}\Big[\int_{\Pi}\int_{D}\int_{u_\eps(t,x)}^{ u_\theta^+(t,y)}\int_{\mu}^{u_\eps(t,x)}\beta_\xi^{\prime\prime}(\mu -\sigma)\text{sgn}^+(\sigma)\sqrt{\Phi'(\sigma)}\,d\sigma \sqrt{\Phi'(\mu)}\,d\mu\,\notag \\ & \hspace{6cm}\times \text{div}_y\big(\nabla_x(\psi(t,x)\rho_m(x-y))\big)\,dy\,dx\,dt\Big]\,. \notag 
\end{align}
Thanks to Lemma \ref{lem:to-deal-with-degnerate-term}, we have 
\begin{align}\label{eq:sneding-xi-0}
  \underset{\xi \rightarrow 0}{\lim} \int_{\mu}^{u_\eps(t,x)}\beta_\xi^{\prime\prime}(\mu -\sigma)\text{sgn}^+(\sigma)\sqrt{\Phi'(\sigma)}\,d\sigma  = -\text{sgn}^+\big(\mu - u_\eps(t,x)\big)\text{sgn}^+(\mu)\sqrt{\Phi'(\mu)}.  
\end{align}
We use \eqref{eq:sneding-xi-0} and the generalized Fatou's lemma to have
\begin{align}
    &\underset{\xi \rightarrow 0}{\liminf} \big(\mathcal{G}_{0,1} + \mathcal{H}_{0,1}\big)\notag \\  &\ge 
-2\,\mathbb{E}\Big[\int_{\Pi}\int_{D}\int_{u_\eps(t,x)}^{u_\theta^+(t,y)}\text{sgn}^+(\mu)\text{sgn}^+(\mu -u_\eps(t,x))\Phi'(\mu)\,d\mu\,\text{div}_y\big(\nabla_x(\psi(t,x)\rho_m(x-y))\big)\,dy\,dx\,dt\Big]\notag \\ 
& = -2\,\mathbb{E}\Big[\int_{\Pi}\int_{D}\int_{u_\eps^+(t,x)}^{u_\theta^+(t,y)}\text{sgn}^+(\mu)\text{sgn}^+(\mu -u_\eps(t,x))\Phi'(\mu)\,d\mu\,\text{div}_y\big(\nabla_x(\psi(t,x)\rho_m(x-y))\big)\,dy\,dx\,dt\Big]\notag \\
&-2\,\mathbb{E}\Big[\int_{\Pi}\int_{D}\int_{u_\eps(t,x)}^{u_\eps^+(t,y)}\text{sgn}^+(\mu)\text{sgn}^+(\mu -u_\eps(t,x))\Phi'(\mu)\,d\mu\,\text{div}_y\big(\nabla_x(\psi(t,x)\rho_m(x-y))\big)\,dy\,dx\,dt\Big]\notag \\
&=: \sum_{i=1}^2\hat{H}_i. \notag 
\end{align}
It is easy to see that $\hat{H}_2 = 0.$ Moreover, thanks to \eqref{id:phi}, 
$\hat{H}_1$ can be re-written as
\begin{align*}
   &\hat{H}_1 = -2\,\mathbb{E}\Big[\int_{\Pi}\int_{D}\int_{u_\eps^+(t,x)}^{u_\theta^+(t,y)}\text{sgn}^+(\mu -u_\eps(t,x))\Phi'(\mu)\,d\mu\,\text{div}_y\big(\nabla_x(\psi(t,x)\rho_m(x-y))\big)\,dy\,dx\,dt\Big]\notag  \\
   & =-2\,\mathbb{E}\Big[\int_{\Pi}\int_D\text{sgn}^+( u_\theta^+(t,y) - u_\eps(t,x))\big(\Phi(u_\theta^+(t,y))-\Phi(u_\eps^+(t,x))\big)\notag \\
   & \hspace{4cm}\times \big(\psi(t,x)\Delta_y\rho_m(x-y) - \nabla \psi(t,x)\nabla_y\rho_m(x-y)\big) \times \,dy\,dx\,dt\Big] = \hat{\mathcal{G}}_1\,.
\end{align*}
Thus
\begin{align}\label{eq:liminf-hat-G1}
    \underset{\xi \rightarrow 0}{\liminf}\, \big(\mathcal{G}_{0,1} + \mathcal{H}_{0,1} \big) - \Hat{\mathcal{G}}_1 \ge 0. 
\end{align}
\noindent Now, we focus on the terms $\mathcal{G}_{11}$ and $\mathcal{G}_{12}$. Since $\beta' , \beta'' \ge 0$, $\mathcal{G}_{12} \le 0$. Similar to the previous analysis, one has
\begin{align}
    &\underset{\xi,\Tilde{\xi} \rightarrow 0}{\lim}\,\underset{l \rightarrow 0}{\lim}\,\underset{\kappa \rightarrow 0}{\lim}\,\underset{n \rightarrow \infty}{\lim}\mathcal{G}_{11} \notag =  -\theta\mathbb{E}\Big[\int_{\Pi}\int_{D}\text{sgn}^+\big(u_\theta^+(t,y)-u_\eps(t,x)\big)\text{sgn}^+(u_\theta(t,y))\nabla u_\theta(t,y)\notag \\
    & \hspace{1cm} \times \psi(t,x)\nabla_y\rho_m(x-y)\,dy\,dx\,dt\Big] =: \tilde{\mathcal{G}}_{11}\,.\notag
\end{align}
By using Cauchy-Schwartz inequality and \eqref{inq:apriori-bound}, we estimate $\tilde{\mathcal{G}}_{11}$ as
\begin{align}
    \tilde{\mathcal{G}}_{11} &\le \theta C(\psi)\mathbb{E}\Big[\int_{\Pi}\int_{D}|\nabla u_\theta(t,y)||\nabla_y\rho_m(x-y)|\,dy\,dx\,dt\Big]\notag \\  &\le \theta^{\frac{1}{2}}C\Big(\theta\mathbb{E}\Big[\int_{\Pi}|\nabla u_\theta(t,y)|^2\,dy\,dt\Big]\Big)^\frac{1}{2}\Big(\mathbb{E}\Big[\int_{\Pi}\int_{D}|\nabla_y\rho_m(x-y)|^2\,dy\,dx\,dt\Big]\Big)^\frac{1}{2} \le C(\psi, m)\theta^{\frac{1}{2}}.\notag
\end{align}
Similarly, we deduce that
$$\mathcal{H}_{10} \le 0,\quad \underset{\xi,\Tilde{\xi} \rightarrow 0}{\lim}\,\underset{l \rightarrow 0}{\lim}\,\underset{\kappa \rightarrow 0}{\lim}\,\underset{n \rightarrow \infty}{\lim}\mathcal{H}_9 \le C\eps^{\frac{1}{2}}.$$ To summarize, we have the following.
\begin{lem}\label{lem:G-11-H-10} We have,
    \begin{align}
        \underset{\xi \rightarrow 0}{\lim}\,\underset{\Tilde{\xi} \rightarrow 0}{\lim}\,\underset{l \rightarrow 0}{\lim}\,\underset{\kappa \rightarrow 0}{\lim}\,\underset{n \rightarrow \infty}{\lim}\big(\mathcal{G}_{11} + \mathcal{G}_{12} + \mathcal{H}_{9} + \mathcal{H}_{10}\big) \le C(\theta^{\frac{1}{2}} +\eps^{\frac{1}{2}}). \notag
    \end{align}
\end{lem}
Next we consider the It\^o correction terms $\mathcal{G}_7$ and $\mathcal{H}_5$.
\begin{lem}\label{lem:ito-correction} It holds that,
\begin{align}
    &\underset{l \rightarrow 0}{\lim}\,\underset{\kappa \rightarrow 0}{\lim}\,\underset{n \rightarrow \infty}{\lim} \mathcal{G}_7 = \frac{1}{2}\,\mathbb{E}\Big[\int_{\Pi}\int_{D}\Big(\beta_\xi^{\prime\prime}\big(\beta_{\Tilde{\xi}}(u_\theta(t,y))- u_\eps(t,x)\big)|\beta_{\Tilde{\xi}}^\prime(u_\theta(t,y))|^2 \notag \\ & + \beta_\xi^{\prime}\big(\beta_{\Tilde{\xi}}(u_\theta(t,y))- u_\eps(t,x)\big)\beta_{\Tilde{\xi}}^{\prime\prime}(u_\theta(t,y))\Big)|\phi(u_\theta(t,y))|^2\psi(t,x)\rho_m(x-y)\,dy\,dx\,dt\Big],\notag \\
    &\underset{l \rightarrow 0}{\lim}\,\underset{\kappa \rightarrow 0}{\lim}\,\underset{n \rightarrow \infty}{\lim} \mathcal{H}_5 = \frac{1}{2}\mathbb{E}\Big[\int_{\Pi}\int_{D}\beta_\xi^{\prime\prime}\big(\beta_{\Tilde{\xi}}(u_\theta(t,y)-u_\eps(t,x)\big)|\phi( u_\eps(t,x))|^2\psi(t,x)\rho_m(x-y)\,dx\,dy\,dt\Big].\notag
\end{align}
\end{lem}
Note that $\mathcal{G}_8 = 0$ and we can re-write $\mathcal{H}_6$ as follows.
\begin{align*}
    \mathcal{H}_6 = -\,&\mathbb{E}\Big[\int_{\Pi^2}\int_{\R}\beta_\xi^{\prime}(k-u_\eps(t,x))\phi( u_\eps(t,x))\psi(t,x)\rho_m(x-y)\rho_n(t-s)\notag \\ & \times \Big(J_l\big(k -\beta_{\Tilde{\xi}}(u_\theta^\kappa(s,y))\big) - J_l\big(k- \beta_{\Tilde{\xi}}(u_\theta^\kappa(s-\frac{2}{n},y))\big)\Big)\,dk\,dx\,dy\,ds\,dW(t)\Big]\,.\notag
\end{align*}
 A simple application of the It\^o-L\'evy formula on $J_l\big(k -\beta_{\Tilde{\xi}}(u_\theta^\kappa(s,y))\big)$
 along with It\^o-L\'evy product rule and integration by parts, one can reformulate the term $\mathcal{H}_6$ as
\begin{align}
    \mathcal{H}_6  & = -\mathbb{E}\Big[ \int_{\Pi} \int_{\mathbb{R}} \mathcal{K}[\beta'', \rho_{m,n}](s,y,k)\int_{s-\frac{2}{n}}^{s}J_l(k-\beta_{\Tilde{\xi}}(u_\theta^\kappa(r,y)))\beta_{\Tilde{\xi}}^\prime(u_\theta^\kappa(r,y))A_\theta(r,y)\,dr\,dk\,ds\,dy \Big]\notag\\
        &-\,\mathbb{E}\Big[ \int_{\Pi} \int_{\mathbb{R}}\int_D\int_{s-\frac{2}{n}}^s\beta_\xi^{\prime\prime}(k-u_\theta^\kappa(r,y))(\phi( u_\theta(r,y))\ast \rho_\kappa)\phi( u_\eps(t, x))\psi(t,x)\rho_m(x-y)\rho_n(t-r)\notag\\ &\hspace{6cm}\times J_l\big(k -\beta_{\Tilde{\xi}}(u_\theta^\kappa(r,y))\big)\beta_{\Tilde{\xi}}^\prime(u_\theta^\kappa(r,y)\big)\,dr\,dy\,dk\,dt\,dx\Big]\notag \\
        &- \frac{1}{2}\mathbb{E}\Big[\int_{\Pi}\int_{\R}\int_{D}\mathcal{K}[\beta^{\prime\prime\prime}, \rho_{m,n}](s,y,k) \int_{s-\frac{2}{n}}^s J_l\big(k -\beta_{\Tilde{\xi}}(u_\theta^\kappa(r,y))\big)|\beta_{\Tilde{\xi}}^\prime(u_\theta^\kappa(r,y)\big)|^2 \notag \\ &  \hspace{8cm}\times|\phi( u_\theta(r, y))\ast\rho_\kappa|^2\,dr\,dy\,dk\,ds\,dx\Big]\notag \\
        &- \frac{1}{2}\mathbb{E}\Big[\int_{\Pi}\int_{\R}\int_{D}\mathcal{K}[\beta^{\prime\prime}, \rho_{m,n}](s,y,k)\int_{s-\frac{2}{n}}^s J_l\big(k -\beta_{\Tilde{\xi}}(u_\theta^\kappa(r,y))\big)\beta_{\Tilde{\xi}}^{\prime\prime}(u_\theta^\kappa(r,y))\notag \\ & \hspace{8cm}\times|\phi( u_\theta(r, y))\ast\rho_\kappa|^2\,dr\,dy\,dk\,ds\,dx\Big]\notag \\
        &+ \mathbb{E}\Big[\int_{\Pi}\int_{\R}\int_{D}\mathcal{K}[\beta^{\prime\prime\prime}, \rho_{m,n}](s,y,k)\int_{s-\frac{2}{n}}^{s}\int_{E}\int_0^1 (1-\lambda)J_l\big(k -\beta_{\Tilde{\xi}}(u_\theta^\kappa(r,y)) - \lambda\nu( u_\theta(r,y);z)\ast \rho_\kappa\big)\notag \\ & \hspace{5cm}\times|\beta_{\Tilde{\xi}}^\prime(u_\theta^\kappa(r,y))|^2\big|\nu( u_\theta(r,y);z) \ast \rho_\kappa\big|^2\,d\lambda\,m(dz)\,dr\,\,dy\,dk\,ds\,dx \Big]\notag \\&+\mathbb{E}\Big[\int_{\Pi}\int_{\R}\int_{D}\mathcal{K}[\beta^{\prime\prime}, \rho_{m,n}](s,y,k)\int_{s-\frac{2}{n}}^{s}\int_{E}\int_0^1 (1-\lambda) J_l\big(k -\beta_{\Tilde{\xi}}(u_\theta^\kappa(r,y))  - \lambda\nu( u_\theta(r,y);z) \ast \rho_\kappa\big)\notag \\ & \hspace{3cm}\times\beta_{\Tilde{\xi}}^{\prime\prime}(u_\theta^\kappa(r,y))\big|\nu( u_\theta(r,y);z) \ast \rho_\kappa\big|^2\,d\lambda\,m(dz)\,dr\,\,dy\,dk\,ds\,dx \Big]\notag
        =: \sum_{i=1}^6\mathcal{H}_{6,i},
\end{align}
where 
\begin{align*}
& A_\theta(r,y) = \theta \Delta u_\theta^\kappa(r,y) + \text{div}(F(u_\theta(r,y)))\ast \rho_\kappa + \Delta \Phi(u_\theta(r,y)) \ast \rho_\kappa, \\
   & \mathcal{K}[\beta, \rho_{m,n}](s,y,k) := \int_{\Pi}\phi( u_\eps(t,x))\beta_\xi(k- u_\eps(t,x))\rho_{m,n}(t,x,s,y)\,dx\,dW(t),
\end{align*}
with $\rho_{m,n}(t,x,s,y) := \psi(t,x)\rho_m(x-y)\rho_n(t-s)$. Note that since $u_\theta^\kappa\in H^2(D)$, $A_\theta(\cdot,\cdot)$ is square integrable. 
Regarding $\mathcal{K}[\beta,\rho_{m,n}](s,y,k)$, we have the following results whose proof can be found in \cite[Lemma 5.3]{Majee-2014}.\\
It holds that,
\begin{align}\label{lem:for-ito-term}
    &\partial_k\mathcal{K}[\beta, \rho_{m,n}](t,x,k) = \mathcal{K}[\beta', \rho_{m,n}](s,y,k), \quad \partial_y\mathcal{K}[\beta,  \rho_{m,n}](s,y,k) =
    \mathcal{K}[\beta, \partial_y\rho_{m,n}](s,y,k),\notag\\ 
    \text{and}\,,\quad
    &\underset{0 \leq s\leq T }{sup}\,\mathbb{E}\Big[||\mathcal{K}[\beta'',  \rho_{m,n}](s,\cdot,\cdot)||_{L^\infty(D \times \mathbb{R})}^2\Big] \leq \frac{C(\psi)n^{\frac{2(p-1)}{p}}m^{\frac{2}{p}}}{\xi^a},\notag\\
    &\underset{0 \leq s \leq T }{sup}\,\mathbb{E}\Big[||\mathcal{K}[\beta''',  \rho_{m,n}](s,\cdot,\cdot)||_{L^\infty(D\times \mathbb{R})}^2\Big] \leq \frac{C(\psi)n^{\frac{2(p-1)}{p}}m^{\frac{2}{p}}}{\xi^b},
\end{align}
where $p$ is a positive integer of the form $p=2^k$ for some $k \in \mathbb{N}$ with $p \geq d + 3 $ and for some $a, b \ge 0$ which only depends on $d$ and hence on $p$.

By using \eqref{lem:for-ito-term}, Young's inequality for convolution and replacing $u_\theta$ by $u_\theta^\kappa$ to adopt the same line of argument as in \cite{Majee-2015, Majee-2019}, one easily get $\mathcal{H}_{6,1}  \rightarrow 0$ as $n \rightarrow \infty.$
Moreover, as $u_0^\theta \in L^4(D)$, we have $u_\theta \in L^4(\Omega \times D)$ for each $t \in (0, T]$. Thus, using the assumption \ref{A5}, Cauchy-Schwarz inequality along with \eqref{lem:for-ito-term}, we have
\begin{align}
    \mathcal{H}_{6,5} &\le C\,\mathbb{E}\Big[\int_{\Pi}\int_{\R}\int_{D}\int_{s-\frac{2}{n}}^{s}\|\mathcal{K}[\beta^{\prime\prime\prime}, \rho_{m,n}](s,\cdot,\cdot)\|_{\infty}\int_{E}\int_0^1 J_l\big(k -\beta_{\Tilde{\xi}}(u_\theta^\kappa(r,y)) - \lambda\nu( u_\theta(r,y);z)\ast \rho_\kappa\big)\notag \\ & \hspace{5cm}\times\big|\nu( u_\theta(r,y);z) \ast \rho_\kappa\big|^2\,d\lambda\,m(dz)\,dr\,\,dy\,dk\,ds\,dx \Big]\notag \\
    & \le C\,\mathbb{E}\Big[\int_0^T\int_{s-\frac{2}{n}}^{s}\|\mathcal{K}[\beta^{\prime\prime\prime}, \rho_{m,n}](s,\cdot,\cdot)\|_{\infty}\int_{E}\|\nu( u_\theta(r,\cdot);z) \ast \rho_\kappa\|_{L^2(D)}^2\,m(dz)\,dr\,ds \Big] \notag \\
    & \le C\,\mathbb{E}\Big[\int_0^T\int_{s-\frac{2}{n}}^{s}\|\mathcal{K}[\beta^{\prime\prime\prime}, \rho_{m,n}](s,\cdot,\cdot)\|_{\infty}\Big(\int_{E}|g(z)|^2\,m(dz)\Big)\| u_\theta(r,\cdot)\|_{L^2(D)}^2\,dr\,ds \Big]\notag \\
    & \le C\,\Big(\underset{0 \le s \le T}{\sup}\mathbb{E}\big[\|\beta^{\prime\prime\prime}, \rho_{m,n}](s, \cdot, \cdot)\|_\infty^2\big]\Big)^\frac{1}{2}\int_0^T\int_{s-\frac{2}{n}}^s\Big(\mathbb{E}\big[\| u_\theta(r, \cdot)\|_{L^2(D)}^4\big]\Big)^\frac{1}{2}\,dr\,ds \notag \\
    & \le C(\xi,m)\frac{n^\frac{p-1}{p}}{n}\underset{0 \le r \le T}{\sup}\Big(\mathbb{E}\big[\|u_\theta(r, \cdot)\|_{L^4(D)}^2\big]\Big)^\frac{1}{2} \le \frac{C}{n^\frac{1}{p}} \rightarrow 0 \quad\text{as} \quad n \rightarrow \infty\,. \notag
\end{align}
Similarly, one can show that $\mathcal{H}_{6,3}, \mathcal{H}_{6,4}$, $\mathcal{H}_{6,6}$ $\rightarrow 0$ as $n \rightarrow \infty.$
Again, it is routine to pass to the limit in $\mathcal{H}_{6,2}$ and arrive at the following conclusion.
\begin{align*}
    &\underset{l \rightarrow \infty}{\lim}\,\underset{\kappa \rightarrow 0}{\lim}\,\underset{n \rightarrow \infty}{\lim}\, \mathcal{H}_{6,2}\notag \\  &= -\,\mathbb{E}\Big[ \int_{\Pi} \int_D\beta_\xi^{\prime\prime}(\beta_{\Tilde{\xi}}(u_\theta(t,y))-u_\eps(t,x))\beta_{\Tilde{\xi}}^\prime(u_\theta(t,y))\phi( u_\eps(t,x))\phi( u_\theta(t, y))\psi(t,x)\rho_m(x-y) \,dt\,dx\,dy\Big]\,.
\end{align*}
Combining the above estimations with Lemma \ref{lem:ito-correction}, and using \eqref{inq:for-beta-xi} along with the assumptions \ref{A3}-\ref{A4}, we get 
\begin{align}
    &\underset{l \rightarrow \infty}{\lim}\,\underset{\kappa \rightarrow 0}{\lim}\,\underset{n \rightarrow \infty}{\lim}\,\big(\mathcal{G}_7 + \mathcal{G}_8 + \mathcal{H}_5 + \mathcal{H}_6\big)\notag \\ & \le \frac{1}{2}\,\mathbb{E}\Big[\int_{\Pi}\int_{D}\beta_\xi^{\prime\prime}\big(\beta_{\Tilde{\xi}}(u_\theta(t,y))- u_\eps(t,x)\big)|\beta_{\Tilde{\xi}}^\prime(u_\theta(t,y))|^2|\phi(u_\theta(t,y))|^2\psi(t,x)\rho_m(x-y)\,dy\,dx\,dt\Big]\notag \\
    & +\frac{1}{2}\mathbb{E}\Big[\int_{\Pi}\int_{D}\beta_\xi^{\prime\prime}\big(\beta_{\Tilde{\xi}}(u_\theta(t,y)-u_\eps(t,x)\big)|\phi( u_\eps(t,x))|^2\psi(t,x)\rho_m(x-y)\,dx\,dy\,dt\Big] \notag \\
    & - \mathbb{E}\Big[ \int_{\Pi} \int_D\beta_\xi^{\prime\prime}(\beta_{\Tilde{\xi}}(u_\theta(t,y))-u_\eps(t,x))\beta_{\Tilde{\xi}}^\prime(u_\theta(t,y))\phi( u_\eps(t,x))\phi( u_\theta(t, y))\psi(t,x)\rho_m(x-y) \,dt\,dx\,dy\Big] \notag \\
     & +\frac{1}{2}\,\mathbb{E}\Big[\int_{\Pi}\int_{D}\beta_\xi^{\prime}\big(\beta_{\Tilde{\xi}}(u_\theta(t,y))- u_\eps(t,x)\big)\beta_{\Tilde{\xi}}^{\prime\prime}(u_\theta(t,y))|\phi(u_\theta(t,y))|^2\psi(t,x)\rho_m(x-y)\,dy\,dx\,dt\Big]\notag \\
     & \le \frac{1}{2}\,\mathbb{E}\Big[\int_{\Pi}\int_{D}\beta_\xi^{\prime\prime}\big(\beta_{\Tilde{\xi}}(u_\theta(t,y))- u_\eps(t,x)\big)\big|\beta_{\Tilde{\xi}}^\prime(u_\theta(t,y))\phi(u_\theta(t,y)) - \phi(u_\eps(t,x))\big|^2\notag \\ &\hspace{5cm}\times\psi(t,x)\rho_m(x-y)\,dy\,dx\,dt\Big] + C(\psi)\Tilde{\xi}\notag \\
& \underset{\Tilde{\xi} \rightarrow 0}{\longrightarrow} \frac{1}{2}\,\mathbb{E}\Big[\int_{\Pi}\int_{D}\beta_\xi^{\prime\prime}\big(u_\theta^+(t,y))- u_\eps(t,x)\big)\big|\phi(u_\theta^+) - \phi(u_\eps)\big|^2\psi(t,x)\rho_m(x-y)\,dy\,dx\,dt\Big]\le 
     C\xi \,. \label{inq:g7g8h5h6}
\end{align}
Thus, we have
\begin{lem}\label{lem:G7-H6}
  \begin{align*}
    \underset{\xi \rightarrow 0}{\lim}\,\underset{\Tilde{\xi} \rightarrow 0}{\lim}\,\underset{l \rightarrow \infty}{\lim}\,\underset{\kappa \rightarrow 0}{\lim}\,\underset{n \rightarrow \infty}{\lim}\,\big(\mathcal{G}_7 + \mathcal{G}_8 + \mathcal{H}_5 + \mathcal{H}_6\big) \le 0.
\end{align*}  
\end{lem}
Next, we focus on the additional terms $\mathcal{G}_9$ and $\mathcal{H}_7$ due to jump noise. Following similar lines of argument as done in \cite{Majee-2015, Majee-2019} along with the properties of convolution, we conclude the following result.
\begin{lem}\label{lem:correction-term-due-levy}
    \begin{align*}
        &\underset{\Tilde{\xi} \rightarrow 0}{\lim}\,\underset{l \rightarrow \infty}{\lim}\,\underset{\kappa \rightarrow 0}{\lim}\,\underset{n \rightarrow \infty}{\lim} \big(\mathcal{G}_9 + \mathcal{H}_7\big)\notag \\ = &\,\mathbb{E}\Big[\int_{\Pi}\int_{D}\int_{E}\Big(\beta_\xi\big(u_\theta^+(t,y) + \nu( u_\theta(t,y); z) - u_\eps(t,x)\big) - \beta_{\xi}(u_\theta^+(t,y)-u_\eps(t,x))\notag\\  & \hspace{2cm}- \nu( u_\theta^+(t,y); z)\beta_\xi^\prime\big(u_\theta^+(t,y) -u_\eps(t,x)\big)\Big)\psi(t,x)\rho_m(x-y)\,m(dz)\,dy\,dx\,dt\Big]\notag \\ 
        + & \mathbb{E}\Big[\int_{\Pi}\int_{D}\int_{E}\Big(\beta_\xi\big(u_\theta^+(t,y)-u_\eps(t,x) - \nu(u_\eps(t,x);z)\big) - \beta_\xi(u_\theta^+(t,y)-u_\eps(t,x))\notag \\ & \hspace{2cm} + \nu(u_\eps(t,x);z)\beta^\prime(u_\theta^+(t,y)- u_\eps(t,x))\Big)
        \psi(t,x)\rho_m(x-y)\,m(dz)\,dx\,dy\,dt\Big]\,.
    \end{align*}
\end{lem}
Note that $\mathcal{G}_{10} = 0$. Invoking similar argument as done for $\mathcal{H}_6$ and making use of similar estimations as in \cite[Lemma 5.4 \& 5.5]{Majee-2015},  we pass to the limit in $\mathcal{H}_{8}$ to arrive at 
\begin{align}
&\underset{\Tilde{\xi} \rightarrow 0}{\lim}\,\underset{l \rightarrow \infty}{\lim}\,\underset{\kappa \rightarrow 0}{\lim}\,\underset{n \rightarrow \infty}{\lim}\, \mathcal{H}_{8}\notag \\   
=&\, \mathbb{E}\Big[\int_{\Pi}\int_{D}\int_{E}\Big\{\beta_\xi\big(u_\theta^+(t,y) + \nu( u_\theta(t,y);z)-u_\eps(t,x) - \nu(u_\eps(t,x);z)\big)\notag \\
&\hspace{2cm} - \beta_\xi\big(u_\theta^+(t,y) + \nu( u_\theta(t,y);z)-u_\eps(t,x)\big) + \beta_{\xi}\big(u_\theta^+(t,y))- u_\eps(t,x)\big)\notag \\
&\hspace{2.5cm} -\beta_\xi\big(u_\theta^+(t,y))-u_\eps(t,x) - \nu(u_\eps(t,x);z)\big)\Big\}\psi(t,x)\rho_m(x-y)\,m(dz)\,dx\,dy\,dt\Big]\notag \\
= & \,\mathbb{E}\Big[\int_{\Pi}\int_{D}\int_{E}\Big\{\beta_\xi\big(u_\theta^+(t,y) + \nu( u_\theta^+(t,y);z)-u_\eps(t,x) - \nu(u_\eps(t,x);z)\big)\notag \\ 
&\hspace{2cm} - \beta_\xi\big(u_\theta^+(t,y) + \nu( u_\theta(t,y);z)-u_\eps(t,x)\big) + \beta_{\xi}\big(u_\theta^+(t,y))- u_\eps(t,x)\big)\notag \\
&\hspace{2.5cm} -\beta_\xi\big(u_\theta^+(t,y))-u_\eps(t,x) - \nu(u_\eps(t,x);z)\big)\Big\}\psi(t,x)\rho_m(x-y)\,m(dz)\,dx\,dy\,dt\Big]\notag \\
 &+ \mathbb{E}\Big[\int_{\Pi}\int_{D}\int_{E}\Big(\beta_\xi\big(u_\theta^+(t,y) + \nu( u_\theta(t,y);z)-u_\eps(t,x) - \nu(u_\eps(t,x);z)\big)\notag \\ 
 & \hspace{2cm}- \beta_\xi\big(u_\theta^+(t,y) + \nu( u_\theta^+(t,y);z)-u_\eps(t,x) - \nu(u_\eps(t,x);z)\big)\Big)\psi(t,x)\rho_m(x-y)\,m(dz)\,dx\,dy\,dt\Big]\notag \\
 &=: \sum_{i=1}^{2}\mathcal{H}_{8,i}\,. \notag
\end{align}
Since $\beta_\xi^\prime \ge 0 $, $\nu$ is non-decreasing and $u_\theta \le u_\theta^+$, one has
\begin{align}
    \mathcal{H}_{8,2} = \mathbb{E}\Big[\int_{\Pi}\int_{D}\int_{E}\beta_\xi^\prime(\hat{a})\big(\nu(u_\theta(t,y);z)-\nu(u_\theta^+(t,y);z)\big)\psi(t,x)\rho_m(x-y)\,m(dz)\,dx\,dy\,dt\Big] \le 0,\notag
\end{align}
for some $\hat{a}\in \Big( u_\theta^+(t,y) + \nu( u_\theta;z)-u_\eps(t,x) - \nu(u_\eps;z),u_\theta^+(t,y) + \nu( u_\theta^+;z)-u_\eps(t,x) - \nu(u_\eps;z)\Big)$.
Thus, we get
\begin{align}\label{eq:H-8-4}
    \underset{\Tilde{\xi} \rightarrow 0}{\lim}\,\underset{l \rightarrow \infty}{\lim}\,\underset{\kappa \rightarrow 0}{\lim}\,\underset{n \rightarrow \infty}{\lim}\, \mathcal{H}_{8} \le \mathcal{H}_{8,1}.
\end{align}
Lemma \ref{lem:correction-term-due-levy} and \eqref{eq:H-8-4} gives
\begin{align}\label{eq-g9g10h7h8}
    &\underset{\Tilde{\xi} \rightarrow 0}{\lim}\,\underset{l \rightarrow \infty}{\lim}\,\underset{\kappa \rightarrow 0}{\lim}\,\underset{n \rightarrow \infty}{\lim}\, \big(\mathcal{G}_9 + \mathcal{G}_{10} + \mathcal{H}_7 + \mathcal{H}_8\big)\notag \\
 & = \mathbb{E}\Big[\int_{\Pi}\int_{D}\int_{E}(1-\lambda)p^2\beta_{\xi}^{\prime\prime}(q + \lambda p) \psi(t,x)\rho_m(x-y)\,m(dz)\,dx\,dy\,dt\Big]\,,
\end{align}
where $p = \nu(u_\theta^+(t,y);z) - \nu(u_\eps(t,x);z)$ and $q=u_\theta^+(t,y) - u_\eps(t,x)$. In view of the assumptions \ref{A6}, we see that
\begin{align*}
    &p^2\beta^{\prime\prime}_{\xi}(q + \lambda p) \le (\lambda^\star)^2q^2\beta^{\prime\prime}_{\xi}(q + \lambda p)g^2(z).
\end{align*}
For $q \ge 0$, since $\nu$ is non-decreasing, it holds that
\begin{align*}
    0 \le q \le (q + \lambda p) \quad \text{for any} \quad \lambda \in [0, 1]\,.
\end{align*}
Thus, using \eqref{inq:for-beta-xi}, one has for $q \ge 0$
\begin{align}\label{inq:q-bigger-0}
    p^2\beta^{\prime\prime}_{\xi}(q + \lambda p) \le (\lambda^\star)^2(q + \lambda p)^2\beta^{\prime\prime}_{\xi}(q + \lambda p)g^2(z) \le C\xi g^2(z).
\end{align}
Since $\nu$ is non-decreasing, $p \le 0$ when $q < 0.$ Thus, for $q < 0$, we get
\begin{align}\label{inq:q-smaller-0}
   \beta^{\prime\prime}_\xi(q + \lambda p) = 0 \,. 
\end{align} 
 Combining \eqref{inq:q-bigger-0} and \eqref{inq:q-smaller-0} in \eqref{eq-g9g10h7h8}, we infer
 \begin{align}
     \underset{\Tilde{\xi} \rightarrow 0}{\lim}\,\underset{l \rightarrow \infty}{\lim}\,\underset{\kappa \rightarrow 0}{\lim}\,\underset{n \rightarrow \infty}{\lim}\, \big(\mathcal{G}_9 + \mathcal{G}_{10} + \mathcal{H}_7 + \mathcal{H}_8\big) \le C\xi.\label{inq:g9g10h7h8}
 \end{align}
Thus, we conclude the following lemma.
\begin{lem}\label{lem:G9-H8}
    \begin{align*}
         &\underset{\xi \rightarrow 0}{\lim}\,\underset{\Tilde{\xi} \rightarrow 0}{\lim}\,\underset{l \rightarrow \infty}{\lim}\,\underset{\kappa \rightarrow 0}{\lim}\,\underset{n \rightarrow \infty}{\lim}\, \big(\mathcal{G}_9 + \mathcal{G}_{10} + \mathcal{H}_7 + \mathcal{H}_8\big) \le 0.
    \end{align*}
\end{lem}
In view of Lemmas \ref{lem:H1-G1}-\ref{lem:G-11-H-10}, \ref{lem:G7-H6}
and \ref{lem:G9-H8} along with \eqref{eq:g5h4} and \eqref{eq:g01-h01}, we have
\begin{align}\label{eq:summarized}
\underset{\xi \rightarrow 0}{\lim}\, (\mathcal{G}_{0,1} + \mathcal{H}_{0,1}) & \le \int_{D}\int_{D}\big((u_\theta^0)^+(y)- (u_\eps^0)^+(x)\big)^+\psi(0,x)\rho_m(x-y)\,dx\,dy \notag \\ & + \mathbb{E}\Big[\int_{\Pi}\int_{D}\big(u_\theta^+(t,y)-u_\eps^+(t,x)\big)^+\partial_t\psi(t,x)\rho_m(x-y)\,dx\,dy\,dt\Big]\notag \\ & -\mathbb{E}\Big[\int_{\Pi}\int_{D}F^+(u_\theta^+(t,y), u_\eps^+(t,x))\nabla_x\psi(t,x)\rho_m(x-y)\,dx\,dy\,dt\Big] + \hat{\mathcal{G}}_1 + \hat{\mathcal{G}}_2 + \mathcal{F}_{u_\eps}(\psi\bar{\rho}_m) \notag \\
& + C(\theta^{\frac{1}{2}} + \eps^{\frac{1}{2}}) =: \sum_{i=1}^{3}\hat{\mathcal{H}}_i + \sum_{i=1}^{2}\hat{\mathcal{G}}_i + \mathcal{F}_{u_\eps}(\psi\bar{\rho}_m) + C(\theta^{\frac{1}{2}} + \eps^{\frac{1}{2}}),
\end{align}
where
\begin{align*}
  \mathcal{F}_{u_\eps}(\psi\bar{\rho}_m) : =
&- \mathbb{E}\Big[\int_{\Pi}\int_{D}\text{sgn}^+(u_\eps^-(t,x))\big \{ u_\eps(t,x)\partial_t\psi(t,x)\rho_m(x-y) - F(u_\eps(t,x)) \nabla_x\big(\psi(t,x)\rho_m(x-y)\big)\notag \\ & + \Phi(u_\eps(t,x))\Delta_x\big(\psi(t,x)\rho_m(x-y)\big) \big \}\,dx\,dt\,dy\Big] -\int_{D}\text{sgn}^+((u_\eps^0)^-)u_\eps^0(x)\psi(0,x)\bar{\rho}_m(x)\,dx\,.
\end{align*}
Using \eqref{eq:liminf-hat-G1} in \eqref{eq:summarized}, we get
\begin{align}\label{inq:surviving-inq}
    0 \le \underset{\xi \rightarrow 0}{\liminf}\, (\mathcal{G}_{0,1} + \mathcal{H}_{0,1}) - \hat{\mathcal{G}}_1 \le \sum_{i=1}^{3}\hat{\mathcal{H}}_i + \hat{\mathcal{G}}_2 + \mathcal{F}_{u_\eps}(\psi\bar{\rho}_m)+ C(\theta^{\frac{1}{2}} + \eps^{\frac{1}{2}}).
\end{align}
To proceed further, we would like to pass to the limit with respect to $\theta$ and $\eps$ in \eqref{inq:surviving-inq} in the Young measure sense. We refer to Dafermos \cite{dafermos} for the deterministic setting and  Balder \cite{Balder} for the stochastic version of Young measure theory. For this purpose, we fix $\theta$ and pass to the limit in $\eps$ in the sense of Young measure. Let us define
\begin{align}
    &\mathcal{C}(t,x, \omega; \mu) := \int_{D}\text{sgn}^+(u_\theta^+(t,y)-\mu)\big(\Phi( u_\theta^+(t,y))-\Phi(\mu)\big)\Delta\psi(t,x)\rho_m(x-y)\,dy.\notag
\end{align}
Observe that $\mathcal{C}$ is a Carath\'eodory function on $\Pi \times \Omega \times \R$ and $\{\mathcal{C}(\cdot, u_\eps)\}$ is  uniformly bounded sequence in $L^2(\Omega \times \Pi)$.
Thus, using \cite[Lemma 4.3]{Majee-2014} we have
\begin{align}
    &\underset{\eps \rightarrow 0}{\lim}\,\hat{\mathcal{G}}_2=
    \mathbb{E}\Big[\int_{\Pi}\int_{D}\int_0^1\text{sgn}^+(u_\theta^+(t,y)-{\tt v}^+(t,x, \gamma))\big(\Phi( u_\theta^+(t,y))-\Phi({\tt v}^+(t,x, \gamma))\big)\notag \\
    & \hspace{3cm} \times \Delta\psi(t,x)\rho_m(x-y)\,d\gamma\,dy\,dx\,dt\Big]\equiv {\pmb G}_2\,,\notag
\end{align}
where ${\tt v}(t,x, \gamma)$ is the Young measure-valued limit  of $\{u_\eps(t,x)\}$. Similarly,
we get
\begin{align}
   &\underset{\theta \rightarrow 0}{\lim}\, {\pmb G}_2
   = \mathbb{E}\Big[\int_{\Pi}\int_{D}\int_0^1\int_0^1\text{sgn}^+({\tt u}^+(t,y, \alpha)-{\tt v}^+(t,x, \gamma))\big(\Phi( {\tt u}^+(t,y, \alpha))-\Phi({\tt v}^+(t,x, \gamma))\big)\notag \\ & \hspace{4.5cm}\times\Delta\psi(t,x)\rho_m(x-y)\,d\alpha\,d\gamma\,dy\,dx\,dt\Big],\notag 
\end{align}
where ${\tt u}(t, y, \alpha)$ is the Young measure-valued limit  of $\{u_\theta(t,y)\}$. Furthermore, one can pass to the limit with respect to $\eps$ and $\theta$ in other terms of \eqref{inq:surviving-inq} to get
\begin{align}\label{inq:1st-step-kato}
0 & \le  \int_{D}\int_{D}\big(\hat{u}_0^+(y)- u_0^+(x)\big)^+\psi(0,x)\rho_m(x-y)\,dx\,dy\notag \\ 
+&\,\mathbb{E}\Big[\int_{\Pi}\int_{D}\int_0^1\int_0^1\big({\tt u}^+(t, y, \alpha)-{\tt v}^+(t,x, \gamma)\big)^+\partial_t\psi(t,x)\rho_m(x-y)\,d\alpha\,d\gamma\,dx\,dy\,dt\Big]\notag \\
    -&\,\mathbb{E}\Big[\int_{\Pi}\int_{D}\int_0^1\int_0^1F^+({\tt u}^+(t, y, \alpha), {\tt v}^+(t,x, \gamma))\nabla_x\psi(t,x)\rho_m(x-y)\,d\alpha\,d\gamma\,dx\,dy\,dt\Big] \notag \\  
+&\,\mathbb{E}\Big[\int_{\Pi}\int_{D}\int_0^1\int_0^1\Phi^+({\tt u}^+(t, y, \alpha), {\tt v}^+(t,x, \gamma))\Delta\psi(t,x)\rho_m(x-y)\,d\alpha\,d\gamma\,dy\,dx\,dt\Big] + \mathcal{F}_{{\tt v}}(\psi\bar{\rho}_m)\notag \\:= & \sum_{i=1}^4 {\tt M}_i + \mathcal{F}_{{\tt v}}(\psi\bar{\rho}_m)\,,
\end{align}
where
\begin{align*}
    \mathcal{F}_{{\tt v}}(\psi\bar{\rho}_m)  : = &- \mathbb{E}\Big[\int_{\Pi}\int_{D}\int_0^1\text{sgn}^+({\tt v}^-(t,x, \gamma))\big \{{\tt v}(t,x,\gamma)\partial_t\psi(t,x)\rho_m(x-y)\notag \\ & - F({\tt v}(t,x, \gamma)) \nabla_x\big(\psi(t,x)\rho_m(x-y)\big) + \Phi({\tt v}(t,x, \gamma))\Delta_x\big(\psi(t,x)\rho_m(x-y)\big) \big \}\,d\gamma\,dx\,dt\,dy\Big] \notag \\ & -\int_{D}\text{sgn}^+((u_0(x))^-)u_0(x)\psi(0,x)\bar{\rho}_m(x)\,dx\,. 
\end{align*}
Observe that, thanks to Cauchy-Schwartz inequality
\begin{align}
   &\Big|{\tt M}_4 - \mathbb{E}\Big[\int_{\Pi}\int_0^1\int_0^1\Phi^+({\tt u}^+(t, x, \alpha), {\tt v}(t,x, \gamma))\Delta\psi(t,x)\,d\alpha\,d\gamma\,dx\,dt\Big]\Big| \notag \\ & \le \mathbb{E}\Big[\int_{\Pi}\int_{D}\int_0^1\int_0^1\Big|\Phi^+({\tt u}^+(t, y, \alpha), {\tt v}(t,x, \gamma)) - \Phi^+({\tt u}^+(t, x, \alpha), {\tt v}(t,x, \gamma))\Big|\notag \\ & \hspace{4cm}\times\Delta\psi(t,x)\rho_m(x-y)\,d\alpha\,d\gamma\,dy\,dx\,dt\Big]\notag \\
   & \le C(\psi)\Big(\mathbb{E}\Big[\int_{\Pi}\int_{D}\int_0^1\big|{\tt u}^+(t, y, \alpha)- {\tt u}^+(t,x, \alpha)\big|^2\rho_m(x-y)\,d\alpha\,dy\,dx\,dt\Big]\Big)^{\frac{1}{2}} \rightarrow 0 \quad \text{as} \quad m \rightarrow \infty.\notag 
\end{align}
Moreover, the entropy formulation for $k = 0$ with any regular non-negative test function $\varpi$ infers that
\begin{align}
 0\le & \,\mathbb{E}\Big[\int_{\Pi}\Big\{\beta_{\xi}(-u_\eps)\partial_t\varpi + \nabla\varpi\int_0^{u_\eps}\beta_{\xi}^{\prime}(-r)F^{\prime}(r)\,dr  -   \Delta\varpi \int_0^{u_\eps}\beta_{\xi}^{\prime}(-r)\Phi^{\prime}(r)\,dr + \eps\beta_{\xi}(-u_\eps)\nabla u_\eps \nabla \varpi \Big\}\,dx\,dt\Big] \notag \\
 &\quad  + \frac{1}{2}\mathbb{E}\Big[\int_{\Pi}\beta_{\xi}^{\prime\prime}(-u_\eps)|\phi(u_\eps)|^2\varpi\,dx\,dt\Big] + \int_{D}\beta_\xi(-u_\eps^0)\varpi(0)\,dx  \notag \\
  & \qquad + \mathbb{E}\Big[\int_{\Pi}\big(\beta_\xi(-u_\eps-\nu(u_\eps;z)) - \beta_\xi(-u_\eps) + \nu(u_\eps;z)\beta_\xi^\prime(-u_\eps)\big)\varpi\,dx\,dt\Big]\,.\label{inq:entropy-k-0}
\end{align}
Note that one can follow similar lines of argument as in \eqref{inq:g7g8h5h6} and \eqref{inq:g9g10h7h8} to conclude that the correction terms in \eqref{inq:entropy-k-0}, due to Brownian noise and jump noise, go to $0$ as $\xi \rightarrow 0$. Passing to the limit as $\xi$ tends to $0$ and then $\eps \goto 0$ in \eqref{inq:entropy-k-0}, we obtain 
\begin{align}
    0 &\le \mathbb{E}\Big[\int_{\Pi}\int_0^1\Big\{(-{\tt v})^+\partial_t\varpi + \text{sgn}^+(-{\tt v})F({\tt v})\nabla\varpi - \text{sgn}^+(-{\tt v})\Phi({\tt v})\Delta\varpi \Big\}\,d\gamma\,dx\,dt\Big] + \int_{D}(-u_0)^+\varpi(0)\,dx\notag \\
    &=  \mathbb{E}\Big[\int_{\Pi}\int_0^1\Big\{{\tt v}^-\partial_t\varpi + F(-{\tt v}^-)\nabla\varpi -\Phi(-{\tt v}^-)\Delta\varpi \Big\}\,d\gamma\,dx\,dt\Big] +\int_{D}(u_0)^-\varpi(0)\,dx  = \mathcal{F}_{{\tt v}}(\varpi). \label{inq:f-v-psi-non-negative}
\end{align}
Thus, $\mathcal{F}_{{\tt v}}$ is a linear and non-negative operator over $\mathcal{D}^+([0,T] \times \R^d)$. Since $0 \le \psi\bar{\rho}_m \le \psi\bar{\rho}_{m+1} \le \psi$,  $\mathcal{F}_{{\tt v}}(\psi\bar{\rho}_m)$ exists as $m \rightarrow \infty$.
In view of the above discussion, we pass to the limit as $m$ tends to $\infty$ in \eqref{inq:1st-step-kato}, to get
\begin{align}\label{inq:1st  step to 1st half of kato}
0 & \le\,\mathbb{E}\Big[\int_{\Pi}\int_0^1\int_0^1\big({\tt u}^+(t, x, \alpha)-{\tt v}^+(t,x, \gamma)\big)^+\partial_t\psi(t,x)\,d\alpha\,d\gamma\,dx\,dt\Big]\notag \\
    &-\,\mathbb{E}\Big[\int_{\Pi}\int_0^1\int_0^1F^+({\tt u}^+(t, x, \alpha), {\tt v}^+(t,x, \gamma))\nabla\psi(t,x)\,d\alpha\,d\gamma\,dx\,dt\Big] +  \int_{D}\big(\hat{u}_0^+(x)- u_0^+(x)\big)^+\psi(0,x)\,dx \notag \\  
&+\,\mathbb{E}\Big[\int_{\Pi}\int_0^1\int_0^1\Phi^+({\tt u}^+(t, x, \alpha), {\tt v}^+(t,x, \gamma))\Delta\psi(t,x)\,d\alpha\,d\gamma\,dx\,dt\Big] + \underset{m \rightarrow \infty}{\lim}\,\mathcal{F}_{{\tt v}}(\psi\bar{\rho}_m). 
\end{align}
\begin{rem}{(\textbf{Local Kato's inequality.})}\label{rem:local-kato} Similar to Section \ref{sec:The first half of Kato's inequality}, for $\psi \in \mathcal{D}^+([0,T] \times D)$, we apply It\^o-le\'vy formula to $\beta_\xi(u_\theta^\kappa(s,y)-k)\psi\rho_m\rho_n$, multiply by $J_l(u_\eps(t,x)-k)$, take expectation and integrate with respect $t,x,k$ to get an entropy inequality  like \eqref{inq:doublingvariable-1st}. Again, applying the It\^o-L\'evy formula to $\int_D \beta_\xi(k-u_\eps(t,x))\psi\rho_m\rho_n\,dx$, multiplying by $J_l(k- u_\theta^\kappa(s,y))$ and integrating with respect to $k,s,y$, taking expectation, one gets the entropy inequality like \eqref{inq:doubling-variable-2nd}. Adding both these inequalities and passing to the limit with respect to various small parameters, by following similar set of argument as done in Section \ref{sec:The first half of Kato's inequality}, we arrive at the following local Kato's inequality: for any $ \psi\in \mathcal{D}^+([0,T] \times D)$, there holds
\begin{align} 
0 &\le\,\mathbb{E}\Big[\int_{\Pi}\int_0^1\int_0^1\big({\tt u}(t, x, \alpha)-{\tt v}(t,x, \gamma)\big)^+\partial_t\psi(t,x)\,d\alpha\,d\gamma\,dx\,dt\Big]\notag \\
    &-\,\mathbb{E}\Big[\int_{\Pi}\int_0^1\int_0^1F^+({\tt u}(t, x, \alpha), {\tt v}(t,x, \gamma))\cdot \nabla\psi(t,x)\,d\alpha\,d\gamma\,dx\,dt\Big] +  \int_{D}\big(\hat{u}_0(x)- u_0(x)\big)^+\psi(0,x)\,dx \notag \\  
&+\,\mathbb{E}\Big[\int_{\Pi}\int_0^1\int_0^1\Phi^+({\tt u}(t, x, \alpha), {\tt v}(t,x, \gamma))\Delta\psi(t,x)\,d\alpha\,d\gamma\,dx\,dt\Big]\,. \notag
\end{align}
\end{rem}
\subsubsection{\textbf{The second half of Kato's Inequality}} 
In this subsection, we are interested in driving the global Kato's inequality for $(b^- - a^-)^+$, as done for the first part of global Kato's inequality. Observe that $-u_\eps$ is a weak solution to the problem 
\begin{equation}\label{eq:-negative-u-eps}
        \displaystyle dv_\eps - \text{div}(\Tilde{F}(v_\eps))\,dt - \Delta \Tilde{\Phi}(v_\eps)\,dt  = \Tilde{\phi}(v_\eps)\,dW(t) + \int_{E} \Tilde{\nu}(v_\eps;z)\widetilde{N}(dz,dt) + \eps\Delta v_\eps,~\text{in}~\Omega \times \Pi,
\end{equation} 
with initial condition $-u_0$, where
$$\Tilde{F}(v_\eps):= - F(-v_\eps),~~ \Tilde{\Phi}(v_\eps):= -\Phi(-v_\eps),~~ \Tilde{\phi}(v_\eps):= - \phi(-v_\eps),~~\Tilde{\nu}(v_\eps;z) := -\nu(-v_\eps;z).$$ The same justification applies to $-u_\theta$. Since both $u_\eps$ and $u_\theta$ play symmetric roles, it is tempting to simply swap out $u_\theta$ and $u_\eps$ and use similar lines of reasoning to acquire  Kato's inequity as done in the previous subsection. However, the argument used in the proof adapted according to viscous solution $u_\theta$ doesn't work if we replace it with $u_\eps$ and vice-versa. Therefore, we prove the second half of Kato's inequality using different argument as needed and illustrate in the following discussion.
\vspace{0.1cm}

To this end, we choose the same partition of unity for $\psi$ in $\mathcal{D}^+([0, T) \times \R^d)$ with $\text{supp}(\psi) \subset \mathcal{B}:= \mathcal{B}_i$ for some $i \in \{1,\cdot\cdot\cdot,\bar{k}\},$ with shifted sequence of mollifier $\rho_m$ in $\R^d$ and $\rho_n$ with $\text{supp}(\rho_n) \subset [-\frac{2}{n}, 0]$ as discussed in Section \ref{sec:Global-Kato -ine}. However, the test function will be $\psi(s,y)\rho_m(y-x)\rho_n(t-s)$. Applying the It\^o-L\'evy formula to $\beta_\xi(\beta_{\Tilde{\xi}}(u_\eps(t,x)) -k)\psi\rho_m\rho_n$, multiplying by $J_l(u_\theta^\kappa(s, y) - k)$, taking expectation and integrating with respect to $t,x,k$, we have
\begin{align}\label{inq:doublingvariabl-3rd}
&\mathbb{E}\Big[\int_{\Pi}\int_{D}\int_{\R}\beta_\xi\big(\beta_{\Tilde{\xi}}(u_\eps(0,x))-k\big)\psi(s,y)\rho_n(-s)\rho_m(y-x)J_l\big(u_\theta^\kappa(s, y) - k\big)\,dk\,dy\,dx\,ds\Big] \notag \\
+\,&\mathbb{E}\Big[\int_{\Pi^2}\int_{\R}\beta_\xi\big(\beta_{\Tilde{\xi}}(u_\eps(t,x))-k\big)\psi(s,y)\partial_t\rho_n(t-s)\rho_m(y-x)J_l\big(u_\theta^\kappa(s, y) - k\big)\,dk\,dy\,ds\,dx\,dt\Big] \notag \\
+ \,& \mathbb{E}\Big[\int_{\Pi^2}\int_{\R}\int_{k}^{\beta_\xi(u_\eps(t,x))}\beta_\xi^{\prime}(r-k)F^\prime(r)\,dr\,\psi(s,y)\rho_n(t-s)\Delta_x\rho_m(y-x)J_l\big(u_\theta^\kappa(s, y) - k\big)\,dk\,dy\,ds\,dx\,dt\Big]\notag \\
- \,& \mathbb{E}\Big[\int_{\Pi^2}\int_{\R}\beta_\xi^{\prime}\big(\beta_{\Tilde{\xi}}(u_\eps(t,x))-k\big)\beta_{\Tilde{\xi}}^\prime(u_\eps(t,x))\nabla\Phi(u_\eps(t,x))\psi(s,y)\rho_n(t-s)\notag \\&\hspace{6cm}\times\nabla_x\rho_m(y-x)J_l\big(u_\theta^\kappa(s, y) - k\big)\,dk\,dy\,ds\,dx\,dt\Big]\notag\\
- \,& \mathbb{E}\Big[\int_{\Pi^2}\int_{\R} \beta_\xi^{\prime}\big(\beta_{\Tilde{\xi}}(u_\eps(t,x))-k\big)\beta_{\Tilde{\xi}}^{\prime\prime}(u_\eps(t,x))\nabla u_\eps(t,x)\nabla\Phi(u_\eps(t,x))\notag \\&\hspace{6cm}\times\psi(s,y)\rho_n(t-s)\rho_m(y-x)J_l\big(u_\theta^\kappa(s, y) - k\big)\,dk\,dy\,ds\,dx\,dt\Big]\notag \\
+\, &\frac{1}{2}\,\mathbb{E}\Big[\int_{\Pi^2}\int_{\R}\Big(\beta_\xi^{\prime\prime}\big(\beta_{\Tilde{\xi}}(u_\eps(t,x))-k\big)|\beta_{\Tilde{\xi}}^\prime(u_\eps(t,x))|^2 + \beta_\xi^{\prime}\big(\beta_{\Tilde{\xi}}(u_\eps(t,x))-k\big)\beta_{\Tilde{\xi}}^{\prime\prime}(u_\eps(t,x))\Big)\notag \\&\hspace{3cm}\times\big|\phi(u_\eps(t,x))\big|^2\psi(s,y)\rho_n(t-s)\rho_m(y-x)J_l\big(u_\theta^\kappa(s, y) - k\big)\,dk\,dy\,ds\,dx\,dt\Big]\notag \\
+&\,\mathbb{E}\Big[\int_{\Pi^2}\int_{\R}\beta_\xi^{\prime}\big(\beta_{\Tilde{\xi}}(u_\eps(t,x))-k\big)\beta_{\Tilde{\xi}}^\prime(u_\eps(t,x))\phi(u_\eps(t,x))\psi(s,y)\rho_n(t-s)\notag \\&\hspace{6cm}\times\rho_m(y-x)J_l\big(u_\theta^\kappa(s, y) - k\big)\,dk\,dy\,ds\,dx\,dW(t)\Big]\notag \\
+&\,\mathbb{E}\Big[\int_{\Pi^2}\int_{\R}\int_{E}\Big(\beta_\xi\big(\beta_{\Tilde{\xi}}(u_\eps(t,x)) + \nu( u_\eps(t,x); z)-k\big) - \beta_{\xi}\big(\beta_{\Tilde{\xi}}(u_\eps(t,x))-k\big)\notag\\  & \hspace{2cm}- \nu( u_\eps(t,x); z)\beta_\xi^\prime\big(\beta_{\Tilde{\xi}}(u_\eps(t,x) -k)\big)\beta_{\Tilde{\xi}}^\prime(u_\eps(t,x)\Big)\psi(s,y)\rho_n(t-s)\rho_m(y-x)\notag\\&\hspace{6cm} \times J_l\big(u_\theta^\kappa(s, y) - k\big)\,dk\,m(dz)\,dy\,ds\,dx\,dt\Big]\notag \\
+ &\,\mathbb{E}\Big[\int_{\Pi^2}\int_{\R}\int_{E}\Big(\beta_\xi\big(\beta_{\Tilde{\xi}}(u_\eps(t,x)) + \nu( u_\eps(t,x); z) -k\big) - \beta_{\xi}\big(\beta_{\Tilde{\xi}}(u_\eps(t,x))-k\big)\Big)\notag \\ &\hspace{3cm}\times\psi(s,y)\rho_n(t-s)\rho_m(y-x)J_l\big(u_\theta^\kappa(s, y) - k\big)\,dk\,\Tilde{N}(dz,dt)\,dy\,dx\,ds\Big] \notag \\
- \,& \eps\mathbb{E}\Big[\int_{\Pi^2}\int_{\R}\beta_\xi^{\prime}\big(\beta_{\Tilde{\xi}}(u_\eps(t,x))-k\big)\beta_{\Tilde{\xi}}^\prime(u_\eps(t,x))\nabla u_\eps(t,x)\psi(s,y)\rho_n(t-s)\notag \\&\hspace{6cm}\times\nabla_x\rho_m(y-x)J_l\big(u_\theta^\kappa(s, y) - k\big)\,dk\,dy\,ds\,dx\,dt\Big]\notag\\
- \,& \eps\mathbb{E}\Big[\int_{\Pi^2}\int_{\R} \Big(\beta_\xi^{\prime}\big(\beta_{\Tilde{\xi}}(u_\eps(t,x))-k\big)\beta_{\Tilde{\xi}}^{\prime\prime}(u_\eps(t,x)) + \beta_\xi^{\prime\prime}\big(\beta_{\Tilde{\xi}}(u_\eps(t,x))-k\big)|\beta_{\Tilde{\xi}}^\prime(u_\eps(t,x))|^2\Big)|\nabla u_\eps(t,x)|^2\notag \\&\hspace{4cm}\times\psi(s,y)\rho_n(t-s)\rho_m(y-x)J_l\big(u_\theta^\kappa(s, y) - k\big)\,dk\,dy\,ds\,dx\,dt\Big]\notag \\
\ge & \, \mathbb{E}\Big[\int_{\Pi^2}\int_{\R}\Big(\beta_\xi^{\prime\prime}\big(\beta_{\Tilde{\xi}}(u_\eps(t,x))-k\big)|\beta_{\Tilde{\xi}}^\prime(u_\eps(t,x))|^2 \nabla u_\eps(t,x)\nabla\Phi(u_\eps(t,x))\notag \\&\hspace{2cm}\times\psi(s,y)\rho_n(t-s)\rho_m(y-x)J_l\big(u_\theta^\kappa(s, y) - k\big)\,dk\,dy\,ds\,dx\,dt\Big]
=: \sum_{i=1}^{11}\mathcal{I}_{i} \ge \mathcal{I}_0\,. 
\end{align}
On the other hand the It\^o-L\'evy formula applied to $\int_{D}\beta_{\xi}(k - u_\theta^\kappa(s, y))\
\psi(s,y)\rho_n(t-s)\rho_m(y-x)\,dy$ together with the multiplication by $J_l(k- \beta_{\xi}(u_\eps(t,x)))$ and integration with respect to $k,t,x$ yields
\begin{align}\label{inq:doubling-variable-4th}
&\mathbb{E}\Big[\int_{\Pi}\int_{D}\int_{\R}\beta_\xi\big(k- u_\theta(0,y)\big)\psi(0,y)\rho_n(t)\rho_m(y-x)J_l\big(k -\beta_{\Tilde{\xi}}(u_\eps(t,x))\big)\,dk\,dx\,dy\,dt\Big] \notag \\   
+\,& \mathbb{E}\Big[\int_{\Pi^2}\int_{\R}\beta_\xi\big(k-u_\theta^\kappa(s, y)\big)\partial_s\big(\psi(s,y)\rho_n(t-s)\big)\rho_m(y-x)J_l\big(k -\beta_{\Tilde{\xi}}(u_\eps(t,x))\big)\,dk\,dx\,dy\,ds\,dt\Big]\notag \\
+ \,& \mathbb{E}\Big[\int_{\Pi^2}\int_{\R}\beta_\xi^\prime(k-u_\theta^\kappa(s,y))(F(u_\theta(s,y))\ast \rho_\kappa)\Delta_y\big(\psi(s,y)\rho_m(y-x)\big)\rho_n(t-s)\notag \\ & \hspace{8cm}\times J_l\big(k -\beta_{\Tilde{\xi}}(u_\eps(t,x))\big)\,dk\,dx\,dy\,ds\,dt\Big]\notag \\
- \,& \mathbb{E}\Big[\int_{\Pi^2}\int_{\R}\beta_\xi^{\prime\prime}(k-u_\theta^\kappa(s,y))(F(u_\theta(s,y))\ast \rho_\kappa)\nabla u_\theta^\kappa(s,y)\nabla_y\big(\psi(s,y)\rho_m(y-x)\big)\rho_n(t-s)\notag \\ & \hspace{8cm}\times J_l\big(k -\beta_{\Tilde{\xi}}(u_\eps(t,x))\big)\,dk\,dx\,dy\,ds\,dt\Big]\notag \\
+ \,& \mathbb{E}\Big[\int_{\Pi^2}\int_{\R}\beta_\xi^\prime\big(k-u_\theta^\kappa(s, y)\big)(\nabla\Phi(u_\theta(s, y))\ast \rho_\kappa)\nabla_y\big(\psi(s,y)\rho_m(y-x)\big)\rho_n(t-s)\notag \\ &\hspace{8cm} \times J_l\big(k -\beta_{\Tilde{\xi}}(u_\eps(t,x))\big)\,dk\,dx\,dy\,ds\,dt\Big]\notag \\
+ \,& \frac{1}{2}\mathbb{E}\Big[\int_{\Pi^2}\int_{\R}\beta_\xi^{\prime\prime}(k-u_\theta^\kappa(s, y))|\phi( u_\theta(s, y))\ast \rho_\kappa|^2\psi(s,y)\rho_m(y-x)\rho_n(t-s) \notag \\
& \hspace{8cm} \times J_l\big(k -\beta_{\Tilde{\xi}}(u_\eps(t,x))\big)\,dk\,dx\,dy\,ds\,dt\Big]\notag \\
-\,&\mathbb{E}\Big[\int_{\Pi^2}\int_{\R}\beta_\xi^{\prime}(k-u_\theta^\kappa(s, y))(\phi( u_\theta(s, y))\ast \rho_\kappa)\psi(s,y)\rho_m(y-x)\rho_n(t-s) \notag \\
& \hspace{7cm} \times J_l\big(k -\beta_{\Tilde{\xi}}(u_\eps(t,x))\big)\,dk\,dx\,dy\,dt\,dW(s)\Big]\notag \\
+ &\mathbb{E}\Big[\int_{\Pi^2}\int_{E}\int_{\R}\Big(\beta_\xi\big(k-u_\theta^\kappa(s, y) - \nu(u_\theta(s, y);z)\ast \rho_\kappa\big) - \beta_\xi(k-u_\theta^\kappa(s, y))\notag \\ & \hspace{3cm}  + (\nu(u_\theta(s, y);z)\ast \rho_\kappa)\beta^\prime(k- u_\theta^\kappa(s, y))\Big)\psi(s,y)\rho_m(y-x)\rho_n(t-s) \notag \\
& \hspace{6cm} \times J_l\big(k -\beta_{\Tilde{\xi}}(u_\eps(t,x))\big)\,dk\,m(dz)\,dx\,dy\,ds\,dt\Big]\notag \\ 
+ & \mathbb{E}\Big[\int_{\Pi^2}\int_{E}\int_{\R}\Big(\beta_\xi\big(k-u_\theta^\kappa(s, y) - \nu(u_\theta(s, y);z)\ast \rho_\kappa\big) - \beta_\xi(k-u_\theta^\kappa(s, y))\Big)\psi(s,y)\rho_m(y-x)\rho_n(t-s)\notag \\ &\hspace{6cm}\times J_l\big(k -\beta_{\Tilde{\xi}}(u_\eps(t,x))\big)\,dk\,\widetilde{N}(dz,ds)\,dx\,dy\,dt\Big]\notag \\ 
+ \,& \theta\mathbb{E}\Big[\int_{\Pi^2}\int_{\R}\beta_\xi^\prime(k-u_\theta^\kappa(s, y))\nabla u_\theta^\kappa(s, y)\nabla_x\big(\psi(s,y)\rho_m(y-x)\big)\rho_n(t-s)J_l\big(k -\beta_{\Tilde{\xi}}(u_\eps(t,x))\big)\,dk\,dx\,dy\,ds\,dt\Big]\notag \\
-&  \theta\mathbb{E}\Big[\int_{\Pi^2}\int_{\R}\beta_\xi^{\prime\prime}(k-u_\theta^\kappa(s, y))|\nabla u_\theta^\kappa(s, y)|^2\psi(s,y)\rho_m(y-x)\rho_n(t-s)\notag \\& \hspace{9cm}\times J_l\big(k -\beta_{\Tilde{\xi}}(u_\eps(t,x))\big)\,dk\,dx\,dy\,ds\,dt\Big]
\notag \\
\ge \, &\mathbb{E}\Big[\int_{\Pi^2}\int_{\R}\beta_\xi^{\prime\prime}\big(k-u_\theta^\kappa(s, y)\big)|\nabla u_\theta^\kappa(s, y)|(\nabla\Phi(u_\theta(s, y))\ast\rho_\kappa)\psi(s,y)\rho_m(y-x)\rho_n(t-s)\notag \\& \hspace{5cm}\times J_l\big(k -\beta_{\Tilde{\xi}}(u_\eps(t,x))\big)\,dk\,dx\,dy\,ds\,dt\Big]
=: \sum_{i=1}^{11}\mathcal{J}_i \ge \mathcal{J}_0\,. 
\end{align}
Similar to the previous section, we  add \eqref{inq:doublingvariabl-3rd} and \eqref{inq:doubling-variable-4th} and pass to the limit with respect to $n, \kappa, l, \Tilde{\xi}, \xi, \eps, \theta$ and  $m$. Since  $\text{supp}(\rho_n) \subset [-\frac{2}{n}, 0]$,  we see that $\mathcal{J}_1 = 0$. Similar to Lemma \ref{lem:H1-G1}, we have
\begin{lem}\label{lem:I-1-J-1}
\begin{align}
 \underset{\xi \rightarrow 0}{\lim}\,\underset{\tilde{\xi} \rightarrow 0}{\lim}\,\underset{l \rightarrow \infty}{\lim}\,\underset{\kappa \rightarrow 0}{\lim}\,\underset{n \rightarrow \infty}{\lim} (\mathcal{I}_1 + \mathcal{J}_1) = & \int_{D}\int_{D}((u_\eps^0)^+(x))-(u_\theta^0)^+(y)\big)^+\psi(0,y)\rho_m(y-x)\,dy\,dx\notag \\
 & - \int_D\text{sgn}^+((u_\theta^0)^-(y))u_\theta^0(y)\psi(0,y)\bar{\rho}_m(y)\,dy\,.\notag
\end{align}
\end{lem} 
To deal with $\mathcal{I}_2$ and $\mathcal{J}_2$, we pass to the limit as done in Lemma \ref{lem:G-2-H-3}, to get 
\begin{lem} 
\begin{align} 
    \underset{\xi \rightarrow 0}{\lim}\,\underset{\tilde{\xi} \rightarrow 0}{\lim}\,\underset{l \rightarrow \infty}{\lim}\,\underset{\kappa \rightarrow 0}{\lim}\,\underset{n \rightarrow \infty}{\lim} (\mathcal{I}_2 + \mathcal{J}_2)= &  \mathbb{E}\Big[\int_{\Pi}\int_{D}\big(u_\eps^+(t,x)-u_\theta^+(t,y)\big)^+\partial_t\psi(t,y)\rho_m(y-x)\,dx\,dy\,dt\Big] \notag \\
    & - \mathbb{E}\Big[\int_{\Pi}\text{sgn}^+(u_\theta^-(t,y))u_\theta(t,y)\partial_t\psi(t,y)\bar{\rho}_m(y)\,dy\Big]\,.\notag
\end{align}    
\end{lem}
Regarding the terms associated with flux function, we have the following lemma.
\begin{lem} 
It holds that
    \begin{align} 
        \underset{\xi \rightarrow 0}{\lim}\,\underset{\tilde{\xi} \rightarrow 0}{\lim}\,\underset{l \rightarrow \infty}{\lim}\,\underset{\kappa \rightarrow 0}{\lim}\,\underset{n \rightarrow \infty}{\lim} \big(\mathcal{I}_3 +\sum_{i=3}^4 \mathcal{J}_i\big) = &  - \mathbb{E}\Big[\int_{\Pi}\int_{D}F^+\big(u_\eps^+(t,x), u_\theta^+(t,y)\big)\nabla_y\psi(t,y)\rho_m(y-x)\,dx\,dy\,dt\Big]\notag\\
        & + \mathbb{E}\Big[\int_{\Pi}\int_{D}\text{sgn}^+(u_\theta^-)F(u_\theta)\nabla_y\big(\psi(t,y)\rho_m(y-x)\big)\,dx\,dy\,dt\Big]\notag\,.
    \end{align}
\end{lem}
Now, we focus on the degenerate terms. Observe that $\underset{l \rightarrow \infty}{\lim}\,\underset{\kappa \rightarrow 0}{\lim}\,\underset{n \rightarrow \infty}{\lim}  \mathcal{I}_5 \le 0,$ and 
\begin{align}\label{eq:hatJ-1-2}
   &\underset{\xi \rightarrow 0}{\lim}\,\underset{\tilde{\xi} \rightarrow 0}{\lim}\,\underset{l \rightarrow \infty}{\lim}\,\underset{\kappa \rightarrow 0}{\lim}\,\underset{n \rightarrow \infty}{\lim} \big(\mathcal{I}_4 + \mathcal{J}_5)\notag \\  & =  \mathbb{E}\Big[\int_{\Pi}\int_{D}\text{sgn}^+\big(u_\eps^+(t,x)-u_\theta(t,y)\big)\big(\Phi(u_\eps^+(t,x)) - \Phi(u_\theta^+(t,y))\big)\psi(t,y)\Delta_x\rho_m(y-x)\,dy\,dx\,dt\Big]\notag\\
   & +\mathbb{E}\Big[\int_{\Pi}\int_{D}\text{sgn}^+\big(u_\eps^+(t,x)-u_\theta^+(t,y)\big)\big(\Phi(u_\eps^+(t,x)) - \Phi(u_\theta^+(t,y))\big)\Delta_y\big(\psi(t,y)\rho_m(y-x)\big)\,dx\,dy\,dt\Big]\notag \\
   &-\mathbb{E}\Big[\int_{\Pi}\int_{D}\text{sgn}^+(u_\theta^-(t,y))\Phi(u_\theta^+(t,y))\Delta_y\big(\psi(t,y)\rho_m(y-x)\big)\,dx\,dy\,dt\Big]\notag \\
   & = 2\mathbb{E}\Big[\int_{\Pi}\int_{D}\text{sgn}^+\big(u_\eps^+(t,x)-u_\theta^+(t,y)\big)\big(\Phi(u_\eps^+(t,x)) - \Phi(u_\theta^+(t,y))\big)\notag \\& \hspace{6cm}\times\big(\psi(t,y)\Delta_x\rho_m(y-x) - \nabla\psi(t,y)\nabla_x\rho_m(y-x)\big)\,dx\,dy\,dt\Big]\notag \\&+\mathbb{E}\Big[\int_{\Pi}\int_{D}\text{sgn}^+\big(u_\eps^+(t,x)-u_\theta^+(t,y)\big)\big(\Phi(u_\eps^+(t,x)) - \Phi(u_\theta^+(t,y))\big)\Delta_y\psi(t,y)\rho_m(y-x)\,dx\,dy\,dt\Big]\notag \\
   &-\mathbb{E}\Big[\int_{\Pi}\int_{D}\text{sgn}^+(u_\theta^-(t,y))\Phi(u_\theta^+(t,y))\Delta_y\big(\psi(t,y)\rho_m(y-x)\big)\,dx\,dy\,dt\Big]\notag \\
   &=: \sum_{i=1}^2 \hat{\mathcal{J}}_i -\mathbb{E}\Big[\int_{\Pi}\int_{D}\text{sgn}^+(u_\theta^-(t,y))\Phi(u_\theta^+(t,y))\Delta_y\big(\psi(t,y)\rho_m(y-x)\big)\,dx\,dy\,dt\Big].
\end{align}
Using the properties of mollifier and the Lebesgue point theorem, passing to the limit in $\mathcal{I}_0$ and $\mathcal{J}_0$ is routine. Thanks to similar reasoning as done for $\mathcal{G}_0$ and $\mathcal{H}_0$, we have 
\begin{align*}
&\underset{\tilde{\xi} \rightarrow 0}{\lim}\,\underset{l \rightarrow \infty}{\lim}\,\underset{\kappa \rightarrow 0}{\lim}\,\underset{n \rightarrow \infty}{\lim} \big(\mathcal{I}_0 + \mathcal{J}_0) \notag \\ & =  
\mathbb{E}\Big[\int_{\Pi}\int_{D}\beta_\xi^{\prime\prime}\big(u_\eps^+(t,x))-u_\theta(t,y)\big)|\nabla u_\eps^+(t,x)|^2\Phi^\prime(u_\eps^+(t,x))\psi(t,y)\rho_m(y-x)\,dy\,dx\,dt\Big]\notag \\
& + \mathbb{E}\Big[\int_{\Pi}\int_{D}\beta_\xi^{\prime\prime}\big(u_\eps^+(t,x)-u_\theta(t,y)\big)|\nabla u_\theta(t,y)|^2\Phi^\prime(u_\theta(t,y))\psi(t,y)\rho_m(y-x)\,dx\,dy\,dt\Big]\notag \\
&\ge \mathbb{E}\Big[\int_{\Pi}\int_{D}\beta_\xi^{\prime\prime}\big(u_\eps^+(t,x))-u_\theta(t,y)\big)\text{sgn}^+(u_\theta(t,y))|\nabla u_\eps^+(t,x)|^2\Phi^\prime(u_\eps^+(t,x))\psi(t,y)\rho_m(y-x)\,dy\,dx\,dt\Big]\notag \\
& + \mathbb{E}\Big[\int_{\Pi}\int_{D}\beta_\xi^{\prime\prime}\big(u_\eps^+(t,x)-u_\theta(t,y)\big)\text{sgn}^+(u_\theta(t,y))|\nabla u_\theta(t,y)|^2\Phi^\prime(u_\theta(t,y))\psi(t,y)\rho_m(y-x)\,dx\,dy\,dt\Big]\notag \\
&\ge 2\mathbb{E}\Big[\int_{\Pi}\int_{D}\int_{u_\theta(t,y)}^{u_\eps^+(t,x)}\int_{\mu}^{u_\theta(t,y)}\beta^{\prime\prime}(\mu-\sigma)\text{sgn}^+(\sigma)\sqrt{\Phi^\prime(\sigma)}\,d\sigma \sqrt{\Phi^\prime(\mu)}\,d\mu \notag \\ & \hspace{6cm}\times\text{div}_x\big(\nabla_y\big(\psi(t,y)\rho_m(y-x)\big)\big)\,dx\,dy\,dt\Big] =: \Hat{\mathcal{I}_0} +\Hat{\mathcal{J}_0}\,.
\end{align*}
Using Lemma \ref{lem:to-deal-with-degnerate-term}, one has
\begin{align}\label{eq:lim-xi-0-degnerate}
    \underset{\xi \rightarrow 0}{\lim}\,\int_{\mu}^{u_\theta(t,y)}\beta^{\prime\prime}(\mu-\sigma)\text{sgn}^+(\sigma)\sqrt{\Phi^\prime(\sigma)}\,d\sigma = -\text{sgn}^+\big(\mu- u_\theta(t,y)\big)\text{sgn}^+(\mu)\sqrt{\Phi^\prime(\mu)}.
\end{align}
With \eqref{eq:lim-xi-0-degnerate} in hand and an application of generalized Fatou's lemma yields
\begin{align}\label{eq:xi-0-hatJ-1}
    &\underset{\xi \rightarrow 0}{\liminf}\big(\Hat{\mathcal{I}_0} +\Hat{\mathcal{J}_0}\big)\notag \\
    & \ge 
-2\mathbb{E}\Big[\int_{\Pi}\int_{D}\int_{u_\theta(t,y)}^{u_\eps^+(t,x)}\text{sgn}^+(\mu)\text{sgn}^+\big(\mu- u_\theta(t,y)\big)\Phi^\prime(\mu)\,d\mu\, \text{div}_x\big(\nabla_y\big(\psi(t,y)\rho_m(y-x)\big)\big)\,dx\,dy\,dt\Big]\notag \\ & = - 2\mathbb{E}\Big[\int_{\Pi}\int_{D}\int_{u_\theta^+(t,y)}^{u_\eps^+(t,x)}\text{sgn}^+\big(\mu- u_\theta(t,y)\big)\Phi^\prime(\mu)\,d\mu\, \text{div}_x\big(\nabla_y\big(\psi(t,y)\rho_m(y-x)\big)\big)\,dx\,dy\,dt\Big]\notag = \Hat{\mathcal{J}}_1.\notag \\
    & \Rightarrow \underset{\xi \rightarrow 0}{\liminf}\big(\Hat{\mathcal{I}_0} +\Hat{\mathcal{J}_0}\big) - \Hat{\mathcal{J}}_1 \ge 0.
\end{align}
Similar to Lemma \ref{lem:G-11-H-10}, we have
\begin{lem}\label{lem:i1011j1011}
    \begin{align*}
        \underset{\xi \rightarrow 0}{\lim}\underset{\Tilde{\xi} \rightarrow 0}{\lim}\,\underset{l \rightarrow \infty}{\lim}\,\underset{\kappa \rightarrow 0}{\lim}\,\underset{n \rightarrow \infty}{\lim} \big(\mathcal{I}_{10}+ \mathcal{I}_{11} + \mathcal{J}_{10} + \mathcal{J}_{11}\big) \le C(\theta^{\frac{1}{2}} + \eps^{\frac{1}{2}}).
    \end{align*}
\end{lem}
Regarding the It\^o correction term, we have the following result.
\begin{lem}\label{lem:I6-I7}It holds that,
    \begin{align}
        \underset{l \rightarrow \infty}{\lim}\,\underset{\kappa \rightarrow 0}{\lim}\,\underset{n \rightarrow \infty}{\lim} \mathcal{I}_6 = \, &\frac{1}{2}\,\mathbb{E}\Big[\int_{\Pi}\int_{D}\Big(\beta_\xi^{\prime\prime}\big(\beta_{\Tilde{\xi}}(u_\eps(t,x))-u_\theta(t,y)\big)|\beta_{\Tilde{\xi}}^{\prime}(u_\eps(t,x))|^2 \notag \\&+ \beta_\xi^{\prime}\big(\beta_{\Tilde{\xi}}(u_\eps(t,x))-u_\theta(t,y)\big)\beta_{\Tilde{\xi}}^{\prime\prime}(u_\eps(t,x))\Big)\big|\phi(u_\eps(t,x))\big|^2\psi(t,y)\rho_m(y-x)\,dy\,dx\,dt\Big],\notag\\
        \underset{l \rightarrow \infty}{\lim}\,\underset{\kappa \rightarrow 0}{\lim}\,\underset{n \rightarrow \infty}{\lim} \mathcal{J}_6  = \,&\frac{1}{2}\mathbb{E}\Big[\int_{\Pi}\int_{D}\beta_\xi^{\prime\prime}(\beta_{\Tilde{\xi}}(u_\eps(t,x))-u_\theta(t,y))|\phi( u_\theta(t,y))|^2\psi(t,y)\rho_m(y-x)\,dx\,dy\,dt\Big]\notag. 
    \end{align}
\end{lem}
Regarding the stochastic terms due to Brownian noise, we have $\mathcal{J}_7 =0$  and $\mathcal{I}_7$ can be re-written as
\begin{align} 
    &\mathcal{I}_7 =  \mathbb{E}\Big[\int_{\Pi^2}\int_{\R}\beta_\xi^{\prime}\big(\beta_{\Tilde{\xi}}(u_\eps(t,x))-k\big)\beta_{\Tilde{\xi}}^\prime(u_\eps(t,x))\phi(u_\eps(t,x))\psi(s,y)\rho_n(t-s)\rho_m(y-x)\notag \\&\hspace{2cm}\times\Big(J_l\big(u_\theta^\kappa(s,y) - k\big)- J_l\big(u_\theta^\kappa(s-\frac{2}{n},y)-k)\big)\Big)\,dk\,dy\,ds\,dx\,dW(t)\Big]\,. \notag 
\end{align}
Let us define
\begin{align}
    \Hat{\mathcal{K}}[\beta^\prime, \rho_{m,n}](s,y,k) := \int_{\Pi}\beta_\xi^{\prime}\big(\beta_{\Tilde{\xi}}(u_\eps(t,x))-k\big)\beta_{\Tilde{\xi}}^\prime(u_\eps(t,x))\phi(u_\eps(t,x))\psi(s,y)\rho_n(t-s)\rho_m(y-x)\,dx\,dW(t)\,.\notag
\end{align}
Observe that, similar to \eqref{lem:for-ito-term}, one has
\begin{align}\label{lem:for-ito-2ndpart}
    &\underset{0 \le s \le T}{\sup}\mathbb{E}\Big[\|\Hat{\mathcal{K}}[\beta^{\prime\prime}, \rho_{m,n}](t, \cdot, \cdot)\|_{{L^\infty}(D \times R)}^2\Big] \le \frac{C(\psi)n^{\frac{2(p-1)}{p}}m^\frac{2}{p}}{\xi^a\Tilde{\xi}^b},\notag \\
    &\underset{0 \le s \le T}{\sup}\mathbb{E}\Big[\|\Hat{\mathcal{K}}[\beta^{\prime\prime\prime}, \rho_{m,n}](t, \cdot, \cdot)\|_{{L^\infty}(D \times R)}^2\Big] \le \frac{C(\psi)n^{\frac{2(p-1)}{p}}m^\frac{2}{p}}{\xi^{\hat{a}}\Tilde{\xi}^{\hat{b}}},
\end{align}
for some $a, b, \hat{a}, \hat{b} \ge 0$ depending upon $d$. 
Using with the  It\^o-L\'evy formula on $J_l\big(u_\theta^\kappa(s,y) - k\big)$ and integration by parts, we obtain
\begin{align}
   \mathcal{I}_7 =  &-\mathbb{E}\Big[ \int_{\Pi} \int_{\mathbb{R}} \Hat{\mathcal{K}}[\beta'', \rho_{m,n}](s,y,k)\int_{s-\frac{2}{n}}^{s}J_l(u_\theta^\kappa(r,y))-k)A_\theta(r,y)\,dr\,dk\,ds\,dy \Big]\notag \\
        &-\,\mathbb{E}\Big[ \int_{\Pi} \int_{\mathbb{R}}\int_D\int_{s-\frac{2}{n}}^s\beta_\xi^{\prime\prime}\big(\beta_{\Tilde{\xi}}(u_\eps(t,x))-k\big)\beta_{\Tilde{\xi}}^\prime(u_\eps(t,x))(\phi( u_\theta(r,y))\ast \rho_\kappa)\phi( u_\eps(t, x))\psi(r,y)\notag\\ &\hspace{4cm}\times\rho_m(y-x)\rho_n(t-r) J_l\big(u_\theta^\kappa(r,y)-k)\big)\,dt\,dx\,dk\,dr\,dy\Big]\notag \\
        &+ \frac{1}{2}\mathbb{E}\Big[\int_{\Pi}\int_{\R}\int_{D}\hat{\mathcal{K}}[\beta^{\prime\prime\prime}, \rho_{m,n}](s,y,k) \int_{s-\frac{2}{n}}^s J_l\big(u_\theta^\kappa(r,y)-k)\big)|\phi( u_\theta(r, y))\ast\rho_\kappa|^2\,dr\,dy\,dk\,ds\,dx\Big]\notag \\
        &+ \mathbb{E}\Big[\int_{\Pi}\int_{\R}\int_{D}\Hat{\mathcal{K}}[\beta^{\prime\prime\prime}, \rho_{m,n}](s,y,k)\int_{s-\frac{2}{n}}^{s}\int_{E}\int_0^1 (1-\lambda)J_l\big(u_\theta^\kappa(r,y) + \lambda\nu( u_\theta(r,y);z)\ast \rho_\kappa -k \big)\notag \\ & \hspace{3cm}\times\big|\nu( u_\theta(r,y);z) \ast \rho_\kappa\big|^2\,d\lambda\,m(dz)\,dr\,\,dy\,dk\,ds\,dx \Big]\notag 
        := \sum_{i=1}^4\mathcal{I}_{7,i} \,.
\end{align}
Similar to the analysis of $\mathcal{H}_6$ and using \eqref{lem:for-ito-2ndpart}, we get
\begin{align}
   & \mathcal{I}_{7,1},~ \mathcal{I}_{7,3},~\mathcal{I}_{7,4} \rightarrow 0 \quad \text{as} \quad n \rightarrow \infty\,,\notag \\
   & \underset{l \rightarrow \infty}{\lim}\,\underset{k \rightarrow 0}{\lim}\,\underset{n \rightarrow \infty}{\lim} \mathcal{I}_{7,2} = -\,\mathbb{E}\Big[ \int_{\Pi} \int_{\mathbb{R}}\int_D\beta_\xi^{\prime\prime}\big(\beta_{\Tilde{\xi}}(u_\eps(t,x))-u_\theta(t,y)\big)\beta_{\Tilde{\xi}}^\prime(u_\eps(t,x))\phi( u_\theta(t,y))\notag \\& \hspace{5cm}\times \phi( u_\eps(t, x))\psi(t,y)\rho_m(y-x) \,dt\,dx\,dy\Big]\,.\notag 
\end{align}
In view of Lemma \ref{lem:I6-I7} and above estimations, we get
\begin{align}
   & \underset{l \rightarrow \infty}{\lim}\,\underset{k \rightarrow 0}{\lim}\,\underset{n \rightarrow \infty}{\lim}\big(\mathcal{I}_6 + \mathcal{I}_7 + \mathcal{J}_6 + \mathcal{J}_7 \big)\notag \\ & \le C\Tilde{\xi} +  \frac{1}{2}\,\mathbb{E}\Big[\int_{\Pi}\int_{D}\beta_\xi^{\prime\prime}\big(\beta_{\Tilde{\xi}}(u_\eps(t,x))-u_\theta(t,y)\big)|\beta_{\Tilde{\xi}}^{\prime}(u_\eps(t,x))\phi(u_\eps(t,x)) -\phi(u_\theta(t,y))|^2\notag \\ & \hspace{6cm}\times\psi(t,y)\rho_m(y-x)\,dy\,dx\,dt\Big]\notag \\ &\underset{\Tilde{\xi} \rightarrow 0}{\longrightarrow}  \frac{1}{2}\,\mathbb{E}\Big[\int_{\Pi}\int_{D}\beta_\xi^{\prime\prime}\big(u_\eps^+(t,x))-u_\theta(t,y)\big)|\phi(u_\eps^+) -\phi(u_\theta)|^2\psi(t,y)\rho_m(y-x)\,dy\,dx\,dt\Big]
   \le C\xi\,. \notag
\end{align}
Thus, we have the following result.
\begin{lem}\label{lem:IJ67}It holds that,
    \begin{align} 
       \underset{\xi \rightarrow 0}{\lim}\,\underset{ \Tilde{\xi} \rightarrow 0}{\lim}\,\underset{l \rightarrow \infty}{\lim}\,\underset{k \rightarrow 0}{\lim}\,\underset{n \rightarrow \infty}{\lim}\big(\mathcal{I}_6 + \mathcal{I}_7 + \mathcal{J}_6 + \mathcal{J}_7 \big) \le 0. \notag 
    \end{align}
\end{lem}
We move on to concentrate on the additional term due to jump noise. Similar to Lemma \ref{lem:correction-term-due-levy}, one has
\begin{lem}\label{lem:i8j8}
    \begin{align}
        &\underset{\tilde{\xi} \rightarrow 0}{\lim}\,\underset{l \rightarrow \infty}{\lim}\,\underset{k \rightarrow 0}{\lim}\,\underset{n \rightarrow \infty}{\lim}\big(\mathcal{I}_8 + \mathcal{J}_8\big)\notag \\ & = \,\mathbb{E}\Big[\int_{\Pi}\int_{D}\int_{E}\Big(\beta_\xi\big(u_\eps^+(t,x) + \nu( u_\eps(t,x); z)-u_\theta(t,y)\big) - \beta_{\xi}\big(u_\eps^+(t,x)-u_\theta(t,y)\big)\notag\\  & \hspace{2cm}- \nu( u_\eps^+(t,x); z)\beta_\xi^\prime\big(u_\eps^+(t,x) -u_\theta(t,y)\big)\psi(t,y)\rho_m(y-x)\,m(dz)\,dy\,dx\,dt\Big]\notag \\
        & +\mathbb{E}\Big[\int_{\Pi}\int_{D}\int_{E}\Big(\beta_\xi\big(u_\eps^+(t,x)-u_\theta(t, y) - \nu(u_\theta(t, y);z)\big) - \beta_\xi(u_\eps^+(t,x)-u_\theta(t, y))\notag \\ & \hspace{2cm} +\nu(u_\theta(t, y);z)\beta^\prime(u_\eps^+(t,x)- u_\theta(t, y))\Big)\psi(t,y)\rho_m(y-x)\,m(dz)\,dx\,dy\,dt\Big]\notag \,.
    \end{align}
\end{lem}
Invoking similar lines of argument as done for $\mathcal{I}_7$ and $\mathcal{H}_8$, we have
\begin{align}\label{eq:I9-aftersendinglimit}
 \underset{\Tilde{\xi} \rightarrow 0}{\lim}\,\underset{l \rightarrow \infty}{\lim}\underset{\kappa \rightarrow 0}{\lim}\,\underset{n \rightarrow \infty}{\lim}\, \mathcal{I}_9 &\le \mathbb{E}\Big[\int_{\Pi}\int_{D}\int_{E}\Big(\beta_\xi\big(u_\eps^+(t,x) + \nu( u_\eps^+(t,x); z) -u_\theta(t,y) - \nu(u_\theta(t,y);z)\big) \notag \\ &- \beta_\xi\big(u_\eps^+(t,x) + \nu( u_\eps(t,x); z) -u_\theta(t,y)\big) + \beta_{\xi}\big(u_\eps^+(t,x) -u_\theta(t,y)\big)\notag \\ & - \beta_{\xi}\big(u_\eps^+(t,x) -u_\theta(t,y)-\nu(u_\theta(t,y);z)\big)\Big)\psi(t,y)\rho_m(y-x) \,m(dz)\,dx\,dy\,dt\Big]\,.
\end{align}
Note that $\mathcal{J}_8 = 0$. Thus, combining \eqref{eq:I9-aftersendinglimit} with Lemma \ref{lem:i8j8}, we get
\begin{align}
    &\underset{\Tilde{\xi} \rightarrow 0}{\lim}\,\underset{l \rightarrow \infty}{\lim}\,\underset{\kappa \rightarrow 0}{\lim}\,\underset{n \rightarrow \infty}{\lim}\, \big(\mathcal{I}_8 + \mathcal{I}_{9} + \mathcal{J}_8 + \mathcal{J}_9\big)\notag \\ 
 & \le  \mathbb{E}\Big[\int_{\Pi}\int_{D}\int_{E}(1-\lambda)p^2\beta_{\xi}^{\prime\prime}(q + \lambda p) \psi(t,y)\rho_m(y-x)\,m(dz)\,dx\,dy\,dt\Big]\notag 
\end{align}
where $p = \nu(u_\eps^+(t,x);z) - \nu(u_\theta(t,y);z)$ and $q=u_\eps^+(t,x) - u_\theta(t,y)$. By using similar lines of argument as done for Lemma \ref{lem:G9-H8}, we get the following result.
\begin{lem}\label{lem:i89j89}It holds that,
    \begin{align*}
         &\underset{\xi \rightarrow 0}{\lim}\,\underset{\Tilde{\xi} \rightarrow 0}{\lim}\,\underset{l \rightarrow \infty}{\lim}\,\underset{\kappa \rightarrow 0}{\lim}\,\underset{n \rightarrow \infty}{\lim}\, \big(\mathcal{I}_8 + \mathcal{I}_{9} + \mathcal{J}_8 + \mathcal{J}_9\big) \le 0.
    \end{align*}
\end{lem}
Combining Lemmas \ref{lem:I-1-J-1}-\ref{lem:i1011j1011}, \ref{lem:IJ67} and \ref{lem:i89j89} along with \eqref{eq:hatJ-1-2} and \eqref{eq:xi-0-hatJ-1} and passing to the limit with respect to $\eps$ and $\theta$ in the sense of Young measure in the resulting inequality, we arrive at the following:
\begin{align}
  0 \le &\,\mathbb{E}\Big[\int_{\Pi}\int_{D}\int_0^1\int_0^1\big({\tt v}^+(t,x, \gamma)-{\tt u}^+(t, y, \alpha)\big)^+\partial_t\psi(t,y)\rho_m(y-y)\,d\alpha\,d\gamma\,dx\,dt\Big]\notag \\
    -&\,\mathbb{E}\Big[\int_{\Pi}\int_{D}\int_0^1\int_0^1F^+({\tt v}^+(t,x, \gamma), {\tt u}^+(t, y, \alpha) )\nabla\psi(t,y)\rho_m(y-x)\,d\alpha\,d\gamma\,dx\,dt\Big] \notag \\
+&\,\mathbb{E}\Big[\int_{\Pi}\int_{D}\int_0^1\int_0^1\Phi^+({\tt v}^+(t,x, \gamma), {\tt u}^+(t, y, \alpha))\Delta\psi(t,y)\rho_m(y-x)\,d\alpha\,d\gamma\,dx\,dt\Big]  \notag \\
+& \int_{D^2}\big(u_0^+(x)-\hat{u}_0^+(y) \big)^+\rho_m(y-x)\psi(0,y)\,dx\,dy  + \widetilde{\mathcal{F}}_{{\tt u}}(\psi\bar{\rho}_m),\label{inq:eps-theta-goes-to-0}
\end{align}
where
\begin{align*}
    \widetilde{\mathcal{F}}_{{\tt u}}(\psi\bar{\rho}_m) : = &- \mathbb{E}\Big[\int_{\Pi}\int_{D}\int_0^1\text{sgn}^+({\tt u}^-(t,y,\alpha))\big \{{\tt u}(t,x,\alpha)\partial_t\psi(t,y)\rho_m(y-x)\notag \\ & - F({\tt u}(t,y, \alpha)) \nabla_y\big(\psi(t,y)\rho_m(y-x)\big) + \Phi({\tt u}(t,y, \alpha))\Delta_y\big(\psi(t,y)\rho_m(y-x)\big) \big \}\,d\alpha\,dx\,dt\,dy\Big] \notag \\ & -\int_{D}\text{sgn}^+((u_0(y))^-)u_0(y)\psi(0,y)\bar{\rho}_m(y)\,dy\,.
\end{align*}
Similar to  \eqref{inq:f-v-psi-non-negative}, we infer that $\widetilde{\mathcal{F}}_{\tt u}$ is a non-negative and linear operator on $\mathcal{D}^+([0,T] \times \R^d)$ and $\underset{m \rightarrow \infty}{\lim}\,\widetilde{\mathcal{F}}_{\tt u}(\psi\bar{\rho}_m)$ exists. Thus, upon passing to the limit as $m \rightarrow \infty$ in \eqref{inq:eps-theta-goes-to-0} gives
\begin{align}
0 \le &\,\mathbb{E}\Big[\int_{\Pi}\int_0^1\int_0^1\big({\tt v}^+(t,x, \gamma)-{\tt u}^+(t, x, \alpha)\big)^+\partial_t\psi(t,x)\,d\alpha\,d\gamma\,dx\,dt\Big]\notag \\
    -&\,\mathbb{E}\Big[\int_{\Pi}\int_0^1\int_0^1F^+({\tt v}^+(t,x, \gamma), {\tt u}^+(t, x, \alpha) )\nabla\psi(t,x)\,d\alpha\,d\gamma\,dx\,dt\Big] + \int_{D}\big(u_0^+(x)-\hat{u}_0^+(x) \big)^+\psi(0,x)\,dx \notag \\
+&\,\mathbb{E}\Big[\int_{\Pi}\int_0^1\int_0^1\Phi^+({\tt v}^+(t,x, \gamma), {\tt u}^+(t, x, \alpha))\Delta\psi(t,x)\,d\alpha\,d\gamma\,dx\,dt\Big] + \underset{m \rightarrow \infty}{\lim}\,\widetilde{\mathcal{F}}_{{\tt u}}(\psi\bar{\rho}_m). \label{inq:2nd-half-kato-inequality-01}
\end{align}
Since ${\tt v}$ (resp. ${\tt u}$) is the Young measure-valued limit of $\{u_\eps\}$ (resp. $\{u_\theta\}$), it is easily seen that $-{\tt v}$ (resp. $-{\tt u}$) is the Young measure-valued limit of $\{-u_\eps\}$ (resp.~$\{-u_\theta\}$), where $-u_\eps$ (resp. $-u_\theta$) is a weak solution to the problem \eqref{eq:-negative-u-eps} with initial condition $-u_0$ (resp. $-\hat{u}_0$). Thus, replacing ${\tt v}(t,x, \gamma)$ by $-{\tt v}(t,x, \gamma)$ and ${\tt u}(t,x, \alpha)$ by $-{\tt u}(t,x, \alpha)$ in \eqref{inq:2nd-half-kato-inequality-01} and using the fact that $a^- = (-a)^+$ for any $a \in \R$, we get
\begin{align}\label{inq:2nd half of kato inquality} 0 \le &\,\mathbb{E}\Big[\int_{\Pi}\int_0^1\int_0^1\big({\tt v}^-(t,x, \gamma)-{\tt u}^-(t, x, \alpha)\big)^+\partial_t\psi(t,x)\,d\alpha\,d\gamma\,dx\,dt\Big]\notag \\
+&\,\mathbb{E}\Big[\int_{\Pi}\int_0^1\int_0^1\text{sgn}^+\big({\tt v}^-(t,x, \gamma) -{\tt u}^-(t, x, \alpha) \big)\big(F(-{\tt v}^-(t,x, \gamma))-F(-{\tt u}^-(t, x, \alpha))\big)\nabla\psi(t,x)\,d\alpha\,d\gamma\,dx\,dt\Big] \notag \\  
-&\,\mathbb{E}\Big[\int_{\Pi}\int_0^1\int_0^1\text{sgn}^+\big({\tt v}^-(t,x, \gamma) -{\tt u}^-(t, x, \alpha)\big)\big(\Phi(-{\tt v}^-(t,x, \gamma)) - \Phi(-{\tt u}^-(t, x, \alpha))\big)\Delta\psi(t,x)\,d\alpha\,d\gamma\,dx\,dt\Big]\notag \\  + &  \int_{D}\big(u_0^-(x)-\hat{u}_0^-(x) \big)^+\psi(0,x)\,dx + \underset{m \rightarrow \infty}{\lim}\hat{\mathcal{F}}_{\tt u}(\psi\bar{\rho}_m)\,,
\end{align}
where
\begin{align}
   \hat{\mathcal{F}}_{\tt u}(\psi\bar{\rho}_m) &:= \,\mathbb{E}\Big[\int_{\Pi}\int_{D}\int_0^1\text{sgn}^+\big({\tt u}^+(t,y, \alpha)\big)\Big\{{\tt u}(t,y, \alpha)\partial_t\psi(t,y)\rho_m(y-x)\notag \\ &  - F({\tt u}(t,y, \alpha))\cdot\nabla_y(\psi(t,y)\rho_m(x-y)) +\Phi({\tt u}(t,y, \alpha))\Delta_y(\psi(t,y)\rho_m(y-x))\Big\}\,d\alpha\,dt\,dy\,dx\Big]\notag \\ &  + \int_{D}\text{sgn}^+(u_0^+(y))u_0(y)\psi(0,y)\bar{\rho}_m(y)\,dy\,.\notag
   \end{align}
\subsubsection{\textbf{Proof of Lemma} \ref{lem:Global-Kato}} 
Observe that,
\begin{align}
&-\text{sgn}^+\big({\tt u}^+-{\tt v}^+\big)\big(F({\tt u}^+) - F({\tt v}^+)\big) + \text{sgn}^+\big({\tt v}^--{\tt u^-}\big)\big(F(-{\tt v^-}) - F(-{\tt u}^-)\big)= -F^+({\tt u}, {\tt v})\,, \label{identity-for-F+} \\
&  \text{sgn}^+\big({\tt u}^+-{\tt v}^+\big)\big(\Phi({\tt u}^+) - \Phi({\tt v}^+)\big) - \text{sgn}^+\big({\tt v}^--{\tt u^-}\big)\big(\Phi(-{\tt v^-}) - \Phi(-{\tt u}^-)\big)= \Phi^+({\tt u}, {\tt v})\,. \label{identity-for-phi+}
\end{align}
Combining \eqref{inq:1st  step to 1st half of kato} and \eqref{inq:2nd half of kato inquality} and using \eqref{eq:idenity-1}, \eqref{identity-for-F+} and \eqref{identity-for-phi+} in the resulting inequality, we get
\begin{align} \label{inq:global-kato-1}
0 \le&\,\mathbb{E}\Big[\int_{\Pi}\int_0^1\int_0^1\big({\tt u}(t, x, \alpha)-{\tt v}(t,x, \gamma)\big)^+\partial_t\psi(t,x)\,d\alpha\,d\gamma\,dx\,dt\Big]\notag \\
    -&\,\mathbb{E}\Big[\int_{\Pi}\int_0^1\int_0^1F^+({\tt u}(t, x, \alpha), {\tt v}(t,x, \gamma))\nabla\psi(t,x)\,d\alpha\,d\gamma\,dx\,dt\Big] \notag \\  
+&\,\mathbb{E}\Big[\int_{\Pi}\int_0^1\int_0^1\Phi^+({\tt u}(t, x, \alpha), {\tt v}(t,x, \gamma))\Delta\psi(t,x)\,d\alpha\,d\gamma\,dx\,dt\Big] + \int_{D}\big(\hat{u}_0(x)- u_0(x)\big)^+\psi(0,x)\,dx \notag \\ + &  \underset{m \rightarrow \infty}{\lim}\,\mathcal{F}_{\tt v}(\psi\bar{\rho}_m) + \underset{m \rightarrow \infty}{\lim}\,\hat{\mathcal{F}}_{\tt u} (\psi\bar{\rho}_m) =: \Lambda({\tt u},{\tt v},\psi) +  \underset{m \rightarrow \infty}{\lim}\,\mathcal{F}_{\tt v}(\psi\bar{\rho}_m) + \underset{m \rightarrow \infty}{\lim}\,\hat{\mathcal{F}}_{\tt u} (\psi\bar{\rho}_m),
\end{align}
for any $\psi \in \mathcal{D}^+([0,T] \times \mathcal{B})$. Since for large $m^\prime$, $\psi(t,x)\bar{\rho}_{m^\prime}\in \mathcal{D}^+([0,T]\times D)$, we have, from the \textbf{local Kato's inequality} cf.~Remark \ref{rem:local-kato},
\begin{align}
    \Lambda({\tt u}, {\tt v}, \psi_{m^\prime}) \ge 0, \quad \text{for sufficiently large} \quad m^\prime,\label{inq:local-kato-1}
\end{align}
where $\psi_{m^\prime}(t,x) = \psi(t,x)\bar{\rho}_{m^\prime}(x)$. Since $\mathcal{F}_{{\tt v}}$ and $\hat{\mathcal{F}}_{{\tt u}}$ are linear operators and $\bar{\rho}_{m^\prime}\bar{\rho}_m = \bar{\rho}_{m^\prime}$ for large enough $m > m ^{\prime}$, we get 
\begin{align}
    &\underset{m^\prime, m \rightarrow \infty}{\lim}\,\mathcal{F}_{\tt u} (\psi(1-\bar{\rho}_{m^\prime})\bar{\rho}_m) + \underset{m^\prime, m \rightarrow \infty}{\lim}\,\hat{\mathcal{F}}_{\tt u} (\psi(1-\bar{\rho}_{m^\prime})\bar{\rho}_m) = 0. \label{limit-operator-final}
\end{align}
Since $\psi = \bar{\rho}_{m^\prime}\psi + (1- \bar{\rho}_{m^\prime})\psi=\psi_{m^\prime} + (1- \bar{\rho}_{m^\prime})\psi$, by using \eqref{inq:local-kato-1}, \eqref{inq:global-kato-1} and \eqref{limit-operator-final}, we arrive at the following Kato's inequality: for any  $\psi \in \mathcal{D}^+([0,T] \times \mathcal{B})$ 
\begin{align}
  0 \le \Lambda({\tt u}, {\tt v}, \psi)\,.\label{inq:global-kato-02}  
\end{align}
 Recall that $\mathcal{B}_0 \subset D$ and $\mathcal{B}_0 \cup \big(\cup_{i=1}^{\bar{k}}\mathcal{B}_i\big)$ is a covering for $D$. Let $\{\varphi_i\}_{0 \le i \le \bar{k}}$ be the partition of unity for the above covering such that $\varphi_i \in \mathcal{D}^+(\mathcal{B}_i)$, for $0 \le i \le \bar{k}$. Let $\psi \in \mathcal{D}^+([0,T] \times \R^d)$ and $\psi_i := \psi\varphi_i,$ for $0 \le i \le \bar{k}$. From \eqref{inq:global-kato-02} and Remark \ref{rem:local-kato}, one has
\begin{align}
     0 \le \Lambda({\tt u}, {\tt v}, \psi_i), \quad \text{for} \quad 0 \le i \le \bar{k}.\label{inq:global-kato-03}
\end{align}
Thanks to the linearity of $\psi \rightarrow \Lambda(\psi)$ and summing \eqref{inq:global-kato-03} over $i$, we get
$$0 \le \Lambda({\tt u}, {\tt v}, \psi)~~~\text{for any}~~~\psi \in \mathcal{D}^+([0,T] \times \R^d).$$ This essentially proves Lemma \ref{lem:Global-Kato}.
\subsection{Uniqueness of measure-valued limit processes}\label{sec:Uniqueness of measure-valued solution}
When $\hat{u}_0 = u_0$, we follow the similar lines of argument as invoked in \cite{Bauzet-2015, Majee-2019} to conclude the following~(from the global Kato's inequality): for a.e. $t > 0$
\begin{align}
\mathbb{E}\Big[\int_{D}\int_0^1\int_0^1\big({\tt u}(t, x, \alpha)-{\tt v}(t,x, \gamma)\big)^+\,d\alpha\,d\gamma\,dx\Big] = 0.\notag
\end{align}
This implies that for a.e. $(t,x)\in \Pi$ and $\mathbb{P}$-a.s.,  ${\tt u}(t, x, \alpha)={\tt v}(t,x, \gamma)$. 
Moreover, it also implies that the unique measure-valued limit process is independent of the additional variables $\alpha$ and $\gamma$, say $u(t,x)$. Thus, we infer that the viscous approximations $\{u_\eps(t,x)\}$ converges weakly to $u(t,x)$ in $L^2(\Omega \times \Pi)$. Since the limit process is independent of the additional variable, one can easily conclude that 
$$ u_\eps\goto u~~~\text{in}~~~L_{\text{loc}}^p(\Omega \times \Pi)~~\text{ for any }~~1 \le p <2.$$ 
\subsection{Well-posedness of entropy solution}
The existence of an entropy solution to the problem \eqref{eq:degenerate-SPDE} is based upon the classical viscosity method, the proof of which can be directly adapted from \cite{Bauzet-2015} and \cite{Majee-2019} for bounded domain. Using the a-priori bounds \eqref{inq:apriori-bound} and strong convergence of viscous solution in $L_{\text{loc}}^p(\Omega \times \Pi)$ for any $1 \le p <2$, we follow \cite[Section 4.2]{Majee-2019} and \cite{Bauzet-2015} to conclude that viscous limit $u(t,x)$ is indeed
an entropy solution to problem \eqref{eq:degenerate-SPDE} in the sense of Definition \ref{def-1}.  For its uniqueness, 
we compare any entropy solution to the viscous solution using Kur\v{z}kov's doubling the variable technique and then pass to the limit as viscosity goes to zero (cf.~ Subsection \ref{sec:Uniqueness of measure-valued solution}). In particular, we conclude that any entropy solution is equal to the limit of the viscous approximations---which then imply uniqueness of entropy solutions of \eqref{eq:degenerate-SPDE} . As a consequence, we also have the following contraction principle.
\begin{prop}Let $u_1$ and $u_2$ be the entropy solution to \eqref{eq:degenerate-SPDE} corresponding to initial data $u_{0_1}\in L^2(D)$ and $u_{0_2} \in L^2(D)$, respectively. Then for any $t\in (0,T)$,
    \begin{align}
        \mathbb{E}\Big[\int_{D}(u_1(t) - u_2(t))^+\,dx\Big] \le \int_{D}(u_{0_1} - u_{0_2})^+\,dx. \notag
    \end{align}
\end{prop}

\vspace{0.5cm}

\noindent{\bf Acknowledgement:} 
The first author would like to acknowledge the financial support of CSIR, India.

\vspace{.2cm}

\noindent{\bf Data availability:}  Data sharing does not apply to this article as no data sets were generated or analyzed during the current study.

\vspace{0.2cm}

\noindent{\bf Conflict of interest:}  The authors have not disclosed any competing interests.


\vspace{0.5cm}
\begin{thebibliography}{}

\bibitem{Sbihi-2015}~B.~Andreianov and K.~Sbihi. \newblock
Well-posedness of general boundary-value problems for scalar conservation laws. \newblock{\em Trans. Amer. Math. Soc.} Vol. 367, No. 6 (June 2015), pp. 3763-3806.

\bibitem{Ammar-2006}
~K.~Ammar, P.~Willbold, and J.~Carrillo. \newblock Scalar conservation laws with general boundary condition and
continuous flux function. \newblock{\em J. Differ. Equ.} 228(1), 111–139 (2006).

\bibitem{Maliki-2010}
B. Andreianov and M. Maliki.
\newblock A note on uniqueness of entropy solutions to degenerate parabolic equations in $\R^N$.
\newblock{\em NoDEA
Nonlinear Differential Equations Appl.} 17 (2010), no. 1, 109-118.
 \bibitem{Balder}
E.~J.~Balder.
\newblock Lectures on Young measure theory and its applications in economics.
\newblock{\em Rend. Is it. Mat.Univ. Trieste}, 31 Suppl. 1:1-69, 2000.


\bibitem{Bauzet-2012}
C.~ Bauzet, G.~ Vallet, and P.~ Wittbold.
\newblock The Cauchy problem for a conservation law with a multiplicative stochastic perturbation. 
\newblock{\em Journal of Hyperbolic Differential Equations.} 9 (2012), no 4, 661-709.

\bibitem{Bauzet-2014}
C.~ Bauzet, G.~ Vallet, and P.~ Wittbold.
\newblock The Dirichlet problem for a conservation law with a multiplicative stochastic perturbation. \newblock{\em Journal of Functional Analysis.}
Volume 266, Issue 4, 15 (2014), Pages 2503-2545.

\bibitem{Bauzet-2015}
 C. Bauzet, G. Vallet and P. Wittbold.
 \newblock A degenerate parabolic-hyperbolic Cauchy problem with a stochastic force.
\newblock{\em J. Hyperbolic Differ. Equ.} 12 (2015), no. 3, 501-533.
\bibitem{SRB-2023}~S.~R. Behera and A.~ K. Majee.
Renormalized stochastic entropy solution for degenerate 
parabolic-hyperbolic equation with noise. (submitted).

\bibitem{Bendahmane-2005}
M. Bendahmane and K. H. Karlsen.
\newblock Uniqueness of entropy solutions for doubly nonlinear anisotropic
degenerate parabolic equations.
\newblock{\em Contemp. Math.} 371 (2005), 1-28. 

\bibitem{Majee-2015}
I. H. Biswas, K. H. Karlsen, and A.~ K. Majee.
\newblock Conservation laws driven by  L\'{e}vy  white noise.
\newblock{\em J. Hyperbolic Differ. Equ.} 12
(2015), no. 3, 581-654.

\bibitem{Majee-2014}
I. H. Biswas and A.~K. Majee.
\newblock Stochastic conservation laws: weak-in-time formulation and strong entropy condition.
\newblock{\em Journal of Functional Analysis} 267 (2014) 2199-2252.
\bibitem{Majee-2019}
I. H. Biswas, A. K. Majee, and G. Vallet.
\newblock On the Cauchy problem of a degenerate parabolic-hyperbolic PDE with L\'{e}vy noise.
\newblock{\em Adv. Nonlinear Anal.} 8 (2019), no. 1, 809-844.
\bibitem{Blouza-SIAM-2001}
A. Blouza and E. L. Dretan. \newblock Up-to-boundary version of Friedrichs's lemma and applications to the linear Koiter shell model.
\newblock{\em SIAM Journal on Mathematical Analysis} 33(4) (2001):877-895.
\bibitem{Burger-1999}
M. C. Bustos, F. Concha, R. B\"{u}rger and E. M. Tory.
\newblock Sedimentation and Thickening: Phenomenological
Foundation and Mathematical Theory. 
\newblock{ \em Mathematical Modelling: Theory and Applications}, vol. 8.
Kluwer Academic Publishers, Dordrecht (1999).


\bibitem{Carrillo-1999}
J. Carrillo.
\newblock Entropy solutions for nonlinear degenerate problems.
\newblock {\em Arch. Ration. Mech. Anal.} 147 (1999), no. 4, 269-361.

\bibitem{Chen:2012fk}
G.~Q.~ Chen, Q.~ Ding, and K. ~H. Karlsen.
\newblock On nonlinear stochastic balance laws.
\newblock {\em Arch. Ration. Mech. Anal.}, 204(3) (2012), 707-743.


\bibitem{Chen-2005}
G. Q. Chen and K. H. Karlsen.
\newblock Quasilinear anisotropic degenerate parabolic equations with time-space dependent diffusion coefficients.
\newblock {\em Commun. Pure Appl. Anal.} 4 (2005), No. 2, 241-266. 

\bibitem{Chen-2003}
G.-Q. Chen and B. Perthame.
\newblock Well-posedness for non-isotropic degenerate parabolic-hyperbolic equations.
\newblock{\em Ann. Inst. H.
Poincaré Anal.} Non Lin\'{e}aire 20 (2003), no. 4, 645-668.
.




\bibitem{dafermos}
C. ~M. Dafermos.
\newblock {\em Hyperbolic conservation laws in continuum physics}, volume 325
  of {\em Grundlehren der Mathematischen Wissenschaften [Fundamental Principles
  of Mathematical Sciences]}.
\newblock Springer-Verlag, Berlin, 2000.

\bibitem{Hofmanova-2016}
 A. Debussche, M. Hofmanov\'{a} and J. Vovelle.
 \newblock Degenerate parabolic stochastic partial differential equations: quasilinear case.
 \newblock{\em Ann. Probab.} 44 (2016), no. 3, 1916–1955.

 \bibitem{Vovelle2010}
A.~ Debussche and J. Vovelle. 
\newblock Scalar conservation laws with stochastic forcing. 
\newblock{\em J. Funct. Analysis}, 259 (2010), 1014-1042. 
\bibitem{Vovelle-2018}
S.  Dotti, and J. Vovelle.
\newblock Convergence of approximations to stochastic scalar conservation laws.
\newblock{\em Arch. Ration. Mech. Anal.} 230 (2018), no. 2, 539-591.

\bibitem{Karlsen-2000-LN}
M.S. Espedal and K. H. Karlsen. 
\newblock Numerical solution of reservoir flow models based on large time step operator splitting algorithms. In: Filtration in Porous Media and Industrial Application (Cetraro, 1998).
\newblock{\em Lecture Notes in Math}, vol. 1734, pp. 9-77. Springer, Berlin (2000).

\bibitem{Kobayasi-2016}
K.~Kobayasi and D.~Noboriguchi. \newblock A stochastic conservation law with nonhomogeneous Dirichlet boundary
conditions. \newblock{\em Acta Math. Vietnam} 41(4), 607–632 (2016).

\bibitem{nualart:2008}
J.~ Feng and D.~ Nualart.
\newblock Stochastic scalar conservation laws.
\newblock {\em J. Funct. Anal.}, 255 (2008), no. 2, 313-373.

\bibitem{Frid-2017}~
H.~Frid and Y.~Li. \newblock A boundary value problem for a class of anisotropic degenerate parabolic-hyperbolic
equations. \newblock{\em Arch. Ration. Mech. Anal.} 226(3), 975–1008 (2017).
\bibitem{Frid-2022}~H.~Frid, Y.~Li, D.~Marroquin, J.~F.~C.~Nariyoshi and  Z.~Zeng. \newblock The Dirichlet problem for stochastic
degenerate parabolic-hyperbolic equations. \newblock{\em Commun. Math. Anal. Appl.} 1 (1) (2022) 1–71. 

\bibitem{Frid-2022a}~H.~Frid, Y.~Li, D.~Marroquin, J.~F.~C.~Nariyoshi and  Z.~Zeng. \newblock A boundary value problem for a class of anisotropic stochastic degenerate parabolic-hyperbolic
equations.\newblock{\em Journal of Functional Analysis}
Volume 285, Issue 9, 1 November 2023.




\bibitem{Girault-1999} V.~Girault and L. R.~ Scott.
\newblock Analysis of a two-dimensional grade-two fluid model with a tangential boundary condition.
\newblock{\em Journal de Mathématiques Pures et Appliquées}.
Volume 78, Issue 10, 10 December 1999, Pages 981-1011.
\bibitem{Grisvard-1985}
P.~ Grisvard. \newblock Elliptic problems in non-smooth domains. \newblock{\em Monographs and Studies in Mathematics}
24, Pitman, Boston, London, 1985.
\bibitem{Karlsen-2017}
K. H.  Karlsen and E. B.  Storr\o{}sten.
\newblock On stochastic conservation laws and Malliavin calculus.
\newblock{\em J. Funct. Anal.} 272 (2017), no. 2, 421-497.

\bibitem{Kim}
J.~U.~ Kim.
\newblock On a stochastic scalar conservation law, 
\newblock{Indiana Univ. Math. J.}  52 (1) (2003) 227-256. 
\bibitem{kruzkov}
S.~ N.~ Kruzkov.
\newblock First-order quasilinear equations with several independent variables.
\newblock {\em Mat. Sb. (N.S.)}, 81(123): 228-255, 1970.

\bibitem{Lions-1994}
P.~L.~Lions, E.~Perthame and E.~Tadmor. \newblock A kinetic formulation of multidimensional scalar conservation
laws and related equations. \newblock{\em J. AMS } 7, 169–191 (1994).

\bibitem{Mascia-2002}~C.~Mascia, A.~Porretta and A.~ Terracina. \newblock Non-homogeneous Dirichlet problems for
degenerate parabolic-hyperbolic equations. \newblock{\em Arch. Rational Mech. Anal.} 163 (2002) 87–124.
\bibitem{Otto-1996}~F.~Otto. \newblock Initial-boundary value problem for a scalar conservation law.\newblock{\em C. R. Acad. Sci. Paris Sér. I Math.}
322(8), 729-734 (1996).

\bibitem{Porretta-2003}~A.~Porretta and J.~ Vovelle. \newblock $L^1$ solutions to first-order hyperbolic equations in bounded domains. \newblock{\em Commun.
Partial Differ. Equ.} 28(1–2), 381–408 (2003).

\bibitem{Vallet-2000} G.~Vallet. \newblock Dirichlet problem for a nonlinear conservation law. \newblock{\em Revista Matem\'atica
Complutense.} XIII (1) (2000) pp 231-250. 

\bibitem{Vallet-2005}
G. Vallet.
\newblock Dirichlet problem for a degenerated hyperbolic-parabolic equation.
\newblock{\em Adv. Math. Sci. Appl.} 15 (2005), no. 2, 423-450.
\bibitem{Vallet-2008} G.~ Vallet. \newblock Stochastic perturbation of nonlinear degenerate parabolic problems, Differ.
\newblock{\em Integral Equ.} 21(11–12) (2008) 1055–1082.

\bibitem{Vallet-2009}
G. Vallet and P. Wittbold. \newblock On a stochastic first-order hyperbolic equation in a bounded domain. \newblock{\em Infin.
Dimens. Anal. Quantum Probab. Relat. Top.} 12(4), 613-651 (2009).



\bibitem{Volpert}
A.~I. Vol'pert.
\newblock Generalized solutions of degenerate second-order quasilinear parabolic and elliptic equations.
\newblock {\em Adv. Differential Equations}, 5(10-12):1493--1518, 2000.

\bibitem{Hudjaev-1969}
 A. I. Vol’pert, S. I.  Hudjaev.
 \newblock The Cauchy problem for second order quasilinear degenerate parabolic equations.
 \newblock{\em Mat. Sb. (N.S.)} 78(120) (1969), 374-396.
\end{thebibliography}
\end{document}